\documentclass[10pt]{article}

\usepackage{authblk}
 \usepackage{IEEEtrantools}
\usepackage{multirow}

\usepackage{array,epsfig}
\usepackage{amsmath}
\usepackage{amsfonts}
\usepackage{amssymb}
\usepackage{amsxtra}
\usepackage{amsthm}
\usepackage{mathrsfs}
\usepackage{color}
\usepackage[dvipsnames]{xcolor}
\usepackage{tikz-cd}
\usepackage{mathabx,epsfig}
\usepackage{comment}
\usepackage{enumerate}
\usepackage{mathtools}
\usepackage{ytableau}
\usepackage{hyperref}
\usepackage{stmaryrd}
\usepackage{extarrows}

\hypersetup{
    colorlinks=true,
    linkcolor=blue,
    filecolor=magenta,      
    urlcolor=cyan,
    citecolor=teal
}
\theoremstyle{theorem}

\newtheorem{ex}{Example}

\newtheorem{theorem}{Theorem}[section]
\newtheorem{corollary}[theorem]{Corollary}
\newtheorem{lemma}[theorem]{Lemma}
\newtheorem{prop}[theorem]{Proposition}

\newtheorem{Note}[theorem]{Notation}

\theoremstyle{definition}
\newtheorem{definition}[theorem]{Definition}

\newtheorem{remark}[theorem]{Remark}
\newtheorem{notation}[theorem]{Notation}

\newcommand{\ra}{\rightarrow}

\newcommand{\bb}[1]{\mathbb{#1}}
\newcommand{\mc}[1]{\mathcal{#1}}
\newcommand{\lin}[1]{\langle #1 \rangle}
\newcommand{\Z}{\bb{Z}}
\newcommand{\Q}{\bb{Q}}
\newcommand{\R}{\bb{R}}
\newcommand{\F}{\bb{F}}
\newcommand{\N}{\bb{N}}

\newcommand{\del}{\partial}

\newcommand{\sgn}{\mathrm{sgn}}
\newcommand{\rad}{\mathrm{rad}}
\newcommand{\Link}{\mathrm{Link}}
\newcommand{\sd}{\mathrm{sd}}
\newcommand{\w}{\omega}
\newcommand{\coker}{\mathrm{coker}}
\newcommand{\im}{\mathrm{im}}
\newcommand{\Ind}{\mathrm{Ind}}

\newcommand{\llb}{\llbracket}
\newcommand{\rrb}{\rrbracket}
\newcommand{\Sp}{\mathrm{Sp}}
\newcommand{\SL}{\mathrm{SL}}
\newcommand{\GL}{\mathrm{GL}}
\newcommand{\St}{\mathrm{St}}
\newcommand{\Sh}{\mathrm{Sh}}

\newcommand{\tE}{\mathrm{E}}
\newcommand{\tH}{\mathrm{H}}

\newcommand{\tC}{\mathrm{C}}
\newcommand{\tD}{\mathrm{D}}

\title{A Projective Resolution of the Symplectic Steinberg Module}
\author[Pal]{Urshita Pal\footnote{urshita@umich.edu; University of Michigan, Department of Mathematics, Ann Arbor MI 48109, USA}}

\begin{document}
%--------------------------------------------------------------------
\maketitle
%--------------------------------------------------------------------

%%----------------------------------------------------
\begin{comment}

Dan Margalit's feedback on intro

before publishing: ensure bibliography is up-to-date

citations:
- cite own p=5 generalization
- cite Ash-Gunnells-McConnell - email Gunnells, ask about the stuff Zach emailed about

- cite Miller-Scalamandre-Sroka
- cite own Hopf algebra equivalence proof
    
\end{comment}

%%--------------------------------------------------------------

\begin{abstract}
    Borel--Serre proved that for a number ring $R$ with fraction field $K$, the symplectic group $\Sp_{2n}(R)$ is a virtual duality group of degree quadratic in $n$, and that the symplectic Steinberg module $\St^\w_{2n}(K)$ is its dualizing module. We construct a projective resolution of this symplectic Steinberg module as an $\Sp_{2n}(R)$-representation, that is similar in form to a resolution of Lee--Szczarba for the special linear group, but whose construction is more involved. When $R$ is a Euclidean number ring, we use this resolution to compute the top degree cohomology of principal level-$p$ congruence subgroups of $\Sp_{2n}(R)$, for primes $p \in R$ such that the natural map $R^\times \to (R/(p))^\times$ is surjective.
\end{abstract}

\tableofcontents

%%--------------------------------------------------------------- 
\section{Introduction}
%%--------------------------------------------------------------- 

The cohomology of arithmetic groups plays a fundamental role in algebraic K-theory and number theory. Examples of arithmetic groups include special linear and symplectic groups over a number ring and their finite-index subgroups.
An important class of finite-index subgroups are \emph{principal congruence subgroups}. These are the kernels $\Gamma_n(J)$ and $\Gamma^\w_{2n}(J)$ of the maps $\SL_n(R) \to \SL_n(R/J)$ and $\Sp_{2n}(R) \to \Sp_{2n}(R/J)$, respectively, that reduce coefficients mod $J$, where $J \subset R$ is an ideal such that $R/J$ is finite. When $J = (p)$ is a principal ideal, we will denote $\Gamma_n(J)$ and $\Gamma^\w_{2n}(J)$ simply as $\Gamma_n(p)$ and $\Gamma^\w_{2n}(p)$, respectively. In this paper we construct a resolution that enables us to compute certain cohomology groups of finite-index subgroups of $\Sp_{2n}(R)$ for $R$ a number ring. This resolution is similar in form to one due to Lee--Szczarba in \cite[Theorem 3.1]{lee1976homology}, but whose construction is more involved and requires new ideas. We also use this resolution to compute the top degree cohomology of some principal congruence subgroups $\Gamma^\w_{2n}(p) \subset \Sp_{2n}(R)$ when $R$ is Euclidean.

\paragraph{Borel-Serre duality.}
Borel--Serre \cite{borel1973corners} proved that when $R$ is a number ring and $\Gamma$ is a finite-index subgroup of $\SL_n(R)$ or $\Sp_{2n}(R)$, then $\Gamma$ is a rational duality group of some dimension $\nu$, which implies that there are natural isomorphisms of $\Q$-vector spaces
\begin{align*}
    \tH^{\nu-i}(\Gamma; \Q) \cong \tH_i(\Gamma; \mathfrak{D}\otimes \Q) \text{ for all } i
\end{align*}
for a $\Gamma$-module $\mathfrak{D}$ called the \emph{dualizing module}. This statement also holds with $\Z$-coefficients if $\Gamma$ is torsion-free.
In particular, we have
\begin{align*}
%\label{eqn:bsduality}
    \tH^\nu(\Gamma; \Q) \cong \tH_0(\Gamma; \mathfrak{D} \otimes \Q) \cong (\mathfrak{D}\otimes\Q)_{\Gamma},
\end{align*}
where the subscript $\Gamma$ indicates that we are taking $\Gamma$-coinvariants.

Suppose $R$ has fraction field $K$. For the special linear group $\SL_n(R)$, the degree $\nu$ is given by
\begin{align*}
    \nu = r \binom{n+1}{2} + cn^2 - n - r - c +1
\end{align*}
and for $\Sp_{2n}(R)$, $\nu$ is given by
\begin{align*}
    \nu = r(n^2+n) + 2cn^2 -n
\end{align*}

where
\begin{itemize}
    \item $r$ is the number of embeddings $K \hookrightarrow \R$
\item $c$ is the number of pairs of complex conjugate embeddings $K \hookrightarrow \bb{C}$ that do not factor through $\R$.
\end{itemize}

\paragraph{Steinberg modules.} In the case of $\SL_n(R)$ and $\Sp_{2n}(R)$ for $R$ a number ring, the dualizing module $\mathfrak{D}$ has a beautiful combinatorial description.
For $K$ a field, let $\bb{T}_n(K)$ be the \emph{Tits building} for $\SL_n(K)$, that is, the simplicial complex whose $k$-simplices are flags of vector subspaces of $K^n$
\begin{align*}
    0 \subsetneq V_0 \subsetneq \dots \subsetneq V_k \subsetneq K^n
\end{align*}
The Solomon--Tits theorem \cite{solomon1969steinberg, brown1989buildings} says that $\bb{T}_n(K)$ is homotopy equivalent to a wedge of spheres of dimension $(n-2)$. The \emph{Steinberg module} for $\SL_n(K)$, denoted $\St_n(K)$, is $\widetilde{\tH}_{n-2}(\bb{T}_n(K))$.
The action of $\SL_n(K)$ (and its subgroups) on $\bb{T}_n(K)$ descends to an action on $\St_n(K)$. When $V$ is a finite-dimensional vector space over $K$, we will use $\bb{T}(V)$ and $\St(V)$, respectively, to denote the Tits building and the Steinberg module associated to $V$. Thus $\bb{T}(K^n) \coloneq \bb{T}_n(K)$ and $\St(K^n) \coloneq \St_n(K)$.

When $K$ is the field of fractions of a number ring $R$, Borel--Serre \cite{borel1973corners} proved that the dualizing module for a finite-index subgroup $\Gamma \subset \SL_n(R)$ is $\St_n(K)$.

We have similar notions for symplectic groups. For $K$ a field, let $\bb{T}^\w_{2n}(K)$ be the \emph{symplectic Tits building} for $\Sp_{2n}(K)$, that is, the simplicial complex whose $k$-simplices are flags of isotropic subspaces of $K^{2n}$ (i.e. subspaces where the standard symplectic form $\w$ on $K^{2n}$ restricts to 0 -- see Definition \ref{defn:sympl})
\begin{align*}
    0 \subsetneq V_0 \subsetneq \dots \subsetneq V_k \subset K^{2n}
\end{align*}
The Solomon--Tits theorem \cite{solomon1969steinberg, brown1989buildings} says that $\bb{T}^\w_{2n}(K)$ is homotopy equivalent to a wedge of spheres of dimension $(n-1)$. The \emph{symplectic Steinberg module} for $\Sp_{2n}(K)$, denoted $\St^\w_{2n}(K)$, is $\widetilde{\tH}_{n-1}(\bb{T}^\w_{2n}(K))$. The action of $\Sp_{2n}(K)$ and its subgroups on $\bb{T}^\w_{2n}(K)$ descends to an action on $\St^\w_{2n}(K)$. When $V$ is a finite-dimensional vector space over $K$ with a non-degenerate alternating bilinear form $\w$, we will use $\bb{T}^\w(V)$ and $\St^\w(V)$, respectively, to denote the symplectic Tits building and symplectic Steinberg module associated to $V$. Thus $\bb{T}^\w(K^{2n}) \coloneq \bb{T}^\w_{2n}(K)$ and $\St^\w(K^{2n}) \coloneq \St^\w_{2n}(K)$.

When $K$ is the field of fractions of a number ring $R$, Borel--Serre \cite{borel1973corners} proved that the dualizing module for a finite-index subgroup $\Gamma \subset \Sp_{2n}(R)$ is $\St^\w_{2n}(K)$.

%%----------------------------------------------------------
\paragraph{The Sharbly resolution.}
One way to compute the twisted homology groups $\tH_{i}(\SL_n(R); \St_n(K) \otimes \Q)$ and $\tH_{i}(\Sp_{2n}(R); \St_{2n}^\w(K) \otimes \Q)$ is by constructing projective resolutions of $\St_n(K)\otimes \Q$ and $\St_{2n}^\w(K)\otimes \Q$ over $\SL_n(R)$ and $\Sp_{2n}(R)$, respectively, and computing homology after taking the coinvariants of these resolutions under the respective group actions. In practice, for computational feasibility, one needs the coinvariants of the terms of the resolution to not be too large.

In \cite[Theorem 3.1]{lee1976homology}, Lee--Szczarba constructed an $\SL_n(R)$-resolution of $\St_n(K)$, for $K$ the field of fractions of a PID $R$. 

\begin{theorem}[Lee--Szczarba, {\cite[Theorem 3.1]{lee1976homology}}]\label{prop:sharbly:intro}
    Let ${R}$ be a PID, and $K$ its field of fractions. Suppose $V$ is a vector field over $K$ of dimension $n$.
    Then we have a resolution of $\Z[\SL_n({R})]$-modules, that we will refer to as the \emph{Sharbly resolution of $\St(V)$}
    \begin{align*}
       \dots \to \Sh_k(V) \to \dots \to \Sh_1(V) \to \Sh_0(V) \to \St(V) \to 0
    \end{align*}
    where the degree $k$ term $\Sh_k(V)$ is the quotient of the free abelian group on ordered spanning sets of lines of $V$ of size $n+k$, quotiented by the relation that
    \begin{align*}
        [\lin{\vec{v}_1}, \dots, \lin{\vec{v}_{n+k}}] = (-1)^{\sgn (\sigma)}[\lin{\vec{v}_{\sigma(1)}}, \dots, \lin{\vec{v}_{\sigma({n+k})}}]
    \end{align*}
    for any permutation $\sigma$.
    Furthermore, if ${R}$ has finitely many units, then tensoring the above resolution with $\otimes_\Z \Q$ in fact yields a resolution of projective $\Q[\SL_n({R})]$-modules.
    See Proposition \ref{prop:sharbly} for a description of the differential on the resolution.
\end{theorem}

Lee--Szczarba used this resolution to prove the following.

\begin{theorem}[Lee--Szczarba, {\cite[Theorems 1.2 and 1.3]{lee1976homology}}]\label{prop:lscoinvariants}

    \begin{enumerate}
        \item Let $R$ be a Euclidean domain with multiplicative norm, and $K$ its field of fractions. Then we have
        \begin{align*}
           \tH_0(\SL_n(R); \St_n(K)\otimes \Q) = (\St_n(K) \otimes \Q)_{\SL_n(R)} \cong 0
        \end{align*}

        \item For $n \geq 3$, we have an isomorphism
        \begin{align*}
    \tH_0(\Gamma_n(3); \St_n(\Q)) =(\St_n(\Q))_{\Gamma_n(3)} \cong \St_n(\F_3)
\end{align*}
    \end{enumerate}
\end{theorem}

The proof for the second statement crucially uses the fact that the map $\Z^\times \to \F_3^\times$ is surjective. Indeed, the above statement is known to be false for primes $p >3$ (See for instance, Miller--Patzt--Putman \cite[Theorem A]{miller2021top}).

%%------------------------------------------------------------
%Lee--Szczarba's proof applies more generally \flag{CITE SELF}.
%%------------------------------------------------------------

%%----------------------------------------------------------

In higher degrees of homology, the groups in Lee--Szczarba's resolution seem too large to be practical for computations (see \cite[Section 4]{lee1976homology}). However, recent work by Ash--Miller--Patzt \cite[Theorem A]{ash2024hopf} used their resolution to prove the existence of a Hopf algebra structure on the rational cohomology of $\GL_n(\Z)$. Independently, Brown--Chan--Galatius--Payne \cite[Theorem 1.7]{brown2024hopf} established a Hopf algebra structure on these groups using completely different methods.

%%-------------------------------------
%%\flag{ASH ET AL USING RESOLUTIONS}

%% CITE own Hopf algebra equiv when it's done
%%-------------------------------------

\paragraph{A Sharbly analogue for the symplectic group.} In this paper we construct a resolution of the symplectic Steinberg module that is analogous in form to Lee--Szczarba's construction from Theorem \ref{prop:sharbly:intro}.

%%----------------------------------------------------------
\begin{theorem}
    \label{thm:introsharblyres}
    Let ${R}$ be a PID with field of fractions $K$. There is a resolution of $\St^\w_{2n}(K)$ of $\Z[\Sp_{2n}(R)]$-modules, where the degree $d$ term of the resolution is given by
        %%----------------------------------------------------
    \begin{align*}
        \bigoplus \Sh_{k_1}(W_1) \otimes \dots \otimes \Sh_{k_m}(W_m),
    \end{align*} 
%%----------------------------------------------------
    where the direct sum ranges over (ordered) direct sum decompositions 
%%----------------------------------------------------
    \begin{align*}
        V = W_1 \oplus \dots \oplus W_m
    \end{align*} 
%%----------------------------------------------------
    of $V$ into perpendicular symplectic subspaces, $\Sh_{k_i}(W_i)$ is the $k_i$-th term in the Sharbly resolution for $W_i$, and 
%%----------------------------------------------------
    \begin{align*}
        d = k_1 + k_2 + \dots + k_m + n -m\\
    \end{align*}
    
    The resolution has the property that if $R$ has finitely many units, then tensoring with $\otimes_\Z \Q$ gives a resolution of projective $\Q[\Sp_{2n}(R)]$-modules. (For a description of the differential on the resolution, see Theorem \ref{thm:sharblyres}). 
\end{theorem}
%%----------------------------------------------------------

As a corollary, we obtain a presentation for the symplectic Steinberg module. A similar presentation was constructed by Tóth \cite[Theorem 2]{toth2005steinberg}. Tóth in fact constructs a presentation of Steinberg modules over many more Chevalley groups and uses it to deduce a cyclicity result for these Steinberg modules, thus generalizing results of Ash--Rudolph \cite{ash1979modular} and Gunnells \cite{gunnells2000symplectic}.

\begin{corollary}\label{cor:Stwpresent}
    Let ${R}$ be a PID with field of fractions $K$. There is a presentation of $\St^\w_{2n}(K)$ over $\Z[\Sp_{2n}(R)]$: 
\begin{align*}
 V_{1,2} \oplus V_{1,1} \xrightarrow{\del_1} V_0 \to \St^\w(V)
\end{align*}
where $V_0, V_{1,2}, V_{1,2}$ are as follows:
\begin{itemize}
    \item $V_0$ is the direct sum of all terms of $\Sh_0(W_1)\otimes\dots\otimes\Sh_0(W_n)$, where $W_1 \oplus \dots \oplus W_n = V$ is a decomposition of $V$ into mutually perpendicular genus 1 symplectic subspaces.
    \item $V_{1,1}$ is the direct sum of terms of the form $\Sh_0(W_1)\otimes \dots \otimes\Sh_1(W_i)\otimes \dots \otimes \Sh_0(W_n)$, where $W_1 \oplus \dots \oplus W_n = V$ is a decomposition of $V$ into mutually perpendicular genus 1 symplectic subspaces.
    \item $V_{1,2}$ is the direct sum over all terms of the form $\Sh_0(W_1)\otimes\dots \otimes \Sh_0(W_{n-1})$, where $V = W_1 \oplus \dots \oplus W_{n-1}$ is a decomposition into mutually perpendicular subspaces, where all but one of the $W_i$ have genus 1 and the remaining one has genus 2.
\end{itemize}
The map $\del_1$ is as follows:
\begin{itemize}
    \item $[\lin{\vec{v}_1}, \lin{\vec{v}_{\bar{1}}}]\otimes\dots\otimes[\lin{\vec{u}_i}, \lin{\vec{v}_i}, \lin{\vec{w}_i}]\otimes\dots\otimes[\lin{\vec{v}_n}, \lin{\vec{v}_{\bar{n}}}] \in V_{1,1}$ maps to the sum
    \begin{align*}
        \left([\lin{\vec{v}_1}, \lin{\vec{v}_{\bar{1}}}]\otimes\dots\otimes[\lin{\vec{v}_i}, \lin{\vec{w}_i}]\otimes\dots\otimes[\lin{\vec{v}_n}, \lin{\vec{v}_{\bar{n}}}]\right) - \left([\lin{\vec{v}_1}, \lin{\vec{v}_{\bar{1}}}]\otimes\dots\otimes[\lin{\vec{u}_i}, \lin{\vec{w}_i}]\otimes\dots\otimes[\lin{\vec{v}_n}, \lin{\vec{v}_{\bar{n}}}]\right)\\ + \left([\lin{\vec{v}_1}, \lin{\vec{v}_{\bar{1}}}]\otimes\dots[\lin{\vec{u}_i}, \lin{\vec{v}_i}]\otimes\dots\otimes[\lin{\vec{v}_n}, \lin{\vec{v}_{\bar{n}}}]\right).
    \end{align*}
    
    \item $[\lin{\vec{v}_1}, \lin{\vec{v}_{\bar{1}}}]\dots[\lin{\vec{v}_{i,0}}, \lin{\vec{v}_{i,1}}, \lin{\vec{v}_{i,2}}, \lin{\vec{v}_{i,3}}]\dots[\lin{\vec{v}_{n-1}}, \lin{\vec{v}_{\overline{n-1}}}] \in V_{1,2}$ maps to a sum of 'split terms' as defined in Proposition \ref{differential}:
    \begin{align*}
        (-1)^\nu[\lin{\vec{v}_1}, \lin{\vec{v}_{\bar{1}}}]\dots[\lin{\vec{w}_i}, \lin{\vec{x}_i}][\lin{\vec{y}_i}, \lin{\vec{z}_i}]\dots[\lin{\vec{v}_{n-1}}, \lin{\vec{v}_{\overline{n-1}}}] 
    \end{align*}
    where $\vec{w}_i, \vec{x}_i$ are a subset of the vectors $\vec{v}_{i,0}, \vec{v}_{i,1}, \vec{v}_{i,2}, \vec{v}_{i,3}$ spanning a genus 1 subspace of $W_i$, and $\vec{y}_i$ and $\vec{z}_i$ are the projections of the other two vectors onto its perpendicular subspace. The sign $\nu$ is given by the formula 
    \begin{align*}
        \nu = j_1+j_2 + n-1-i
    \end{align*}
    where $j_1, j_2$ are the indices of the $\vec{v}_{i,t}$ that correspond to $\vec{w}_i$ and $\vec{x}_i$.
\end{itemize}

If $R$ has finitely many units, then tensoring with $\otimes_\Z \Q$ gives a presentation of projective $\Q[\Sp_{2n}(R)]$-modules.

\end{corollary}

We remark here that though this presentation is quite general, it comes at the cost of its groups being fairly large in many cases, and it is very hard to construct small presentations in general. In work of Brück--Patzt--Sroka \cite{bruck2023presentation}, they give a presentation for the symplectic Steinberg module $\St^\w_{2n}(\Q)$ over $\Sp_{2n}(\Z)$ where it is much easier to calculate coinvariants (\cite[Theorem B]{bruck2023presentation}). They used this presentation to show the vanishing of the rational cohomology groups $\tH^{n^2-1}(\Sp_{2n}(\Z); \Q)$ for $n \geq 2$ (\cite[Theorem A]{bruck2023presentation}). Their ideas were inspired by Church--Putman's proof in \cite{church2017codimension} of the Bykovskii presentation for $\St_n(\Q)$, where they relate the construction of such presentations to the high connectivity of suitably defined simplicial complexes. These complexes tend to get highly intricate and technical, and it is not clear whether it is feasible to carry out these constructions for rings other than $\Z$. The presentation from Corollary \ref{cor:Stwpresent} can be thought of as a weaker analogue of Brück--Patzt--Sroka's presentation in \cite[Theorem B]{bruck2023presentation}, but for symplectic Steinberg modules over arbitrary fields.

%%-----------------------------------------------------------
When $R$ is a Euclidean domain and $K$ its field of fractions, we use the presentation from Corollary \ref{cor:Stwpresent} to compute the coinvariants of $\St^\w_{2n}(K)$ by the level $p$ principal congruence subgroup of $\Sp_{2n}(R)$ for certain primes $p \in R$. Our result is as follows.

%%------------------------------------------------
\begin{theorem}
    \label{thm:congsubgp}
    Let $R$ be a Euclidean domain and $K$ its field of fractions. Let $\F = R/(p)$ for a prime $p \in R$. Suppose the map $R^\times \to \F^\times$ is surjective. Then for all $n \geq 1$, we have an isomorphism
    \begin{align*}
        (\St_{2n}^\w(K))_{\Gamma^\w_{2n}(p)} \cong \St^\w_{2n}(\F)
    \end{align*}
\end{theorem}
%%------------------------------------------------

Combined with Borel-Serre duality \cite{borel1973corners}, we get the following corollary.

\begin{corollary}
    Let $R$ be a Euclidean number ring with fraction field $K$, and $p \in R$ a prime. Let $\F$ be the quotient field $R/(p)$. If the natural map $R^\times \to \F^\times$ is surjective, then we have an isomorphism
    \begin{align*}
        \tH^\nu(\Gamma^\w_{2n}(p); \Q) \cong \St^\w_{2n}(\F) \otimes \Q
    \end{align*}
    where $\nu$ is given by:
\begin{align*}
    \nu = r(n^2+n) + 2cn^2 -n
\end{align*}
where
\begin{itemize}
    \item $r$ is the number of embeddings $K \hookrightarrow \R$
\item $c$ is the number of pairs of complex conjugate embeddings $K \hookrightarrow \bb{C}$ that do not factor through $\R$.
\end{itemize}

    If $\Gamma^\w_{2n}(p)$ is torsion-free, this result holds integrally, i.e. 
    \begin{align*}
        \tH^\nu(\Gamma^\w_{2n}(p); \Z) \cong \St^\w_{2n}(\F)
    \end{align*}
\end{corollary}

There are many examples of pairs of rings and primes $(R,p)$ that satisfy the hypotheses of Theorem \ref{thm:congsubgp}, including: $(\Z,3), (\Z[\iota], 1+2\iota), (\Z[\sqrt{-2}], 1+\sqrt{-2})$, $(\Z[\sqrt{2}], 3+\sqrt{2})$, $(\Z[\w], 2\w+1)$, $(\Z[\w], 3\w+1)$, and so on.
A simple way to generate such examples is to take a Euclidean domain $R$ that has multiplicative norm, and take an element $p \in R$ with minimal norm such that $p$ is not a unit. Any such $p$ will necessarily be prime, and will satisfy the conditions of Theorem \ref{thm:congsubgp}. Such a prime $p$ is called a \emph{universal side divisor} of $R$.

We expect that it is also possible to use the presentation from Corollary \ref{cor:Stwpresent} to show vanishing of the coinvariants $(\St^\w_{2n}(\Q))_{\Sp_{2n}(\Z)}$ in a similar way to Lee--Szczarba's result in \cite[Theorem 1.3]{lee1976homology}, but we have not pursued this in the present article. This result is already known, for example, by work of Gunnells \cite{gunnells2000symplectic}, with a new proof via studying group actions on simplicial complexes by Brück--Sroka \cite{bruck2025apartment}.

\paragraph{Proof outline.}
An obstruction to constructing an analogue of Lee--Szczarba's resolution from Theorem \ref{prop:sharbly:intro} is that there is no clear way to generalize their proof to the symplectic group. We overcome this obstruction by studying a genus filtration on a simplicial complex closely related to the Tits buildings. The associated spectral sequence yields a finite resolution of $\St^\w_{2n}(\F)$ involving (symplectic) Steinberg modules of vector spaces of strictly smaller dimension, which may be a result of independent interest. The combinatorial arguments used to analyze the associated spectral sequence are inspired by methods used in Section 10 of \cite{bruck2023presentation}.

\begin{prop}
\label{prop:introsteinbergres}
    Let $K$ be a field and let $V$ be a symplectic $K$-vector space of genus $n$. We have a resolution of $\Z[\Sp_{2n}(K)]$-modules
    \begin{equation*}
       0 \ra \St(V)  \longrightarrow  \bigoplus_{g(W) = n-1} \St(W) \otimes \St^{\w}(W^{\perp})  \longrightarrow  \dots  \longrightarrow  \bigoplus_{g(W) = 1} \St(W) \otimes \St^{\w}(W^{\perp})  \longrightarrow  \St^{\w}(V) \to 0 
    \end{equation*}  
    where the direct sums are taken over symplectic subspaces of $V$, and $g(W)$ denotes the genus of a symplectic subspace $W \subset V$. 
\end{prop}

In work in preparation, Miller--Scalamandre--Sroka proved that the Koszul dual of an equivariant ring built out of symplectic Steinberg modules (see \cite[Theorem 1.5]{miller2025rank}) are the even Steinberg modules for the special linear group. From this perspective, the resolution constructed in Proposition \ref{prop:introsteinbergres} can be viewed as a Koszul resolution. Likely, one can use Proposition \ref{prop:introsteinbergres} to give an alternative derivation of Miller--Scalamandre--Sroka’s computation of the Koszul dual of symplectic Steinberg modules. 

%%----------
%INSERT REFERENCE
%%----------

Proposition \ref{prop:introsteinbergres} is proved in Section \ref{sec:steinbergres}, and a description of the differential on it is given in Proposition \ref{Eq1Diff}.

The above resolution then allows us to inductively assemble projective resolutions of its individual terms to get a projective resolution of the symplectic Steinberg module. The process of assembling resolutions is described in Section \ref{glueprocess}, and uses a fairly standard homological algebra argument. 

The next obstruction then is to get an explicit description of the differential on this resolution. This is a fairly involved process -- owing to the large number of indices involved in the description of the groups of the resolution and the number of maps that arise from the assembly process described in Section \ref{glueprocess} -- and takes up most of Section \ref{sec:sharblypf}. See Theorem \ref{thm:sharblyres} for a description of the differential.

%Theorem \ref{thm:introsharblyres} is proved in section \ref{sec:sharblypf}. 

Theorem \ref{thm:congsubgp} is proved in Section \ref{sec:topdegcoh}. We outline the proof as follows: Suppose $\F$ is the quotient of a Euclidean domain $R$ by a prime $p \in R$ and let $K$ be the field of fractions of $R$. There is a map $\bb{T}^\w_{2n}(K) \to \bb{T}^\w_{2n}(\F)$ that intersects an isotropic subspace of $K^{2n}$ with $R^{2n}$, and then reduces it mod $p$. Passing to homology, and observing that the map is $\Gamma^\w_{2n}(p)$-invariant (with $\Gamma^\w_{2n}(p)$ acting trivially on the codomain), we get a map
\begin{align*}
    (\St^\w_{2n}(K))_{\Gamma^\w_{2n}(p)} \to \St^\w_{2n}(\F)
\end{align*}
It is not hard to see that this map is surjective. We use the presentation of $\St^\w_{2n}(\F)$ given by our resolution to show that this map is in fact an isomorphism when all units of $\F$ can be lifted to units of $R$. This property is crucial in our proof, and in fact this map is known to not be injective when for example, $R=\Z$ and $p>3$, by recent work of Capovilla-Searle \cite[Theorem A]{capovilla2026top}.  

%%------------------------------
%%INSERT REFERENCE
%%------------------------------

Our proof requires a cyclicity result for $\St^\w_{2n}(K)$ as an $\Sp_{2n}(R)$-module when $R$ is a Euclidean domain, which was proven by Gunnells \cite[Theorem 4.11]{gunnells2000symplectic}, with a new proof using high-connectivity of certain simplicial complexes later given for the case $R = \Z$ by Brück--Sroka \cite{bruck2025apartment}.

%%--------------------------------------------------------------------
\paragraph{Acknowledgements.} I would like to thank my advisor Jenny Wilson, for the unending support, encouragment, for lending so many hours to listening to me talk about math, and for close readings of this paper. I would also like to thank Jeremy Miller and Peter Patzt for suggesting the topic of this paper, sharing key ideas used in this project, and for feedback on drafts of this paper. This project greatly benefitted from many discussions with them. I thank Matthew Scalamandre and Robin Sroka for helpful discussions, and Gahl Shemy for feedback on a draft of this paper. I am grateful for the support of the University of Michigan's Rackham one-term dissertation fellowship and the Rackham predoctoral fellowship. I am grateful for travel support from the NSF grant DMS-2142709.
%%--------------------------------------------------------------------

%%--------------------------------------------------------------------
\section{A finite resolution of (symplectic) Steinberg modules}
%%--------------------------------------------------------------------
In this section we start by giving relevant background and stating our main theorem \ref{thm:introsharblyres} in more detail (Theorem \ref{thm:sharblyres}). Then we describe how to use a genus filtration on the Tits building to obtain the finite resolution stated in Proposition \ref{prop:introsteinbergres}, which will be a key input in our proof of Theorem \ref{thm:sharblyres}.
%%--------------------------------------------------------------------
\subsection{Background}\label{sec:background}
%%--------------------------------------------------------------------
We first collect relevant foundations, set up notation, and then state our constructed resolution of the symplectic Steinberg module in Theorem \ref{thm:sharblyres}.
%%----------------------------------------------------
\begin{definition}\label{defn:sympl}
    \begin{itemize}
    \item  \textbf{Symplectic Vector Spaces.}
    For a field $K$, a $K$-vector space $V$ is \emph{symplectic} if it is equipped with a symplectic bilinear form $\w$, i.e. $\w$ is bilinear, alternating, and non-degenerate.
    A finite dimensional symplectic vector space must necessarily have even dimension.
    These conditions on $\w$ are equivalent to the condition that there is a basis $\{\vec{e}_1, \vec{e}_{\bar{1}}, \vec{e}_2, \vec{e}_{\bar{2}}, \dots, \vec{e}_n, \vec{e}_{\bar{n}}\}$ for the bilinear form $\w$ such that for $1 \leq i,j \leq n$,
    \begin{align*}
        & \w(\vec{e}_i,\vec{e}_j) = 0 = \w(\vec{e}_{\bar{i}}, \vec{e}_{\bar{j}}), & \w(\vec{e}_i, \vec{e}_{\bar{j}}) = 0 = \w(\vec{e}_{\bar{j}}, \vec{e}_i) \text{ for }i\neq j,  & \hspace{1cm}  \w(\vec{e}_i, \vec{e}_{\bar{i}}) = -\w(\vec{e}_{\bar{i}}, \vec{e}_i) = 1 
    \end{align*}
    
    \item \textbf{Symplectic Bases.}
    For a symplectic $K$-vector space $V$, we call a set $\{\vec{v}_1,\vec{v}_{\bar{1}},\dots , \vec{v}_n, \vec{v}_{\bar{n}}\} \subset V$ a \emph{symplectic basis} if there is a symplectic (i.e., form-preserving) automorphism $V \to V$ that sends $\vec{e}_i \mapsto \vec{v}_i$ and $\vec{e}_{\bar{i}} \mapsto \vec{v}_{\bar{i}}$ for all $1 \leq i \leq n$.
    We use $\Sp(V)$ to denote the group of all such symplectic automorphisms of $V$.

        \item \textbf{Radical of a Subspace.} Given a symplectic $K$-vector space $(V, \w)$ and a subspace $W \subset V$, the \emph{radical of $W$} is 
        \begin{align*}
            \rad (W) := \{ \vec{v} \in W \text{ such that } \w(\vec{v},\vec{w}) = 0 \hspace{0.25cm} \forall \vec{w} \in W \}
        \end{align*}

        \item \textbf{Genus of a Subspace.} Given a symplectic vector space $V$ of dimension $2n$, we say $V$ has \emph{genus} $n$.
        In general, for a subspace $W \subset V$, the \emph{genus of $W$} is half the rank of the matrix of the restriction of the form $\w|_W$.
         
        Later we will need to index certain sums by symplectic subspaces of $V$ (i.e. subspaces $W$ where the restriction of the form $\w|_W$ to $W$ is non-degenerate). For this reason, we shall use $g(W)$ to denote the genus of a symplectic vector subspace $W \subset V$.

        \item \textbf{Isotropic Subspaces.} For a symplectic $K$-vector space $V$, a subspace $J \subset V$ is \emph{isotropic} if the restriction of the form $\w|_J$ is identically zero.
    \end{itemize}
\end{definition}
%%----------------------------------------------------

%%----------------------------------------------------
\begin{remark}
    Let $V$ be a genus $n$ symplectic vector space over a field $K$. Suppose $W$ is a symplectic subspace of $V$ of genus $k$. Then $W^{\perp}$ is a symplectic subspace of genus $n-k$, and $W \oplus W^{\perp} = V$.
\end{remark}
%%----------------------------------------------------

We shall also sometimes need to work with symplectic spaces over PIDs, as opposed to fields. For this we make the following definition.
%%----------------------------------------------------
\begin{definition}[Symplectic Space over a PID]
    Suppose ${R}$ is a PID. An ${R}$-module $M$ is said to be \emph{symplectic} if $M \cong {R}^{2n}$ for some $n \in \N$, and $M$ is equipped with the standard symplectic form. That is, $M$ has a basis
    \begin{align*}
        \{\vec{e}_1, \vec{e}_{\bar{1}}, \vec{e}_2, \vec{e}_{\bar{2}}, \dots, \vec{e}_n, \vec{e}_{\bar{n}}\}
    \end{align*}
    and the bilinear form $\w$ is such that for $1 \leq i,j \leq n$,
    \begin{align*}
        & \w(\vec{e}_i,\vec{e}_j) = 0 = \w(\vec{e}_{\bar{i}}, \vec{e}_{\bar{j}}), & \w(\vec{e}_i, \vec{e}_{\bar{j}}) = 0 = \w(\vec{e}_{\bar{j}}, \vec{e}_i) \text{ for }i\neq j,  & \hspace{1cm}  \w(\vec{e}_i, \vec{e}_{\bar{i}}) = -\w(\vec{e}_{\bar{i}}, \vec{e}_i) = 1 
    \end{align*}
We denote the group of form-preserving automorphisms of $M$ over ${R}$ as $\Sp(M)$.
\end{definition}
%%----------------------------------------------------

    Suppose $K$ is the field of fractions of ${R}$. Note that $M \cong {R}^{2n}$ is canonically embedded in $V \coloneq M \otimes_{{R}} K \cong K^{2n}$, and the form $\w$ on $M$ is the restriction of a corresponding symplectic form on $V$. 
    %Note also that any automorphism in $\Sp_{{R}}(M)$ extends uniquely to a form-preserving automorphism of $V$. We thus make the following definition.
%%----------------------------------------------------
   % \begin{definition}[Integral Symplectic Automorphisms]
   %     Let ${R}$ be a PID with field of fractions $\F$. Suppose $M$ is a symplectic ${R}$-module of genus $n$, and let $V = \F \otimes_{{R}}M$ (thus $V$ is a symplectic $\F$-vector space).
   %     We define $\Sp_{{R}}(V)$ \flag{do you need to distinguish b/w $\Sp_R$ and $\Sp_\F$? - yes, in Prop \ref{sharblyterms}.} to be the subgroup of $\Sp_\F(V)$ consisting of symplectic automorphisms of $V$ that are the extension of a symplectic automorphism of $M$ to $V$. 
   % \end{definition}
%%----------------------------------------------------

%%----------------------------------------------------
\begin{definition}[Tits building]
    Let $V$ be a finite-dimensional vector space over a field $K$. 
    \begin{itemize}
        \item Let $\mc{T}(V)$ denote the poset of proper nonzero subspaces of $V$, ordered by inclusion. Let $\bb{T}(V)$ denote the geometric realization of $\mc{T}(V)$, viewed as a simplicial complex. This is called the \emph{Tits building} for $V$.
        \item Suppose $V$ is symplectic. Let $\mc{T}^\w(V)$ denote the poset of nonzero isotropic subspaces of $V$, ordered by inclusion. Let $\bb{T}^\w(V)$ denote the geometric realization of $\mc{T}^\w(V)$, viewed as a simplicial complex. This is called the \emph{symplectic Tits building} for $V$.
    \end{itemize}
\end{definition}
%%----------------------------------------------------

%%----------------------------------------------------
\begin{definition}[Apartment Classes]
\label{Apt}
    Let $V$ be a finite-dimensional vector space over a field $K$. 
    \begin{itemize}
        \item Let $\vec{v}_1, \vec{v}_2, \dots, \vec{v}_{n}$ be an ordered basis for $V$. The \emph{apartment} corresponding to this basis is a $(n-2)$-dimensional cycle (i.e. element of the simplicial chain group) of $\bb{T}(V)$ given by 
\begin{align*}
    \sum_{\sigma \in \Sigma_{n}} (-1)^{\sgn(\sigma)} [\lin{\vec{v}_{\sigma(1)}}, \lin{\vec{v}_{\sigma(1)}, \vec{v}_{\sigma(2)}}, \dots, \lin{\vec{v}_{\sigma(1)}, \dots, \vec{v}_{\sigma(n-1)}}]
\end{align*}
where $\Sigma_n$ denotes the Weyl group of permutations of the set $\{ 1, 2, \dots, n\}$.
The reduced homology class in $\widetilde{\tH}_{n-2}(\bb{T}(V))$ determined by this cycle is called the \emph{apartment class} corresponding to this basis.

\item Suppose $V$ is symplectic of genus $n$. Let $\vec{v}_1, \vec{v}_{\bar{1}}, \dots, \vec{v}_n, \vec{v}_{\bar{n}}$ be an ordered symplectic basis for $V$. The \emph{symplectic apartment} corresponding to this basis is an $(n-1)$-dimensional cycle of $\bb{T}^\w(V)$ given by 
\begin{align*}
    \sum_{\sigma \in \mc{B}_{n}} (-1)^{\sgn(\sigma)} [\lin{\vec{v}_{\sigma(1)}}, \lin{\vec{v}_{\sigma(1)}, \vec{v}_{\sigma(2)}}, \dots, \lin{\vec{v}_{\sigma(1)}, \dots, \vec{v}_{\sigma(n)}}]
\end{align*}
where $\mc{B}_n$ denotes the Weyl group of signed permutations of the set $\{ 1, \bar{1}, \dots, n, \bar{n}\}$.
The reduced homology class in $\widetilde{\tH}_{n-1}(\bb{T}^\w(V))$ determined by this cycle is called the \emph{symplectic apartment class} corresponding to this basis.
    \end{itemize}
\end{definition}
%%----------------------------------------------------

%%----------------------------------------------------
\begin{theorem}[Solomon--Tits, \cite{solomon1969steinberg, brown1989buildings}] 
\label{SolTit}
Suppose $V$ is a finite dimensional $K$-vector space. 
\begin{enumerate}
    \item Suppose $\dim V = n$. The Tits building $\bb{T}(V)$ is homotopy equivalent to a wedge of spheres of dimension $n-2$. The \emph{Steinberg module}, denoted $\St(V)$, is defined as the reduced homology group $\widetilde{\tH}_{n-2}(\bb{T}(V))$.
    The Steinberg module is generated by apartment classes.
    
    \item Suppose $V$ is symplectic of genus $n$ (so that $\dim V = 2n$). The symplectic Tits building $\bb{T}^\w(V)$ is homotopy equivalent to a wedge of spheres of dimension $n-1$. The \emph{Symplectic Steinberg module}, denoted $\St^\w(V)$, is defined as the reduced homology group $\widetilde{\tH}_{n-1}(\bb{T}^\w(V))$. The symplectic Steinberg module is generated by symplectic apartment classes.
\end{enumerate}
\end{theorem}
%%----------------------------------------------------

We now state a slight variant of Lee--Szczarba's construction of a $\SL_n(R)$-resolution of the Steinberg module.
%%----------------------------------------------------
\begin{prop}[Lee--Szczarba, {\cite[Theorem 3.1]{lee1976homology}}]\label{prop:sharbly}
    Let ${R}$ be a PID, and $K$ its field of fractions. Suppose $V$ is a vector field over $K$ of dimension $n$.
    Then we have a resolution of $\Z[\SL_n({R})]$-modules, that we will refer to as the \emph{Sharbly resolution of $\St(V)$}
    \begin{align*}
       \dots \to \Sh_k(V) \to \dots \to \Sh_1(V) \to \Sh_0(V) \to \St(V) \to 0
    \end{align*}
    where:
    \begin{itemize}
        \item The degree $k$ term $\Sh_k(V)$ is the quotient of the free abelian group on ordered spanning sets of lines of $V$ of size $n+k$, quotiented by the relation that
    \begin{align*}
        [\lin{\vec{v}_1}, \dots, \lin{\vec{v}_{n+k}}] = (-1)^{\sgn (\sigma)}[\lin{\vec{v}_{\sigma(1)}}, \dots, \lin{\vec{v}_{\sigma({n+k})}}]
    \end{align*}
    for any permutation $\sigma$.
    \item For $k \geq 1$, the differential $\Sh_k(V) \to \Sh_{k-1}(V)$ maps a set of vectors $[\lin{\vec{v}_1}, \dots, \lin{\vec{v}_{n+k}}]$ to the usual sum where we omit one vector at a time -- if omitting a vector results in the span becoming strictly smaller than $V$, the corresponding term is taken to be 0 instead.
    \item The map $\Sh_0(V) \to \St(V)$ maps a set of vectors $[\lin{\vec{v}_1}, \dots, \lin{\vec{v}_{n}}]$ to the apartment class determined by the basis $\vec{v}_1, \dots, \vec{v}_n$.
    \end{itemize}
    Furthermore, if ${R}$ has finitely many units, then tensoring the above resolution with $\otimes_\Z \Q$ in fact yields a resolution of projective $\Q[\SL_n({R})]$-modules.
\end{prop}
%%----------------------------------------------------

Lee--Szczarba's original version dealt with ordered sets of unimodular vectors of $R^n$ as opposed to lines of $V$. But the version that uses lines can be shown via the same proof. The projectivity statement can be shown by using, for example, \cite[Lemma 3.2]{church2017codimension}.

%%----------------------------------------------------
\begin{notation}
    We will denote the tensor products $\Sh_k(V)\otimes_\Z \Q$, $\St(V) \otimes_\Z \Q$, and $\St^\w(V)\otimes_\Z \Q$ using $\Sh_k(V)_\Q$, $\St(V)_\Q$, and $\St^\w(V)_\Q$, respectively.
\end{notation}
%%----------------------------------------------------

We can now state our constructed resolution of the symplectic Steinberg module. To avoid clutter, we will write a Sharbly $[\lin{v_1}, \dots, \lin{v_k}]$ simply as $[v_1, \dots, v_k]$.
%%--------------------------------------------------------------
\begin{theorem}
    \label{thm:sharblyres}
    Let ${R}$ be a PID with field of fractions $K$. Suppose $V$ is a symplectic $K$-vector space of genus $n$. Then we have a resolution of $\St^\w(V)$ of $\Z[\Sp_{2n}(R)]$-modules such that:
%%---------------------------------------------------------------    
    \begin{enumerate}
        \item The degree-$d$ term in the resolution is
        %%----------------------------------------------------
    \begin{align*}
        \bigoplus \Sh_{k_1}(W_1) \otimes_\Z \dots \otimes_\Z \Sh_{k_m}(W_m),
    \end{align*} 
%%----------------------------------------------------
    where the direct sum ranges over (ordered) direct sum decompositions 
%%----------------------------------------------------
    \begin{align*}
        V = W_1 \oplus \dots \oplus W_m
    \end{align*} 
%%----------------------------------------------------
    of $V$ into perpendicular symplectic subspaces, $\Sh_{k_i}(W_i)$ is the $k_i$-th term in the Sharbly resolution for $W_i$, and 
%%----------------------------------------------------
    \begin{align*}
        d = k_1 + k_2 + \dots + k_m + n -m\\
    \end{align*}

    Here $d$ and $n$ are fixed, $m, k_1, \dots, k_m$ vary over $\Z_{\geq 0}$, and $1 \leq m \leq n$.
%%--------------------------------------------------------------
\item Suppose that
%%----------------------------------------------------
\begin{align*}
    [{v}^1_{0}, {v}^1_1, \dots, {v}^1_{d_1}]\otimes \dots \otimes [{v}^m_0, {v}^m_1, \dots, {v}^m_{d_m}] \in \Sh_{k_1}(W_1) \otimes_\Z \dots \otimes_\Z \Sh_{k_m}(W_m)
\end{align*}
%%----------------------------------------------------
is a simple tensor in one of the terms of the resolution. Then the image of the differential on this term is the sum of the following two types of terms:
%%--------------------------------------------------------
\begin{itemize}
    \item \emph{Omit-terms:} Suppose ${v}^j_i$ is such that the span of ${v}^j_0, \dots, \widehat{{v}^j_i}, \dots, {v}^j_{d_j}$ is the same as the span of $v^j_0, \dots, v^j_{d_j}$. Then the \emph{omit-term} corresponding to $v^j_i$ is
%%----------------------------------------------------
    \begin{align*}
        (-1)^{d_1 -1 + \dots + d_{j-1}-1}(-1)^i [v^1_{0}, v^1_1, \dots, v^1_{d_1}]\otimes \dots \otimes [v^j_0, \dots, \widehat{v^j_i}, \dots, v^j_{d_j}] \otimes \dots \otimes [v^m_0, v^m_1, \dots, v^m_{d_m}]
    \end{align*}
%%----------------------------------------------------
    \item \emph{Split-terms:} Suppose $v^j_{i_0}, v^j_{i_1}, \dots, v^j_{i_p}$ are vectors such that $v^j_0, \dots, \widehat{v^j_{i_0}}, \dots, \widehat{v^j_{i_p}}, \dots, v^j_{d_j}$ spans a proper nonzero symplectic subspace $X_j$ of $W_j$. 
    Let 
%%----------------------------------------------------
    \begin{align*}
        \lin{w^j_0}, \dots, \lin{w^j_p}
    \end{align*}
%%----------------------------------------------------
    be the projections of
%%----------------------------------------------------
    \begin{align*}
        \lin{v^j_{i_0}}, \dots, \lin{v^j_{i_p}}
    \end{align*}
%%----------------------------------------------------
    respectively onto $X_j^\perp \subset W_j$.
    Then the \emph{split-term} corresponding to these vectors is
%%----------------------------------------------------
    \begin{IEEEeqnarray*}{l}
       \hspace*{-3em} (-1)^\nu [v^1_{0}, \dots, v^1_{d_1}]\otimes \dots \otimes [v^j_0, \dots, \widehat{v^j_{i_0}}, v^j_{i_0+1}, \dots, \widehat{v^j_{i_1}}, \dots, v^j_{i_p-1}, \widehat{v^j_{i_p}}, \dots, v^j_{d_j}] \otimes [w^j_0, \dots w^j_p] \otimes \dots\otimes [v^m_0, \dots, v^m_{d_m}]
    \end{IEEEeqnarray*}
%%----------------------------------------------------
    with
%%----------------------------------------------------
    \begin{align*}
        \nu = (d_1-1 + d_2-1 + \dots + d_{j-1}-1) + pd_j + \sgn(\sigma) + (d_{j+1} + d_{j+2} + \dots + d_m)
    \end{align*}
%%----------------------------------------------------
    where $\sigma$ is the shuffle permutation that shuffles $v^j_{i_0}, \dots, v^j_{i_p}$ to the first $p+1$ positions in $[v^j_0, \dots, v^j_{d_j}]$. (Here we are using the convention that $\sgn(\sigma) = 0$ if $\sigma$ is an even permutation and $\sgn(\sigma) = 1$ if $\sigma$ is odd).
\end{itemize}
%%----------------------------------------------------
    \end{enumerate}
%%--------------------------------------------------
Furthermore, if $R$ has finitely many units, then tensoring this resolution with $\otimes_\Z \Q$ in fact yields a resolution of projective $\Q[\Sp_{2n}(R)]$-modules.
\end{theorem}
%%--------------------------------------------------------------------

%%--------------------------------------------------------------------
\subsection{Simplicial Complexes Associated to Symplectic Vector Spaces}
\label{sec:simpcpx}
%%--------------------------------------------------------------------

Here and throughout the rest of this paper, we will define simplicial complexes by specifying that their simplices are certain sets.
What we mean by this is that the $k$-simplices are such sets containing $(k + 1)$-elements, and the face relations between simplices are simply inclusions of sets.

%%----------------------------------------------------
\begin{definition}
    \begin{itemize}
    \item For a poset $\mc{P}$, let $|\mc{P}|$ denote its geometric realization. Thus $|\mc{P}|$ is the simplicial complex whose vertices correspond to elements of $\mc{P}$, and $k$-simplices correspond to length $(k+1)$ ordered chains $p_0<p_1<\dots<p_k$ of elements in $\mc{P}$.
    For an order-preserving map of posets $f: \mc{P} \to \mc{Q}$, let $|f|: |\mc{P}| \to |\mc{Q}|$ denote the corresponding map on geometric realizations.
        \item Suppose $X$ is a simplicial complex. We shall use $\mc{P}(X)$ to denote the associated poset, i.e. the elements of $\mc{P}(X)$ correspond to simplices of $X$, and the order relation comes from face relations.
        \item Suppose $f: \mc{P} \to \mc{Q}$ is an order-preserving map of posets. For an element $q \in \mc{Q}$, let the \emph{downward fiber of $q$}, denoted $f_{\leq q}$, be the subposet of $\mc{P}$ consisting of all elements $p$ such that $f(p) \leq q$.
    \end{itemize}
\end{definition}
%%----------------------------------------------------

%%----------------------------------------------------
\begin{remark}
For a simplicial complex $\bb{X}$, we have $|\mc{P}(\bb{X})|$ is homeomorphic to $\bb{X}$. Specifically, we have a simplicial isomorphism $|\mc{P}(\bb{X})| \cong \sd(\bb{X})$, where $\sd(\bb{X})$ denotes the barycentric subdivision of $\bb{X}$.
\end{remark}
%%----------------------------------------------------
%%----------------------------------------------------
\begin{definition}\label{defn:simpcpx}
Let $K$ be a field, and $V$ a symplectic $K$-vector space.
    \begin{enumerate}
        \item Let $\mc{K}(V)$ be the simplicial complex whose vertices correspond to lines in $V$, and such that every finite set of vertices forms a simplex. 
        \item Let $\mc{L}(V)$ be the subcomplex of $\mc{K}(V)$ consisting of those simplices whose vertices span a \emph{proper} subspace of $V$.
        \item Let $\mc{I}^g(V)$ be the subcomplex consisting of simplices spanning a genus $\leq g$ subspace.
    \end{enumerate}
\end{definition}
%%----------------------------------------------------

%%----------------------------------------------------
\begin{lemma}\label{HtpyEquiv} Let $V$ be a symplectic vector space over a field $K$ of genus $n$.
    \begin{enumerate}
        \item $\mc{K}(V)$ is contractible. 
        \item $\mc{L}(V) $ is homotopy equivalent to $\bb{T}(V)$. The homology group $\widetilde{\tH}_{2n-2}(\mc{L}(V))$ is generated by the classes determined by cycles of the form
        \begin{align*}
            \sum_{\sigma \in \Sigma_{2n}} (-1)^{\sgn(\sigma)}[\lin{\vec{v}_{\sigma(1)}}, \dots, \lin{\vec{v}_{\sigma(2n-1)}}]
        \end{align*}
        where $\vec{v}_1, \dots, \vec{v}_{2n}$ is an ordered basis of $V$, and $\Sigma_{2n}$ is the Weyl group of permutations of the set $\{1, 2, \dots, 2n\}$.
        \item $\mc{I}^0(V)$ is homotopy equivalent to $\bb{T}^{\w}(V)$. The homology group $\widetilde{\tH}_{n-1}(\mc{I}^0(V))$ is generated by the classes determined by cycles of the form
        \begin{align*}
            \sum_{\sigma \in \mc{B}_n} (-1)^{\sgn(\sigma)} [\lin{\vec{v}_{\sigma(1)}}, \dots, \lin{\vec{v}_{\sigma(n)}}]
        \end{align*}
        where $\vec{v}_1, \vec{v}_{\bar{1}}, \dots, \vec{v}_n, \vec{v}_{\bar{n}}$ is an ordered symplectic basis of $V$, and $\mc{B}_n$ is the Weyl group of signed permutations of the set $\{1, \bar{1}, \dots, n, \bar{n}\}$.
    \end{enumerate}
\end{lemma}
%%----------------------------------------------------
\begin{proof}
    \begin{enumerate}
        \item Pick any vertex $v$ in $\mc{K}(V)$, and note that $\sigma \cup \{v\}$ is a simplex in $\mc{K}(V)$ for any simplex $\sigma \in \mc{K}(V)$. Thus $\mc{K}(V)$ is a cone on $v$ and hence contractible.
        \item Define a map of posets $f: \mc{P}(\mc{L}(V)) \to \mc{T}(V)$ that sends a simplex $\sigma = \{\lin{\vec{v}_0}, \dots, \lin{\vec{v}_k}\}$ to the span of its vertices $\langle \vec{v}_0, \dots, \vec{v}_k\rangle$. Note that this is an order-preserving map, and that for any subspace $W \subset V$ the downward fiber $f_{\leq W}$ can be identified with $\mc{P}(\mc{K}(W))$, and is hence contractible. Hence by Quillen's Fiber Lemma (see for example, \cite[Proposition 2.5]{hatcher2017tethers}), $|f|$ is a homotopy equivalence. The assertion about $\widetilde{\tH}_{2n-2}(\mc{L}(V))$ follows by observing that the cycles in question map to apartments of $\mc{T}(V)$ under $f$, and apartment classes generate $\widetilde{\tH}_{2n-2}(\mc{T}(V))$ by Theorem \ref{SolTit}.
        \item This proof is exactly similar to the previous one. We have a map of posets $f: \mc{P}(\mc{I}^0(V)) \to \mc{T}^{\w}(V)$ that sends a simplex of $\mc{I}^0(V)$ to the span of the set of lines corresponding to its vertices. For any isotropic subspace $J$ of $V$, its downward fiber $f_{\leq J}$ is isomorphic to $\mc{P}(\mc{K}(J))$, and is hence contractible. Thus $|f|$ is a homotopy equivalence. The assertion about $\widetilde{\tH}_{n-1}(\mc{I}^0(V))$ is also similar, since the cycles in question map to symplectic apartments of $\mc{T}^{\w}(V)$ under $f$.
    \end{enumerate}
\end{proof}
%%----------------------------------------------------

%%--------------------------------------------------------------------
\subsection{A Resolution of (Symplectic) Steinberg Modules}
\label{sec:steinbergres}
%%--------------------------------------------------------------------
%%----------------------------------------------------

Now we will prove Proposition \ref{prop:introsteinbergres}, which we restate here for convenience.

\begin{prop}
\label{steinbergres}
    Let $K$ be a field, and $V$ a symplectic $K$-vector space of genus $n$. We have a resolution of $\Z[\Sp(V)]$-modules
    \begin{equation}\label{Steinbergs}
       0 \ra \St(V)  \longrightarrow  \bigoplus_{g(W) = n-1} \St(W) \otimes \St^{\w}(W^{\perp})  \longrightarrow  \dots  \longrightarrow  \bigoplus_{g(W) = 1} \St(W) \otimes \St^{\w}(W^{\perp})  \longrightarrow  \St^{\w}(V) \to 0 
    \end{equation}  
    where the direct sums are taken over symplectic subspaces of $V$.
\end{prop}
%%----------------------------------------------------
We will prove Proposition \ref{steinbergres} by studying the spectral sequence associated to the filtration $\mc{I}^0(V) \subset \mc{I}^1(V) \subset \dots \subset \mc{I}^n(V) = \mc{K}(V)$ defined in Definition \ref{defn:simpcpx}.

To do this, we need to determine the relative homology groups $\tH_\bullet(\mc{I}^{g+1}(V),\mc{I}^g(V))$. 
Our result is as follows.
%%----------------------------------------------------
\begin{lemma}
    \label{IgQuotient}
    Suppose $V$ is a genus $n$ symplectic vector space over a field $K$. For each $g \geq 0$, the homology of the pair $(\mc{I}^{g+1}(V), \mc{I}^g(V))$ is given by
    \begin{align*}
        \tH_k(\mc{I}^{g+1}(V), \mc{I}^g(V)) = \begin{cases}
            \bigoplus_{W\subset V \text{ symplectic, } g(W)=g+1} \St(W) \otimes \St^\w(W^\perp) & k=n+g \\ 0 & \text{otherwise}
        \end{cases}
    \end{align*}
\end{lemma}
%%----------------------------------------------------
We will prove this proposition by defining certain intermediate complexes $\mc{I}^{g+\frac{1}{2}}(V)$ such that $\mc{I}^{g}(V) \subset \mc{I}^{g+\frac{1}{2}}(V) \subset \mc{I}^{g+1}(V)$, and analyzing the homologies of the pairs $(\mc{I}^{g+1}(V), \mc{I}^{g+\frac{1}{2}}(V))$ and $(\mc{I}^{g+1}(V), \mc{I}^g(V))$.
%%----------------------------------------------------
\begin{definition}
    Suppose $V$ is a symplectic $K$-vector space of genus $n$. Let $\mc{I}^{g+\frac{1}{2}}(V)$ be the subcomplex consisting of $\mc{I}^{g}(V)$ along with simplices $\sigma = \{ \vec{w}_1, \dots, \vec{w}_p, \vec{j}_1, \dots, \vec{j}_q\}$ such that $\vec{w}_1, \dots, \vec{w}_p$ span a symplectic vector space $W \subset V$, and $\vec{j}_1, \dots, \vec{j}_q$ span an isotropic subspace $J$ of $W^\perp$.
\end{definition}
%%----------------------------------------------------
Thus we have $\mc{I}^{g}(V) \subset \mc{I}^{g+\frac{1}{2}}(V) \subset \mc{I}^{g+1}(V)$.

%%----------------------------------------------------
\begin{ex}
Let $V = \lin{\vec{e}_1, \vec{e}_{\bar{1}}, \dots, \vec{e}_n, \vec{e}_{\bar{n}}}$. The simplex $\{\lin{\vec{e}_1+\vec{e}_2}, \lin{\vec{e}_{\bar{1}}}, \lin{\vec{e}_{\bar{2}}}\}$ is in $\mc{I}^1(V)$ but not in $\mc{I}^{0.5}(V)$. But $\{\lin{\vec{e}_1}, \lin{\vec{e}_{\bar{1}}}, \lin{\vec{e}_{\bar{2}}}\}$ is in $\mc{I}^{0.5}(V)$, with $W = \langle \vec{e}_1, \vec{e}_{\bar{1}} \rangle$.
\end{ex}
%%----------------------------------------------------

First we analyse the quotient $\mc{I}^{g+1}(V)/\mc{I}^{g+\frac{1}{2}}(V)$.
%%----------------------------------------------------
\begin{lemma}\label{IgContractible}
   Let $V$ be a symplectic $K$-vector space of genus $n$. For $g \geq 0$, $\mc{I}^{g+1}(V)/\mc{I}^{g+\frac{1}{2}}(V)$ is contractible.
\end{lemma}
%%----------------------------------------------------

Our proof will make use of the following Lemma from Miller--Patzt--Wilson (\cite{miller2025rank}):
	
%%----------------------------------------------------
	\begin{lemma}[{Miller--Patzt--Wilson \cite[Lemma 3.2]{miller2025rank}}]
    \label{Morse}
		Suppose $X$ is a simplicial complex with a subcomplex Y, and suppose there is a set of simplices $S$ of $X$ such that the following conditions are met:
		\begin{enumerate}
			\item Suppose $\sigma$ is a simplex in $X$. Then $\sigma$ is a simplex in $Y$ if and only if no face of $\sigma$ is in $S$. 
            %(Thus $X$ is obtained from $Y$ by gluing in the simplices in $S$, and nothing else)
			\item If $\sigma$ and $\tau$ are disinct simplices in $S$, then the union of their vertices does not form a simplex in $X$.
		\end{enumerate}	
		
		Then 
		\begin{align*}
		X/Y \simeq \bigvee_{\sigma \in S} \Sigma^{\dim \sigma + 1} \Link_X(\sigma). \\
		\end{align*}
	\end{lemma}	
%%----------------------------------------------------
To prove Lemma \ref{IgContractible}, we shall construct a filtration
%%----------------------------------------------------
\begin{align*}
    \mc{I}^{g+\frac{1}{2}}(V) = F_{2g+2} \subset F_{2g+3} \subset ... \subset F_m \subset... \subset \mc{I}^{g+1}(V)
\end{align*}
%%----------------------------------------------------
such that
%%----------------------------------------------------
\begin{align*}
    \bigcup_{i \geq 2g+2} F_i = \mc{I}^{g+1}(V)
\end{align*}
%%----------------------------------------------------
and show that each successive quotient $F_{i+1}/F_i$ is contractible. 

We make some definitions as follows.
%%----------------------------------------------------
\begin{notation}
    \begin{itemize}
        \item For a simplex $\sigma$ in $\mc{K}(V)$, let $n(\sigma)$ be the number of vertices in $\sigma$. (Thus $n(\sigma)$ = $\dim \sigma +1$).
        \item Let $\rad(\sigma)$ denote the radical of the subspace spanned by the vectors of $\sigma$, and let $n(\rad(\sigma))$ be the number of vertices of $\sigma$ that are in $\rad(\sigma)$.
    \end{itemize}
\end{notation}
%%----------------------------------------------------

%%----------------------------------------------------
\begin{ex}
    If $\sigma = \{\lin{\vec{e}_1+\vec{e}_2}, \lin{\vec{e}_{\bar{1}}}, \lin{\vec{e}_{\bar{2}}}\}$, then $n(\sigma) = 3$. The radical of the subspace $\langle \vec{e}_1+\vec{e}_2, \vec{e}_{\bar{1}}, \vec{e}_{\bar{2}}\rangle$ is $\langle \vec{e}_{\bar{1}}-\vec{e}_{\bar{2}} \rangle$, and so $\rad(\sigma) = \langle \vec{e}_{\bar{1}}-\vec{e}_{\bar{2}}\rangle$. Since no vertices of $\sigma$ are in $\rad(\sigma)$, we have $n(\rad (\sigma)) = 0$.
\end{ex}
%%----------------------------------------------------
%We now define the filtration $ I^{g+\frac{1}{2}} = F_{2g+2} \subset ... \subset F_m \subset... \subset I^{g+1}$.

%%----------------------------------------------------
\begin{definition}
		For $i \geq 2g+2$, let $F_i = \mc{I}^{g+\frac{1}{2}}(V) \cup \{\sigma \in \mc{I}^{g+1}(V) | n(\sigma) - n(\rad(\sigma)) \leq i\}$. We are requiring $i \geq 2g+2$ since we need atleast $2g+2$ lines to span a genus $g+1$ symplectic space.
\end{definition}
%%----------------------------------------------------		

%%----------------------------------------------------
\begin{ex} 
    Suppose $V = \lin{\vec{e}_1, \vec{e}_{\bar{1}}, \dots, \vec{e}_n, \vec{e}_{\bar{n}}}$ and $g=0$. The simplex $\sigma = \{\lin{\vec{e}_1+\vec{e}_2}, \lin{\vec{e}_{\bar{1}}}, \lin{\vec{e}_{\bar{2}}}\}$ belongs to $F_3$.
\end{ex}
%%----------------------------------------------------
	
%%----------------------------------------------------	
\begin{remark}\label{faces}
		$F_i$ is closed under taking faces because both $\mc{I}^{g+\frac{1}{2}}(V)$ and the set $\{\sigma \in \mc{I}^{g+1}(V) | n(\sigma) - n(\rad(\sigma)) \leq i\}$ are individually closed under taking faces. Indeed, suppose $\sigma = [u_1, \dots, u_t, w_1, \dots, w_j]$, where $u_1, \dots, u_t$ are exactly the vertices in $\rad(\sigma)$, and $j = n(\sigma) - n(\rad(\sigma))$. Now suppose $\tau$ is a face of $\sigma$. Then note that $\rad(\sigma) \subset \rad(\tau)$, and so in particular the vertices in $\tau$ that are $\textit{not}$ in $\rad(\tau)$ are a subset of $w_1, \dots, w_j$, and thus, $n(\tau) - n(\rad(\tau)) \leq j = n(\sigma) - n(\rad(\sigma))$.
\end{remark}
%%----------------------------------------------------	

%%----------------------------------------------------
\begin{proof}[Proof of Lemma \ref{IgContractible}]
We shall first show that $F_{i+1}/F_i$ is contractible for each $i \geq 2g+2$.
	We will use Lemma \ref{Morse}, by setting $X = F_{i+1}$ and $Y = F_i$. Let $S = \{ \sigma \in \mc{I}^{g+1}(V) \setminus \mc{I}^{g+\frac{1}{2}}(V) | n(\sigma) = i+1, n(\rad (\sigma)) = 0\} $. Let us check that the conditions of Lemma \eqref{Morse} are satisfied:
	
	\begin{enumerate}
		\item Suppose $\sigma$ is in $X = F_{i+1} = \mc{I}^{g+\frac{1}{2}}(V) \cup \{\sigma \in \mc{I}^{g+1}(V) | n(\sigma) - n(\rad(\sigma)) \leq i+1\}$. 
        First we show that if $\sigma \in Y$, then no face of it is in $S$. 
        
        If $\sigma \in \mc{I}^{g+\frac{1}{2}}(V)$, then this implies that $\sigma \in Y$, and since all faces of $\sigma$ lie in $\mc{I}^{g+\frac{1}{2}}(V)$, no face can be in $S$. So assume that $\sigma \not\in \mc{I}^{g+\frac{1}{2}}(V)$ and that $\sigma \in Y$.
		Suppose that some face $\tau$ of $\sigma$ is in $S$. By definition, we have $n(\tau) - n(\rad(\tau)) = i+1$ and $\tau \not\in \mc{I}^{g+\frac{1}{2}}(V)$. 
		Then by Remark \eqref{faces}, we have $n(\sigma)-n(\rad(\sigma)) \geq n(\tau)-n(\rad(\tau)) \geq i+1$. Since $\sigma \in Y = F_i = \mc{I}^{g+\frac{1}{2}}(V) \cup \{\sigma \in \mc{I}^{g+1} | n(\sigma) - n(\rad(\sigma)) \leq i\}$, it follows that the only way that $\sigma$ can be in $Y$ is if $\sigma \in \mc{I}^{g+\frac{1}{2}}(V)$. But this gives us a contradiction. Thus no face of $\sigma$ can be in $S$.
		
		Now suppose that no face of $\sigma$ is in $S$. Suppose that $\sigma \not\in Y$. Then $\sigma \in X \setminus Y$ implies that we must have $n(\sigma)-n(\rad(\sigma)) = i+1$. Let $\tau$ be the face of $\sigma$ consisting of all those vertices that are not in the radical of $\sigma$. Then note that $n(\rad(\tau))=0$ and $n(\tau) = i+1$. Thus $\tau \in S$, a contradiction. Thus $\sigma$ has to be in $Y$.
		
		\item Suppose $\sigma$, $\tau$ are distinct simplices in $S$. For $\sigma \cup \tau$ to be in $X$, we would need $n(\sigma \cup \tau) - n(\rad(\sigma \cup \tau)) \leq i+1$. But note that $n(\sigma \cup \tau) > n(\tau) = i+1$, and $n(\rad(\sigma \cup \tau)) \leq n(\rad(\sigma)) + n(\rad(\tau)) = 0$, and so $n(\sigma \cup \tau) - n(\rad(\sigma \cup \tau)) > i+1$. Thus $\sigma \cup \tau$ cannot be in $S$.
		\end{enumerate}
		
		To now apply Lemma \eqref{Morse}, let us now examine the links of the simplices in $S$. Suppose $\sigma \in S$. Note that since $n(\sigma) = i+1 \geq 2g+3$, and $\sigma$ spans a genus $g+1$ subspace, the span of $\sigma$ has a nonempty radical. Let $v$ be a vertex in the radical of the span of $\sigma$. Suppose that $\tau \in X = F_{i+1}$ is such that $\sigma \cup \tau \in X$. In particular, this implies that $n(\sigma \cup \tau) - n(\rad(\sigma \cup \tau)) \leq i+1$. Since we have $n(\rad(\sigma \cup \tau)) \leq n(\rad(\tau)) + n(\rad(\sigma)) = n(\rad(\tau))$, any vertex of $\tau$ that is not a vertex of $\sigma$ must be in $\rad(\sigma \cup \tau)$, and in particular must be perpendicular to all vertices of $\sigma$. In particular, all vertices of $\tau$ are perpendicular to $v$, since $v$ is in the span of $\sigma$. This implies that $\sigma \cup \tau \cup \{ v\}$ is also in $X$. This argument shows that $\Link_X(\sigma)$ is a cone with cone point $v$. In particular, $\Link_X(\sigma)$ is contractible.
		Thus $F_{i+1}/F_i \simeq \bigvee_{\sigma \in S} \Sigma^{\dim \sigma + 1} \Link_X(\sigma)$ is contractible.
		
		 To deduce the contractibility of $\mc{I}^{g+1}(V)/\mc{I}^{g+\frac{1}{2}}(V) = \bigcup_{i} F_{i}/\mc{I}^{g+\frac{1}{2}}(V)$, note that the image of any map from a $p$-sphere to $\mc{I}^{g+1}(V)/\mc{I}^{g+\frac{1}{2}}(V)$ is compact, and so must lie in some $F_i/\mc{I}^{g+\frac{1}{2}}(V)$, and so we can nullhomotope it. Since $\mc{I}^{g+1}(V)/\mc{I}^{g+\frac{1}{2}}(V)$ is a CW-complex whose homotopy groups all vanish, it is contractible.
\end{proof}
%%----------------------------------------------------

Now we will analyse the homology of the pair $(\mc{I}^{g+\frac{1}{2}}(V),\mc{I}^{g}(V))$.
%%----------------------------------------------------
\begin{lemma}\label{IgSteinberg}
   Suppose $V$ is a genus $n$ symplectic vector space over a field $K$. For $g \geq 0$, the homology of the pair $(\mc{I}^{g+\frac{1}{2}}(V), \mc{I}^{g}(V))$ is concentrated in degree $n+g$, and the homology group in this degree is equal to
    \begin{align*}
        \bigoplus_{g(W)=g+1} \St(W) \otimes \St^{\w}(W^{\perp}) \\
       \end{align*} 
       where the direct sum is over symplectic subspaces $W \subset V$.
\end{lemma}
%%----------------------------------------------------
%%----------------------------------------------------
\begin{proof}
    The relative homology groups $\tH_{\bullet}(\mc{I}^{g+\frac{1}{2}}(V), \mc{I}^{g}(V))$ are calculated from the chain complex whose chain groups are
%%----------------------------------------------------
    \begin{align*}
       \frac{\tC_m(\mc{I}^{g+\frac{1}{2}}(V))}{\tC_m(\mc{I}^g(V))} \cong \bigoplus_{p+q=m-1}\Z[\lin{\vec{w}_0}, \dots, \lin{\vec{w}_p}, \lin{\vec{l}_0}, \dots, \lin{\vec{l}_q}]
    \end{align*}
%%----------------------------------------------------
    where the direct sum is over sets of vectors $\{\vec{w}_0, \dots, \vec{w}_p, \vec{l}_0, \dots, \vec{l}_q\}$ such that
%%----------------------------------------------------
    \begin{align*}
       & \vec{w}_0, \dots, \vec{w}_p \text{ span a symplectic subspace $W \subset V$ with $g(W)=g+1$}\\
       & \vec{l}_0, \dots, \vec{l}_q \in W^\perp \text{ are mutually perpendicular vectors}
    \end{align*}
%%----------------------------------------------------
    When we apply the differential to a term $[\lin{\vec{w}_0}, \dots, \lin{\vec{w}_p}, \lin{\vec{l}_0}, \dots, \lin{\vec{l}_q}]$ and we omit a $\vec{w}_i$ so that the span of $\vec{w}_0, \dots, \widehat{\vec{w}}_i, \dots, \vec{w}_p$ is strictly smaller than $W$, then this lands us in $\tC_{m-1}(\mc{I}^g(V))$.
    Thus the chain complex that computes $\tH_{\bullet}(\mc{I}^{g+\frac{1}{2}}(V), \mc{I}^{g}(V))$ is quasi-isomorphic to the following sum across tensor products of chain complexes, shifted up by one degree:
%%----------------------------------------------------
    \begin{align*}
        \bigoplus_{g(W)=g+1}\tC_{\bullet}(\mc{K}(W), \mc{L}(W)) \otimes \tC_{\bullet}(\mc{I}^0(W^\perp), \{*\})
    \end{align*}
%%----------------------------------------------------
    That is, the chain group in degree $m$ can be identified with
%%----------------------------------------------------
    \begin{align*}
        \bigoplus_{g(W)=g+1, p+q=m-1}\tC_{p}(\mc{K}(W), \mc{L}(W)) \otimes \tC_{q}(\mc{I}^0(W^\perp), \{*\})
    \end{align*}
%%----------------------------------------------------
    Consider the spectral sequence associated to this double complex. By Lemma \ref{HtpyEquiv}, each $\tC_\bullet(\mc{K}(W), \mc{L}(W))$ has homology equal to $\St(W)$ in degree $2(g+1)-2+1 = 2g+1$, and 0 otherwise, so filtering the double complex in the $p$-direction gives us an $\tE^1$ page with
%%----------------------------------------------------
    \begin{align*}
        E^1_{p,q} = \begin{cases}\bigoplus_{g(W)=g+1}\St(W) \otimes \tC_q(\mc{I}^0(W^{\perp})) & \text{ if } p = 2g+1 \\ 0 &\text{ otherwise }
        \end{cases}
    \end{align*}
%%----------------------------------------------------
    Also by Lemma \ref{HtpyEquiv}, each $\tC_\bullet(\mc{I}^0(W^\perp), \{*\})$ has homology concentrated in degree $n-(g+1)-1 = n-g-2$ equal to $\St^\w(W^\perp)$, so now filtering in the $q$-direction gives us an $\tE^2$ page with
%%----------------------------------------------------
    \begin{align*}
         \tE^2_{p,q} = \begin{cases}\bigoplus_{g(W)=g+1} \St(W) \otimes \St^{\w}(W^{\perp}) & \text{ if } p=2g+1\text{ and } q=n-g-2 \\ 0 & \text{ otherwise}
         \end{cases}
    \end{align*}
%%----------------------------------------------------
     Thus the spectral sequence collapses on the $\tE^2$ page, and so the homology of this double complex is
     \begin{align*}
         \bigoplus_{g(W)=g+1} \St(W) \otimes \St^{\w}(W^{\perp})
     \end{align*}
     in degree $(2g+1)+(n-g-2) = n+g-1$. Since we are considering a shift of this complex up by 1 degree, the result follows.
\end{proof}
%%----------------------------------------------------

%%----------------------------------------------------
\begin{proof}[Proof of Lemma \ref{IgQuotient}]
    Lemma \ref{IgContractible} says that the homology of the pair $(\mc{I}^{g+1}(V), \mc{I}^{g+\frac{1}{2}}(V))$ is 0 in every degree. Thus if we consider the long exact sequence in homology associated to the triple $(\mc{I}^{g+1}(V), \mc{I}^{g+\frac{1}{2}}(V), \mc{I}^{g}(V))$, we get
%%----------------------------------------------------
    \begin{align*}
        \tH_\bullet(\mc{I}^{g+1}(V), \mc{I}^g(V)) \cong \tH_\bullet(\mc{I}^{g+\frac{1}{2}}(V), \mc{I}^g(V))
    \end{align*}
%%----------------------------------------------------
    and so the result follows immediately from Lemma \ref{IgSteinberg}.
\end{proof}
%%----------------------------------------------------
%%----------------------------------------------------
    \begin{proof}[Proof of Proposition \ref{steinbergres}]
    Take the spectral sequence associated to the filtration $\mc{I}^0(V) \subset \mc{I}^1(V) \subset \dots \subset \mc{I}^n(V) = \mc{K}(V)$.
    The $\tE^1$ page of the sequence is given by
%%----------------------------------------------------
    \begin{align*}
        \tE^1_{p,q} =\begin{cases} \tH_{p+q}(\mc{I}^p(V), \mc{I}^{p-1}(V)) & p\geq1 \\ \tH_q(\mc{I}^0(V)) & p = 0 \end{cases}
    \end{align*}
%%----------------------------------------------------
    Lemma \ref{IgQuotient} says that for $p \geq 1$
%%----------------------------------------------------
    \begin{align*}
        \tH_{p+q}(\mc{I}^p(V),\mc{I}^{p-1}(V)) = \begin{cases}
            \bigoplus_{W\subset V \text{ symplectic, }g(W)=p} \St(W) \otimes \St^{\w}(W^{\perp}) & p+q = n+p-1 \\ 0 & \text{otherwise}
        \end{cases}
    \end{align*}
%%----------------------------------------------------
    Thus for $p \geq 1$, the only non-zero groups among $\tE^1_{p,q}$ are all in row $q=n-1$, given by 
%%----------------------------------------------------
    \begin{align*}
        \tE^1_{p, n-1} = \bigoplus_{W \subset V \text{symplectic, }g(W)=p} \St(W) \otimes \St^{\w}(W^{\perp})
    \end{align*}
%%----------------------------------------------------
    In the $p=0$ column the only nonzero entries are
%%----------------------------------------------------
    \begin{align*}
        \tE^1_{0,0} =& \tH_0(\mc{I}^0(V)) \cong \Z \\
        \tE^1_{0,n-1} =& \tH_{n-1}(\mc{I}^0(V)) \cong \St^\w(V)
    \end{align*}
%%----------------------------------------------------
    where these equalities follow from Lemma \ref{HtpyEquiv}.
The differential on the $\tE^1$-page goes horizontally from $\tE^1_{p, q} \to \tE^1_{p-1, q}$. Since the $\tE^1$ page only has nonzero entries in the row $q=n-1$ and at the entry $p=q=0$, we conclude that the spectral sequence has to collapse on page $\tE^2$. Since $\mc{K}(V)$ is contractible, the only nonzero entry on the $\tE^\infty$ page is $\tE^\infty_{0,0} \cong \Z$. In particular, the $q=n-1$ row on the $\tE^2$ page consists of all zeroes, and so it follows that the $q=n-1$ row on the $\tE^1$ page is exact. This proves our proposition.
\end{proof}
%%----------------------------------------------------

%%--------------------------------------------------------------------
\section{The Symplectic Sharbly Resolution}
\label{sec:sharblypf}
%%--------------------------------------------------------------------

%\flag{Incorporate language of `mapping cones'}

In this section we will prove Theorem \ref{thm:sharblyres}. It will be easier to state and prove a rational version of Theorem \ref{thm:sharblyres} where we tensor everything with $\otimes_\Z \Q$ and assume that $R$ has finitely many units, but the argument turns out to be essentially the same for the general case. For this reason, we will describe the proof in detail for the rational version, and then briefly describe how to handle the integral case in Section \ref{sec:integralres}.

%%--------------------------------------------------------------------
    \subsection{Piecing projective resolutions together}\label{glueprocess}
%%--------------------------------------------------------------------
    In this subsection we deduce the groups we get in our symplectic Sharbly resolution:
%%----------------------------------------------------
    \begin{prop}
    \label{sharblyterms}
    Let ${R}$ be a PID with finitely many units and $K$ its field of fractions. Suppose $M \cong {R}^{2n}$ is a symplectic ${R}$-module and let $V = M \otimes_{{R}} K$. There is a resolution of $\St^\w(V)_\Q$ of projective $\Q[\Sp(M)]$-modules, which we shall refer to as the \emph{symplectic Sharbly resolution}, whose degree-$d$ term is given by 
%%----------------------------------------------------
    \begin{align*}
        \bigoplus \Sh_{k_1}(W_1)_\Q \otimes_\Q \dots \otimes_\Q \Sh_{k_m}(W_m)_\Q,
    \end{align*} 
%%----------------------------------------------------
    where the direct sum ranges over (ordered) direct sum decompositions 
%%----------------------------------------------------
    \begin{align*}
        V = W_1 \oplus \dots \oplus W_m
    \end{align*} 
%%----------------------------------------------------
    of $V$ into perpendicular symplectic subspaces, $\Sh_{k_i}(W_i)$ is the $k_i$-th term in the Sharbly resolution for $W_i$ (as defined in Proposition \ref{prop:sharbly}), and 
%%----------------------------------------------------
    \begin{align*}
        d = k_1 + k_2 + \dots + k_m + n -m
    \end{align*}

    with $d$ and $n$ fixed, $1 \leq m \leq n$, and $k_1, \dots k_m$ varying in $\Z_{\geq0}$
%%----------------------------------------------------
\end{prop}
%%----------------------------------------------------
We proceed by induction on $n$. Note that we always have $1 \leq m \leq n$. So for $n=1$, the only possibility is that $m=1$, and in this case we recover the usual Sharbly resolution for $V \cong K^2$ over $\SL_2R = \Sp_{2 \cdot 1}R$.
For the induction step, note that we have now established the following resolution:
%%----------------------------------------------------
    \begin{align*}
         0 \ra \St(V)_\Q \longrightarrow  \bigoplus_{g(W) = n-1} \St(W)_\Q \otimes \St^{\w}(W^{\perp})_\Q  \longrightarrow  \dots  \longrightarrow  \bigoplus_{g(W) = 1} \St(W)_\Q \otimes \St^{\w}(W^{\perp})_\Q  \longrightarrow  \St^{\w}(V)_\Q \to 0 
    \end{align*}
%%----------------------------------------------------
We will use this to obtain a projective $\Q[\Sp(M)]$-resolution of $\St^{\w}(V)_\Q$ by taking projective resolutions of each intermediate term $\bigoplus \St(W)_\Q \otimes \St^{\w}(W^{\perp})_\Q$, and `piecing them together'.
To do this we will need the following lemma.

%%----------------------------------------------------
\begin{lemma}
\label{gluestep}
    Let $\mc{R}$ be a ring. Suppose we have a resolution of left $\mc{R}$-modules
%%----------------------------------------------------
    \begin{align*}
        \dots \xrightarrow{\del_3} F_2 \xrightarrow{\del_2} F_1 \xrightarrow{\del_1} F_0 \xrightarrow{\del_0} M_g \xrightarrow{d_g} M_{g-1} \xrightarrow{d_{g-1}} \dots \xrightarrow{d_2} M_1 \xrightarrow{d_1} M \ra 0 
    \end{align*}
%%----------------------------------------------------
 where each $F_i$ is projective, and suppose that 
%%----------------------------------------------------
    \begin{align*}
        \dots \xrightarrow{\del_3'} F_2' \xrightarrow{\del_2'} F_1' \xrightarrow{\del_1'} F_0' \xrightarrow{\del_0'} M_g \ra 0 
    \end{align*}
%%----------------------------------------------------
    is a resolution (not necessarily projective) of $M_g$.

    Then the following is also a resolution of left $\mc{R}$ modules
%%----------------------------------------------------
    \begin{align*}
     \dots \ra F_1 \oplus F_2' \xrightarrow{\del_1'' \oplus \del_2'} F_0 \oplus F_1' \xrightarrow{\del_0'' \oplus \del_1'} F_0' \xrightarrow{d_g \circ \del_0'} M_{g-1} \ra \dots \ra M_1 \ra M \ra 0 \\
    \end{align*} 

    Here the $\del_i''$ are a choice of (anti-commuting) chain map from the projective resolution

    \begin{align*}
        \dots \xrightarrow{\del_3} F_2 \xrightarrow{\del_2} F_1 \xrightarrow{\del_1} F_0 \xrightarrow{\del_0} \ker(M_g \to M_{g-1}) \to 0
    \end{align*}

    to the resolution

    \begin{align*}
        \dots \xrightarrow{\del_3'} F_2' \xrightarrow{\del_2'} F_1' \xrightarrow{\del_1'} F_0' \xrightarrow{\del_0'} M_g \ra 0 
    \end{align*}
%%----------------------------------------------------
\end{lemma}
%%----------------------------------------------------

%%----------------------------------------------------
\begin{proof}
    Let $\mc{K} = \ker(M_g \xrightarrow{d_g} M_{g-1})$. We have the following diagram:

%%----------------------------------------------------
\begin{center}
\begin{tikzcd}
 & \vdots \arrow{d}{\del_2} & \vdots \arrow{d}{\del_2'} \\
        & F_1 \arrow{d}{\del_1} & F_1' \arrow{d}{\del_1'} \\
        & F_0 \arrow{d}{\del_0} &  F_0' \arrow{d}{\del_0'} \\
     0 \arrow[r] & \mc{K} \arrow[d] \arrow[r] & M_g \arrow[d] \\
        & 0 & 0 \\
\end{tikzcd}
\end{center}
%%----------------------------------------------------
    where $\mc{K} \ra M_g$ is the inclusion map. Both the columns and the bottom row are exact.
    Since the $F_i$ are projective, we can extend $\mc{K} \to M_g$ to a chain map:
%%----------------------------------------------------
\begin{center}
\begin{tikzcd}
 & \vdots \arrow{d}{\del_2} & \vdots \arrow{d}{\del_2'} \\
        & F_1 \arrow{d}{\del_1} \arrow{r}{\del_1''} & F_1' \arrow{d}{\del_1'} \\
        & F_0 \arrow{d}{\del_0} \arrow{r}{\del_0''} &  F_0' \arrow{d}{\del_0'} \\
     0 \arrow[r] & \mc{K} \arrow[r] & M_g \\
\end{tikzcd}
\end{center}
%%----------------------------------------------------

We can choose the $\del_i''$ so that the chain maps are \emph{anti-commuting}, i.e. $\del_{i}' \circ \del_i'' + \del_{i-1}'' \circ \del_i = 0$. The two columns consisting of the $F_i$ and $F_i'$ form a double complex.
If we first take homology in the vertical direction, we get 0 in each row, except in the bottom row, where we get $\mc{K} \ra M_g$. If we now take homology in the horizontal direction, we get $\coker(\mc{K} \ra M_g) \cong \im(M_g \xrightarrow{d_g} M_{g-1})$. 
Thus the total complex of the above double complex gives a resolution of $\coker(\mc{K} \ra M_g) \cong \im(M_g \xrightarrow{d_g} M_{g-1})$, i.e.
%%----------------------------------------------------
\begin{align*}
    \dots \ra F_1 \oplus F_2' \xrightarrow{\del_1'' \oplus \del_2'} F_0 \oplus F_1' \xrightarrow{\del_0'' \oplus \del_1'} F_0' \xrightarrow{d_g \circ \del_0'} \im(M_g \ra M_{g-1}) \ra 0 
\end{align*}
%%----------------------------------------------------
is exact.
It follows that we have a resolution
%%----------------------------------------------------
     \begin{align*}
     \dots \ra F_1 \oplus F_2' \xrightarrow{\del_1'' \oplus \del_2'} F_0 \oplus F_1' \xrightarrow{\del_0'' \oplus \del_1'} F_0' \xrightarrow{d_g \circ \del_0'} M_{g-1} \ra \dots \ra M_1 \ra M \ra 0 
    \end{align*} 
%%----------------------------------------------------
    as desired.
\end{proof}
%%----------------------------------------------------

We will apply Lemma \ref{gluestep} when $\mc{R} = \Q [\Sp(M)]$, and where the $F_i'$ (and thus also the $F_i \oplus F_{i+1}')$  are also projective. Thus this Lemma gives us a way of replacing $M_g$, which is not necessarily projective, with projective modules.

We need a way to recognise projective $\Q[\Sp(M)]$-modules. For this, we will use the following lemma from Church-Putman \cite{church2017codimension}. This is the part where we will require ${R}$ to have finitely many units.

%%----------------------------------------------------
\begin{lemma}[{Church-Putman\cite[Lemma 3.2]{church2017codimension}}]
 \label{CPproj} Let $G$ be a group acting simplicially on a simplicial complex $X$, and let $Y \subset X$ be a subcomplex preserved by the action of $G$. For some $n \geq 0$, assume that the setwise stabilizer group $G_{\sigma}$ is finite for every $n$-simplex $\sigma$ of $X$ that is not contained in $Y$. Then the $G$-module $\tC_n(X,Y; \Q)$ of relative simplicial chains is projective over $\Q [G]$.
\end{lemma}

We will also need the following lemma, whose proof is standard.

\begin{lemma}
    \label{lem:prodproj}
If $P_1$ and $P_2$ are projective modules over $\Q [G_1]$ and $\Q [G_2]$, respectively, then $P_1 \otimes_\Q P_2$ is projective over $\Q [G_1] \otimes_\Q \Q [G_2] \cong \Q[G_1 \times G_2]$.
\end{lemma}

%%----------------------------------------------------

We are now ready to prove Proposition \ref{sharblyterms}.

%%----------------------------------------------------
\begin{proof}[Proof of Proposition \ref{sharblyterms}]
If we have a resolution 
%%----------------------------------------------------
\begin{align*}
    0 \to M_g \to M_{g-1} \ra \dots M_1 \to M \to 0
\end{align*}
%%----------------------------------------------------
 and a projective resolution $F_{i, \bullet} \to M_i$ for each $1 \leq i \leq n$, then repeatedly applying Lemma \ref{gluestep} gives us a way to get a projective resolution of $M$.
 Note that in doing so, the $d$-th term in the resulting resolution will be $\bigoplus_{a+b-1=d, 0 \leq a \leq g} F_{a,b}$. (The shift of 1 is due to the fact that we chose to index the resolution $\dots \to M_1 \to M \to 0$ starting at 1, and not 0.)
We shall use this to apply Lemma \ref{gluestep} to our resolution \eqref{Steinbergs}, tensored with $\Q$:

%%----------------------------------------------------
\begin{align*}
    0 \ra \St(V)_\Q \longrightarrow  \bigoplus_{g(W) = n-1} \St(W)_\Q \otimes_\Q \St^\w(W^{\perp})_\Q \longrightarrow  \dots  \longrightarrow \bigoplus_{g(W) = 1} \St(W)_\Q \otimes_\Q \St^\w(W^{\perp})_\Q \longrightarrow \St^\w(V)_\Q \to 0 
\end{align*}
%%----------------------------------------------------

To do this, we will need a projective resolution for each of the terms $\bigoplus_{g(W)=g} \St(W)_\Q \otimes \St^{\w}(W^{\perp})_\Q$.

We can inductively assume Proposition \ref{sharblyterms} for symplectic spaces with genus $\leq n-1$. Thus, for a fixed $W$ of genus $g \geq 1$, we have the normal Sharbly resolution $\Sh_{k}(W)_\Q$ for $\St(W)_\Q$, and for $\St^{\w}(W^{\perp})_\Q$, we have a resolution whose degree-$d$ term is given by
%%----------------------------------------------------
\begin{align*}
\bigoplus \Sh_{k_2}(W_2)_\Q \otimes_\Q \dots \otimes_\Q \Sh_{k_m}(W_m)_\Q
\end{align*} 
%%----------------------------------------------------
where
%%----------------------------------------------------
\begin{align*}
   & W_2 \oplus \dots \oplus W_m = W^\perp \\
    & k_2+\dots+k_m + (n-g) - (m-1) = d 
\end{align*}
%%----------------------------------------------------

By tensoring these resolutions for $\St(W)_\Q$ and $\St^{\w}(W^{\perp})_Q$ we get a resolution for $\St(W)_\Q \otimes_\Q \St^{\w}(W^{\perp})_\Q$. The degree $d$ term of this resolution will be 
%%----------------------------------------------------
\begin{align*}
    \bigoplus \Sh_{k_1}(W)_\Q \otimes_\Q \Sh_{k_2}(W_2)_\Q \otimes_\Q \dots \otimes_\Q \Sh_{k_m}(W_m)_\Q
\end{align*}
%%----------------------------------------------------
where 
%%----------------------------------------------------
\begin{align*}
    & W_2 \oplus \dots \oplus W_m = W^\perp \\
    & d= (k_2+\dots+k_m + (n-g) - (m-1)) + k_1 \\
    & = k_1 + k_2 + \dots + k_m + n-m - (g-1) 
\end{align*}
%%----------------------------------------------------

Taking the direct sum of these resolutions as $W$ ranges over all genus $g$ symplectic subspaces of $V$ gives us a resolution of $\bigoplus_{g(W)=g} \St(W)_\Q \otimes_\Q \St^{\w}(W^{\perp})_\Q$.

To apply Lemma \ref{gluestep}, we need to check that this resolution is projective over $\Q[\Sp(M)]$, which we defer to the end of the proof. Assuming projectivity, let us see what resolution this yields for $\St^{\w}(V)_\Q$.

Let $F_{a,\bullet}$ be the resolution obtained for $\bigoplus_{g(W)=a} \St(W)_\Q \otimes_\Q \St^{\w}(W^{\perp})_\Q$. 
Thus, note that 
%%----------------------------------------------------
\begin{align*}
    F_{a,b} = \bigoplus \Sh_{k_1}(W)_\Q \otimes_\Q \Sh_{k_2}(W_2)_\Q \otimes_\Q \dots \otimes_\Q \Sh_{k_m}(W_m)_\Q
\end{align*}
%%----------------------------------------------------
where the direct sum ranges over:
%%----------------------------------------------------
\begin{align*}
   & g(W)=a \\
    & W_2 \oplus \dots \oplus W_m = W^\perp \text{ is a decomposition into symplectic subspaces} \\
    & k_1 + k_2 + \dots + k_m + n-m - (a-1) = b  \implies k_1 + k_2 + \dots + k_m + n-m = a+b-1
\end{align*}
%%----------------------------------------------------

Thus applying Lemma \ref{gluestep} gives us a resolution whose degree $d$ terms are:
%%----------------------------------------------------
\begin{align*}
    \bigoplus \Sh_{k_1}(W_1)_\Q \otimes_\Q \dots \otimes_\Q \Sh_{k_m}(W_m)_\Q
\end{align*}
%%----------------------------------------------------
where 
%%----------------------------------------------------
\begin{align*}
    k_1 + \dots + k_m + n-m = d 
\end{align*}
%%----------------------------------------------------
as desired.

We are left now to check projectivity over $\Q[\Sp(M)]$. Lemma \ref{CPproj} implies that $\Sh_k(W)_\Q$ is projective over $\Q[\Sp(W)]$ for any symplectic $W \subset V$: we can take $X = \mc{K}(W)$ and $Y = \mc{L}(W)$ (See \ref{defn:simpcpx} for the definitions of $\mc{K}(W)$ and $\mc{L}(W)$), and note that $\Sh_k(W)_\Q \cong \tC_{\dim W +k}(\mc{K}(W), \mc{L}(W); \Q)$.
Lemma \ref{lem:prodproj} implies that $\Sh_{k_1}(W_1)_\Q \otimes_\Q \dots \otimes_\Q \Sh_{k_m}(W_m)_\Q$ is projective over $\Q[\Sp(W_1) \times \dots \times \Sp(W_m)]$. Note that for fixed $k_1, \dots, k_m$, we have:

\begin{align*}
    \bigoplus_{\dim W_i' = \dim W_i} \Sh_{k_1}(W_1')_\Q \otimes_\Q \dots \otimes_\Q \Sh_{k_m}(W_m')_\Q \cong \Ind_{\Sp(W_1) \times \dots \times \Sp(W_m)}^{\Sp(V)} \Sh_{k_1}(W_1)_\Q \otimes_\Q \dots \otimes_\Q \Sh_{k_m}(W_m)
\end{align*}

where the direct sum ranges over all symplectic decompositions $V = W_1' \oplus \dots \oplus W_m'$ with $\dim W_i' = \dim W_i$ for all $i$.
Thus 
\begin{align*}
    \bigoplus_{\dim W_i' = \dim W_i} \Sh_{k_1}(W_1')_\Q \otimes_\Q \dots \otimes_\Q \Sh_{k_m}(W_m')_\Q
\end{align*}

is a direct summand of 

\begin{align*}
    \Ind_{\Sp(W_1) \times \dots \times \Sp(W_m)}^{\Sp(V)} F,
\end{align*}

where $F$ is a free module over $\Q[\Sp(W_1) \times \dots \times \Sp(W_m)]$. Thus the above induced representation is a free module over $\Q[\Sp(V)]$. Since $\bigoplus_{\dim W_i' = \dim W_i} \Sh_{k_1}(W_1')_\Q \otimes_\Q \dots \otimes_\Q \Sh_{k_m}(W_m')_\Q$ is a direct summand of this free $\Q[\Sp(V)]$-module, it is hence projective over $\Q[\Sp(V)]$. The module we want to be projective for this proof is a direct sum of projective modules of the above form, and is thus also projective.
\end{proof}
%%----------------------------------------------------

%%--------------------------------------------------------------------
%\section{The Differential}
\subsection{The Differential on Eq \eqref{Steinbergs}}
\label{diffeq1} 
%%--------------------------------------------------------------------
We will find the differential on Eq \eqref{Steinbergs} by evaluating it on chain representatives of homology classes in $\tH_{n+g}(\mc{I}^{g+1}, \mc{I}^g) \cong \bigoplus_{g(W)=g+1}\St(W) \otimes \St^\w(W^\perp)$. So we first set some notation for dealing with chains in a chain complex.

%%----------------------------------------------------
\begin{Note}\label{notn:simchains}
\begin{itemize} Let $V$ be a symplectic $K$-vector space of genus $n$.
    \item     For a linearly independent ordered set of vectors $\vec{v}_0 , \dots, \vec{v}_{2g-1}$ spanning a symplectic subspace $W \subset V$, let $\llbracket \lin{\vec{v}_0}, \dots, \lin{\vec{v}_{2g-1}} \rrbracket$ denote the cycle of $\tC_\bullet(\mc{K}(V), \mc{L}(V))$ corresponding to the (oriented) simplex with $\langle \vec{v}_i \rangle$ as its vertices.

    \item When $\vec{w}_1, \vec{w}_{\bar{1}}, \dots, \vec{w}_g, \vec{w}_{\bar{g}}$ form an ordered symplectic basis of a symplectic subspace $W \subset V$, let 
    \newline
    $\llbracket \lin{\vec{w}_1}, \lin{\vec{w}_{\bar{1}}}, \dots, \lin{\vec{w}_g}, \lin{\vec{w}_{\bar{g}}} \rrb ^\w$ denote the cycle of $\mc{I}^0(V)$, defined inductively as follows:
    When $g=1$, we let $\llbracket \lin{\vec{w}_1}, \lin{\vec{w}_{\bar{1}}} \rrbracket ^\w$ denote the chain $[\lin{\vec{w}_1}] - [\lin{\vec{w}_{\bar{1}}}]$. 
    In general, for a partial symplectic basis $v_1, w_1, \dots, v_g, w_g$, inductively define
%%----------------------------------------------------
    \begin{IEEEeqnarray*}{l}
        \hspace*{-2em} \llbracket \lin{\vec{w}_1}, \lin{\vec{w}_{\bar{1}}}, \dots, \lin{\vec{w}_g}, \lin{\vec{w}_{\bar{g}}} \rrbracket ^\w \coloneq (\llbracket \lin{\vec{w}_1}, \lin{\vec{w}_{\bar{1}}} \rrbracket ^\w) * (\llbracket \lin{\vec{w}_2}, \lin{\vec{w}_{\bar{2}}}, \dots, \lin{\vec{w}_g}, \lin{\vec{w}_{\bar{g}}} \rrbracket ^\w)\\
       \quad = [\lin{\vec{w}_1}]*(\llbracket \lin{\vec{w}_2}, \lin{\vec{w}_{\bar{2}}}, \dots, \lin{\vec{w}_g}, \lin{\vec{w}_{\bar{g}}} \rrbracket ^\w) - [\lin{\vec{w}_{\bar{1}}}]*(\llbracket \lin{\vec{w}_2}, \lin{\vec{w}_{\bar{2}}}, \dots, \lin{\vec{w}_g}, \lin{\vec{w}_{\bar{g}}} \rrbracket ^\w) 
    \end{IEEEeqnarray*}
%%----------------------------------------------------
        \item Suppose $(\tC_\bullet, \del_{\bullet})$ is a chain complex with a subcomplex $\tD_\bullet$, and that $c_1, c_2 \in \tC_k$ are chains of degree $k$. We will say that $c_1 \simeq c_2$ if there exists a chain $c \in \tC_{k+1}$ and $d \in \tD_k$ such that $c_1-c_2 = \del c + d$.
        \item Suppose $(\tC_\bullet(X), \del_\bullet)$ is the chain complex of simplicial chains in a simplicial complex $X$. Suppose that
%%----------------------------------------------------
        \begin{align*}
            c_1 = \sum n_i\sigma_i \text{ , } c_2 = \sum m_j \tau_j
        \end{align*} 
%%----------------------------------------------------
are simplicial chains of degrees $k$ and $l$ respectively, such that
%%----------------------------------------------------
        \begin{align*}
            \sigma_i * \tau_j \in X
        \end{align*} 
%%----------------------------------------------------
for every $i,j$.
        Then the join $c_1*c_2$ is defined to be the degree $(k+l+1)$-chain
%%----------------------------------------------------
        \begin{align*}
            \sum_{i,j} n_im_j (\sigma_i*\tau_j)
        \end{align*}
%%----------------------------------------------------
        
    \end{itemize}
\end{Note}
%%----------------------------------------------------

%%----------------------------------------------------
\begin{remark}
\label{ChainReps}
    \begin{itemize}
        \item Note that Lemma \ref{HtpyEquiv} says that the images of the cycles $\llbracket \lin{\vec{v}_0}, \dots, \lin{\vec{v}_{2n-1}} \rrbracket$ of $\tC_\bullet(\mc{K}(V), \mc{L}(V))$ under the homology equivalence map $\tC_{\bullet}(\mc{K}(V), \mc{L}(V)) \to \widetilde{\tC}_{\bullet -1}(\mc{L}(V))$ generate the homology group $\widetilde{\tH}_{2n-2}(\mc{L}(V))$. Thus $\tH_{2n-1}(\mc{K}(V), \mc{L}(V)) \cong \widetilde{\tH}_{2n-2}(\bb{T}(V))$ is generated by the homology classes determined by the cycles $\llbracket \lin{\vec{v}_0}, \dots, \lin{\vec{v}_{2n-1}} \rrbracket$.
        \item Similarly, Lemma \ref{HtpyEquiv} says that the homology group $\widetilde{\tH}_{n-1}(\mc{I}^0(V)) \cong \widetilde{\tH}_{n-1}(\bb{T}^\w(V))$ is generated by the homology classes determined by the cycles $\llbracket \lin{\vec{w}_1}, \lin{\vec{w}_{\bar{1}}}, \dots, \lin{\vec{w}_n}, \lin{\vec{w}_{\bar{n}}} \rrbracket ^\w$, where $\vec{w}_1, \vec{w}_{\bar{1}}, \dots, \vec{w}_n, \vec{w}_{\bar{n}}$ form a symplectic basis of $V$.
    \end{itemize}
\end{remark}
%%----------------------------------------------------

We now state our main proposition:
%%----------------------------------------------------
\begin{prop}
    \label{Eq1Diff}
    Let $V$ be a symplectic vector space of genus $n$ over a field $K$.
    Then on the level of chain representatives, the differential on Eq \eqref{Steinbergs} is given as follows:
    
    Suppose 
    \begin{align*}
        \vec{v}_0, \dots, \vec{v}_{2g+1}
    \end{align*} form a basis of a symplectic genus $g+1$ subspace $W \subset V$, and suppose 
    \begin{align*}
       \vec{u}_1, \vec{u}_{\bar{1}}, \dots, \vec{u}_{n-g-1}, \vec{u}_{\overline{n-g-1}} 
    \end{align*} is a symplectic basis of $W^\perp$.
    For each pair $i < j$, say that $i \sim j$ if $\vec{v}_0, \dots, \widehat{\vec{v}}_i, \vec{v}_{i+1}, \dots, \vec{v}_{j-1}, \widehat{\vec{v}}_j, \dots, \vec{v}_{2g+1}$ span a symplectic subspace $W_{ij} \subset W$ of genus $g$.
    Then the differential on \eqref{Steinbergs} maps the homology class of
    \begin{align*}
        \llb \lin{\vec{v}_0}, \dots, \lin{\vec{v}_{2g+1}} \rrb * \llb  \lin{\vec{u}_1}, \lin{\vec{u}_{\bar{1}}}, \dots, \lin{\vec{u}_{n-g-1}}, \lin{\vec{u}_{\overline{n-g-1}}} \rrb ^\w
    \end{align*}
    to
    \begin{align*}
     \sum_{i \sim j} (-1)^{\sgn(\sigma)} \llb \lin{\vec{v}_0}, \dots, \lin{\widehat{\vec{v}}_i}, \lin{\vec{v}_{i+1}}, \dots, \lin{\vec{v}_{j-1}}, \lin{\widehat{\vec{v}}_j}, \dots, \lin{\vec{v}_{2g+1}} \rrb * \llb \lin{\vec{w}_i}, \lin{\vec{w}_j}, \lin{\vec{u}_1}, \lin{\vec{u}_{\bar{1}}}, \dots, \lin{\vec{u}_{n-g-1}}, \lin{\vec{u}_{\overline{n-g-1}}} \rrbracket ^\w   
    \end{align*}
    where 
    \begin{align*}
        \lin{\vec{w}_i}, \lin{\vec{w}_j}
    \end{align*} are the projections of $\lin{\vec{v}_i}, \lin{\vec{v}_j}$ respectively onto $W_{ij}^{\perp} \subset W$, and $\sigma$ is the shuffle permutation taking $\lin{\vec{v}_i}$ and $\lin{\vec{v}_j}$ to positions 0 and 1 in $[\lin{\vec{v}_0}, \dots, \lin{\vec{v}_{2g+1}}]$, i.e. $(-1)^{\sgn(\sigma)} = (-1)^{i+j-1}$.
\end{prop}
%%----------------------------------------------------
The proof will require the following main lemma.

%%----------------------------------------------------
\begin{comment}
\begin{lemma}
    \label{chainjoinequiv}
    Suppose we have a chain complex $(\tC_\bullet, \del_{\bullet})$ and a subcomplex $(\tD_\bullet, \del_{\bullet})$. Let $d$ be a chain in $\tD_{\bullet}$ with $\del d = 0$. Then, for chains $c_1, c_2$ in $\tC_{\bullet}$, if $\del c_1 \simeq c_2$, then $\del(c_1 * d) \simeq c_2 * d$.
\end{lemma}
%%----------------------------------------------------
%%----------------------------------------------------
\begin{proof}
Since $\del c_1 \simeq c_2$, we have
%%----------------------------------------------------
\begin{align*}
    \del c_1 = c_2 + \del c
\end{align*}
%%----------------------------------------------------
for some chain $c$.
Also note that for any chain $c$, we have
%%----------------------------------------------------
\begin{align*}
    \del(c*c_3) = & (\del c) * c_3 + c*(\del c_3) \\
    = & \del c * c_3 + 0
\end{align*}
%%----------------------------------------------------
and thus $(\del c)*c_3$ is a boundary.
    Now we have
%%----------------------------------------------------
    \begin{align*}
        \del(c_1*c_3) = & \del(c_1)*c_3 + c_1 *(\del c_3)\\
        = &  \del(c_1)*c_3 + 0 \\
        \simeq & (c_2+\del c)*c_3 \\
        =& c_2*c_3 + (\del c)*c_3
    \end{align*}
%%----------------------------------------------------
    and thus $\del(c_1*c_3) \simeq c_2*c_3$, as desired.
\end{proof}
\end{comment}
%%----------------------------------------------------

%%----------------------------------------------------
\begin{lemma}\label{Eq1diffchains}
Let $V$ be a symplectic $K$-vector space of genus $n$. For $g \geq 0$, suppose $W \subset V$ is a genus $g+1$ symplectic subspace, and let $\vec{v}_0, \dots, \vec{v}_{2g+1}$ be a spanning set of vectors for $W$.
    For each pair $i < j$, say that $i \sim j$ if $\vec{v}_0, \dots, \widehat{\vec{v}}_i, \vec{v}_{i+1}, \dots, \vec{v}_{j-1}, \widehat{\vec{v}}_j, \dots, \vec{v}_{2g+1}$ span a symplectic subspace $W_{ij} \subset W$ of genus $g$.
    Then, with respect to the chain complex $\tC_{\bullet}(\mc{I}^g(V))$ and its subcomplex $\tC_{\bullet}(\mc{I}^{g-1}(V))$, we have the following equivalence of chains in the sense of Notation \ref{notn:simchains}:
%%----------------------------------------------------
    \begin{align*}
        \del \llbracket \lin{\vec{v}_0}, \lin{\vec{v}_1}, \dots, \lin{\vec{v}_{2g+1}} \rrbracket \simeq \sum_{i \sim j} (-1)^{i+j} \llb \lin{\vec{v}_0}, \dots, \lin{\widehat{\vec{v}}_i}, \lin{\vec{v}_{i+1}}, \dots, \lin{\vec{v}_{j-1}}, \lin{\widehat{\vec{v}}_j}, \dots, \lin{\vec{v}_{2g+1}} \rrb * \llb \lin{\vec{w}_i}, \lin{\vec{w}_j} \rrb^\w 
     \end{align*}
%%----------------------------------------------------
    Here $\lin{\vec{w}_i}, \lin{\vec{w}_j}$ are the projections of $\lin{\vec{v}_i}, \lin{\vec{v}_j}$ respectively onto $W_{ij}^{\perp} \subset W$.
\end{lemma}
%%----------------------------------------------------
%%----------------------------------------------------
\begin{proof}
    We have 
%%----------------------------------------------------
    \begin{align*}
       \del \llbracket \lin{\vec{v}_0}, \dots, \lin{\vec{v}_{2g+1}} \rrbracket = \sum_{0 \leq i \leq 2g+1} (-1)^i [\lin{\vec{v}_0}, \dots, \lin{\widehat{\vec{v}}_i}, \dots, \lin{\vec{v}_{2g+1}}]
    \end{align*}
%%---------------------------------------------------- 
     The span of the vectors $\vec{v}_0, \dots, \hat{\vec{v}}_i, \dots, \vec{v}_{2g+1}$ is a codimension 1 subspace of $\lin{\vec{v}_0, \dots, \vec{v}_{2g+1}}$ and must have one-dimensional radical. Let $\vec{a}_i$ be a vector spanning this radical. The vector $\vec{a}_i$ may or may not be present among the $\vec{v}_j$, for $j \neq i$.

Consider the simplex $[\lin{\vec{v}_0}, \lin{\vec{v}_1}, \dots, \lin{\vec{a}_i}, \lin{\vec{v}_{i+1}}, \dots, \lin{\vec{v}_{2g+1}}]$, where we replace $\vec{v}_i$ with $\vec{a}_i$. The span of these vectors has genus $g$, and so lies in $\mc{I}^{g}$. In particular, the boundary of this simplex forms a boundary in $\tC_\bullet(\mc{I}^{g})$.

From this we have the following equivalence:
%%----------------------------------------------------
\begin{align*}
    (-1)^i[\lin{\vec{v}_0}, \dots, \lin{\widehat{\vec{v}}_i},\dots, \lin{\vec{v}_{2g+1}}] = &(-1)^i[\lin{\vec{v}_0}, \dots, \lin{\widehat{\vec{a}}_i},\dots, \lin{\vec{v}_{2g+1}}] \\\simeq & \sum_{j \neq i} (-1)^{j-1}[\lin{\vec{v}_0}, \dots, \lin{\widehat{\vec{v}}_j}, \dots,  \lin{\vec{a}_i}, \dots, \lin{\vec{v}_{2g+1}}] 
\end{align*}
%%----------------------------------------------------

Now, $\vec{a}_i$ lies in the span of $\vec{v}_0, \dots, \widehat{\vec{v}}_i, \dots, \vec{v}_{2g-1}$, and so is a linear combination of them. Take the minimal set of $\vec{v}_j$'s so that $\vec{a}_i$ is a linear combination of them - in other words, take those $\vec{v}_j$'s that have non-zero coefficient when we write down $\vec{a}_i$ as a linear combination of $\vec{v}_0, \dots, \widehat{\vec{v}}_i, \dots, \vec{v}_{2g-1}$. For $\vec{v}_j$ that lie in this minimal set, call the pair $(i,j)$ \emph{good}.

If $(i,j)$ is \emph{not} good, note that the vertices of $[\lin{\vec{v}_0}, \dots, \lin{\widehat{\vec{v}}_j}, \dots, \lin{\vec{a}_i}, \dots, \lin{\vec{v}_{2g+1}}]$ has a span of dimension $2g$, with a non-zero radical (since $\vec{a}_i$ is in the radical), and thus must have genus $\leq g-1$, and thus this chain is in $\tC_{\bullet}(\mc{I}^{g-1}(V))$.

On the other hand, for $(i, j)$ good, the span of $[\lin{\vec{v}_0}, \dots, \lin{\widehat{\vec{v}}_j}, \dots, \lin{\vec{a}_i}, \dots, \lin{\vec{v}_{2g+1}}]$ has genus $g$, and can be rewritten as 
%%----------------------------------------------------
\begin{align*}
    (-1)^{\sgn(\sigma)}[\lin{\vec{v}_0}, \dots, \lin{\widehat{\vec{v}}_j}, \dots, \lin{\vec{v}_{2g+1}}] * [\lin{\vec{a}_i}], 
\end{align*} 
%%----------------------------------------------------
where $\sigma$ is the permutation that takes $\vec{a}_i$ to the right end (Thus $(-1)^\sigma = (-1)^{2g-i+1}$ if $i>j$, and $=(-1)^{2g-i}$ otherwise). In doing so, we have rewritten the simplex as a join of an apartment representative for a genus $g$ symplectic space, and a vector perpendicular to it.

Note that the above shows that $[\lin{\vec{v}_0}, \dots, \lin{\widehat{\vec{v}}_j}, \dots, \lin{\widehat{\vec{v}}_i}, \dots, \lin{\vec{v}_{2g+1}}]$ spans a genus $g$ subspace if and only if $(i, j)$ is good. This implies $(i,j)$ being good is equivalent to saying that $i \sim j$, and also that $\sim$ is symmetric: i.e. $i \sim j \implies j \sim i$.

For each such pair $i, j$ with $i \sim j$, (and suppose WLOG that $i<j$), collecting the terms, we have 
%%----------------------------------------------------
\begin{align*}
& (-1)^{j-1}(-1)^{2g-i}[\lin{\vec{v}_0}, \dots, \lin{\widehat{\vec{v}}_i}, \lin{\vec{v}_{i+1}}, \dots, \lin{\vec{v}_{j-1}}, \lin{\widehat{\vec{v}}_j}, \dots, \lin{\vec{v}_{2g+1}}]*[\lin{\vec{a}_i}] \\ &+ (-1)^{i-1}(-1)^{2g-j+1}[\lin{\vec{v}_0}, \dots, \lin{\widehat{\vec{v}}_i}, \lin{\vec{v}_{i+1}}, \dots, \lin{\vec{v}_{j-1}}, \lin{\widehat{\vec{v}}_j}, \dots, \lin{\vec{v}_{2g+1}}]*[\lin{\vec{a}_j}] \\
& = (-1)^{i+j}[\lin{\vec{v}_0}, \dots, \lin{\widehat{\vec{v}}_i}, \lin{\vec{v}_{i+1}}, \dots, \lin{\vec{v}_{j-1}}, \lin{\widehat{\vec{v}}_j}, \dots, \lin{\vec{v}_{2g+1}}]*([\lin{\vec{a}_j}]-[\lin{\vec{a}_i}]) 
\end{align*}
%%----------------------------------------------------

Let $W_{ij}$ be the span of $\vec{v}_0, \dots, \widehat{\vec{v}}_i, \vec{v}_{i+1}, \dots, {\vec{v}_{j-1}}, \widehat{\vec{v}}_j, \dots, \vec{v}_{2g+1}$. We have seen that $W_{ij}$ is symplectic of genus $g$, and that $\vec{a}_i, \vec{a}_j \in W_{ij}^{\perp}$. Moreover, the lines $\langle \vec{a}_j \rangle, \langle \vec{a}_i \rangle$ are actually the projections of $\langle \vec{v}_i \rangle, \langle \vec{v}_j \rangle$ respectively, onto $W_{ij}^{\perp}$ -- indeed, note that $\vec{a}_i$ is in the span of $\vec{v}_0, \vec{v}_1, \dots, \widehat{\vec{v}}_i, \dots, \vec{v}_{2g+1}$. By assumption, $\vec{v}_j$ has nonzero coefficient when we write $\vec{a}_i$ as a linear combination of $\vec{v}_0, \vec{v}_1, \dots, \widehat{\vec{v}}_i, \dots, \vec{v}_{2g+1}$. Thus, we can write
%%----------------------------------------------------
\begin{align*}
    \vec{v}_j = c \vec{a}_i + \text{ a linear combination of }\{\vec{v}_0, \dots, \widehat{\vec{v}}_i, \vec{v}_{i+1}, \dots, {\vec{v}_{j-1}}, \widehat{\vec{v}}_j, \dots, \vec{v}_{2g+1}\}
\end{align*}
%%----------------------------------------------------
Thus, if we let $\vec{w}_i, \vec{w}_j$ denote the projections of $\vec{v}_i, \vec{v}_j$, respectively, onto $W_{ij}^\perp$, then we have $\langle \vec{w}_i \rangle = \langle \vec{a}_j \rangle$ and $\langle \vec{w}_j \rangle = \langle \vec{a}_i \rangle$. 
Thus we have 
%%----------------------------------------------------
\begin{align*}
    [\lin{\vec{a}_j}]-[\lin{\vec{a}_i}] = \llbracket \lin{\vec{w}_i}, \lin{\vec{w}_j} \rrbracket ^\w 
\end{align*}
%%----------------------------------------------------

And so we have the following equivalence on chains with respect to the chain complex $\tC_{\bullet}(\mc{I}^g(V))$ and its subcomplex $\tC_{\bullet}(\mc{I}^{g-1}(V))$:
%%----------------------------------------------------
\begin{IEEEeqnarray*}{l}
    \hspace*{-2em} \del \llbracket \lin{\vec{v}_0}, \lin{\vec{v}_1}, \dots, \lin{\vec{v}_{2g+1}} \rrbracket  = \sum_{0 \leq i \leq 2g+1} (-1)^i [\lin{\vec{v}_0}, \dots, \lin{\widehat{\vec{v}}_i}, \dots, \lin{\vec{v}_{2g+1}}] \\
     \quad = \sum_{0 \leq i \leq 2g+1} (-1)^i [\lin{\vec{v}_0}, \dots, \lin{\widehat{\vec{a}}_i}, \dots, \lin{\vec{v}_{2g+1}}] \\
     \quad \simeq \sum_{0 \leq i \leq 2g+1} \sum_{j \neq i} (-1)^{j-1}[\lin{\vec{v}_0}, \dots, \lin{\widehat{\vec{v}}_j}, \dots,  \lin{\vec{a}_i}, \dots, \lin{\vec{v}_{2g+1}}] \\
     \quad \simeq \sum_{i \sim j} = (-1)^{i+j}[\lin{\vec{v}_0}, \dots, \lin{\widehat{\vec{v}}_i}, \lin{\vec{v}_{i+1}}, \dots, \lin{\vec{v}_{j-1}}, \lin{\widehat{\vec{v}}_j}, \dots, \lin{\vec{v}_{2g+1}}]*([\lin{\vec{a}_j}]-[\lin{\vec{a}_i}]) \\
    \quad \simeq \sum_{i \sim j} (-1)^{i+j} \llb \lin{\vec{v}_0}, \dots, \lin{\widehat{\vec{v}}_i}, \lin{\vec{v}_{i+1}}, \dots, \lin{\vec{v}_{j-1}}, \lin{\widehat{\vec{v}}_j}, \dots, \lin{\vec{v}_{2g+1}} \rrb * \llb \lin{\vec{w}_i}, \lin{\vec{w}_j} \rrb^w 
\end{IEEEeqnarray*}
%%----------------------------------------------------
as desired.
\end{proof}
%%----------------------------------------------------

%%----------------------------------------------------
%%---       Maybe insert some examples here
%%----------------------------------------------------

%%----------------------------------------------------
\begin{proof}[Proof of Proposition \ref{Eq1Diff}]
    From Lemma \ref{IgQuotient}, we have 
%%----------------------------------------------------
    \begin{align*}
        \tH_{n+g}(\mc{I}^{g+1}(V), \mc{I}^{g}(V)) \cong \bigoplus_{g(W) = g+1} \St(W) \otimes \St^\w(W^\perp)
    \end{align*} 
%%----------------------------------------------------
By Remark \ref{ChainReps}, chain representatives for generators of this homology are
%%---------------------------------------------------- 
    \begin{align*}
        \llbracket \lin{\vec{v}_0}, \lin{\vec{v}_1}, \dots, \lin{\vec{v}_{2g+1}} \rrbracket * \llbracket \lin{\vec{u}_1}, \lin{\vec{u}_{\bar{1}}}, \dots, \lin{\vec{u}_{n-g-1}}, \lin{\vec{u}_{\overline{n-g-1}}} \rrbracket ^\w,
    \end{align*} 
%%----------------------------------------------------
where 
%%----------------------------------------------------
    \begin{align*}
        \vec{v}_0, \vec{v}_1, \dots, \vec{v}_{2g+1}
    \end{align*} 
%%----------------------------------------------------
is a (not necessarily symplectic) basis of a genus $g+1$ symplectic subspace $W$, and 
%%----------------------------------------------------
    \begin{align*}
        \vec{u}_1, \vec{u}_{\bar{1}}, \dots, \vec{u}_{n-g-1}, \vec{u}_{\overline{n-g-1}}
    \end{align*}
%%----------------------------------------------------
a symplectic basis for $W^\perp$. Thus to figure out the differential in Eq \eqref{Steinbergs}, it is enough to evaluate the differential on these chain representatives.
    From the proof of Proposition \ref{steinbergres} we see that the way the differential 
%%----------------------------------------------------
    \begin{align*}
        \tH_{n+g}(\mc{I}^{g+1}(V), \mc{I}^g(V)) \to \tH_{n+g-1}(\mc{I}^g(V), \mc{I}^{g-1}(V))
    \end{align*}
%%----------------------------------------------------
    on the $\tE^1$-page works on a homology class $h\in \tH_{n+g}(\mc{I}^{g+1}(V), \mc{I}^g(V))$ is as follows: we take a chain $c \in \tC_{n+g}(\mc{I}^{g+1}(V))$ such that when $c$ is viewed in $\tC_{\bullet}(\mc{I}^{g+1}(V), \mc{I}^g(V))$, that $c$ is a chain representative for the homology class $h$. Note in particular that $\del c \in \tC_{n+g-1}(\mc{I}^{g}(V))$. Then we write a chain in $\tC_{n+g-1}(\mc{I}^g(V))$ that is a representative of a homology class in $\tH_{n+g-1}(\mc{I}^g(V), \mc{I}^{g-1}(V))$ and is equivalent to $\del c$, i.e. differs from $\del c$ by a chain in $\tC_{n+g-1}(\mc{I}^{g-1}(V))$.

    Now Lemma \ref{Eq1diffchains} tells us that with respect to $\tC_\bullet(\mc{I}^{g}(V))$ and its subcomplex $\tC_{\bullet}(\mc{I}^{g-1}(V))$, we have the following equivalence of chains:
%%----------------------------------------------------
    \begin{align*}
        \del \llbracket \lin{\vec{v}_0}, \lin{\vec{v}_1}, \dots, \lin{\vec{v}_{2g+1}} \rrbracket \simeq \sum_{i \sim j} (-1)^{i+j} \llb \lin{\vec{v}_0}, \dots, \lin{\widehat{\vec{v}}_i}, \lin{\vec{v}_{i+1}}, \dots, \lin{\vec{v}_{j-1}}, \lin{\widehat{\vec{v}}_j}, \dots, \lin{\vec{v}_{2g+1}} \rrb * \llb \lin{\vec{w}_i}, \lin{\vec{w}_j} \rrb^\w 
     \end{align*}
%%----------------------------------------------------
    where $\lin{\vec{w}_i}, \lin{\vec{w}_j}$ are the projections of $\lin{\vec{v}_i}, \lin{\vec{v}_j}$ respectively onto $W_{ij}^{\perp} \subset W$.

%%----------------------------------------------------

    Using this, and the fact that $c * \llbracket \lin{\vec{u}_1}, \lin{\vec{u}_{\bar{1}}}, \dots, \lin{\vec{u}_{n-g-1}}, \lin{\vec{u}_{\overline{n-g-1}}} \rrbracket ^\w$ is a chain in $\tC_{\bullet}(\mc{I}^{g-1}(V))$ for any chain $c \in \tC_{\bullet}(\mc{I}^{g-1}(V))$, we get that:
%%----------------------------------------------------
    \begin{align*}
      &  \del(\llbracket \lin{\vec{v}_0}, \lin{\vec{v}_1}, \dots, \lin{\vec{v}_{2g+1}} \rrbracket * \llbracket \lin{\vec{u}_1}, \lin{\vec{u}_{\bar{1}}}, \dots, \lin{\vec{u}_{n-g-1}}, \lin{\vec{u}_{\overline{n-g-1}}} \rrbracket ^\w) \\
      = & \del \left( \llbracket \lin{\vec{v}_0}, \lin{\vec{v}_1}, \dots, \lin{\vec{v}_{2g+1}} \rrbracket \right) * \llbracket \lin{\vec{u}_1}, \lin{\vec{u}_{\bar{1}}}, \dots, \lin{\vec{u}_{n-g-1}}, \lin{\vec{u}_{\overline{n-g-1}}} \rrbracket ^\w \\
       \simeq & \sum_{i \sim j} (-1)^{i+j} \llb \lin{\vec{v}_0}, \dots, \lin{\widehat{\vec{v}}_i}, \lin{\vec{v}_{i+1}}, \dots, \lin{\vec{v}_{j-1}}, \lin{\widehat{\vec{v}}_j}, \dots, \lin{\vec{v}_{2g+1}} \rrb * \llb \lin{\vec{w}_i}, \lin{\vec{w}_j} \rrb^\w * \llbracket \lin{\vec{u}_1}, \lin{\vec{u}_{\bar{1}}}, \dots, \lin{\vec{u}_{n-g-1}}, \lin{\vec{u}_{\overline{n-g-1}}} \rrbracket ^\w \\
           \simeq &\sum_{i \sim j} (-1)^{i+j} \llb \lin{\vec{v}_0}, \dots, \lin{\widehat{\vec{v}}_i}, \lin{\vec{v}_{i+1}}, \dots, \lin{\vec{v}_{j-1}}, \lin{\widehat{\vec{v}}_j}, \dots, \lin{\vec{v}_{2g+1}} \rrb * \llb \lin{\vec{w}_i}, \lin{\vec{w}_j}, \lin{\vec{u}_1}, \lin{\vec{u}_{\bar{1}}}, \dots, \lin{\vec{u}_{n-g-1}}, \lin{\vec{u}_{\overline{n-g-1}}} \rrbracket ^\w
    \end{align*}
%%----------------------------------------------------
    This almost completes the proof -- the only problem is that for a pair of indices $i<j$, the sign of the shuffle taking $\lin{\vec{v}_i}$ and $\lin{\vec{v}_j}$ to positions 0 and 1 is $(-1)^{i+j-1}$, not $(-1)^{i+j}$.
    But note that this is not really an issue, because if we have an exact sequence of $\mc{R}$-modules
%%----------------------------------------------------
    \begin{align*}
        \dots \xrightarrow{\del_{k+1}} M_k \xrightarrow{\del_k} M_{k-1} \xrightarrow{\del_{k-1}} \dots 
    \end{align*}
%%----------------------------------------------------
    then replacing the $\del_k$ with $-\del_k$ does not change exactness.
\end{proof}
%%----------------------------------------------------

%%----------------------------------------------------
\begin{remark}
    We chose to state the differential in Proposition \ref{Eq1Diff} using the sign $(-1)^{i+j-1}$ instead of $(-1)^{i+j}$ to make results more consistent to state in Section \ref{difffinal}.
\end{remark}
%%----------------------------------------------------

%%--------------------------------------------------------------
\subsection{Preliminaries on the differential}
\label{diffprelim}
%%--------------------------------------------------------------
In this section, we deduce some properties of the differential on the symplectic Sharbly resolution that will help us obtain a full description of it in Section \ref{difffinal}.

We remark that we will be following the anticommutativity sign convention for chain maps between chain complexes, i.e. when we say that $g_{\bullet} : (\tC^1_{\bullet}, \del^1_{\bullet}) \to (\tC_{\bullet}^2, \del^2_{\bullet})$ is a chain map, we mean that the map
%%----------------------------------------------------
\begin{align*}
 \del^2_k \circ g_k + g_{k-1} \circ \del^1_k: \tC^1_k \to \tC^2_{k-1}
\end{align*}
%%----------------------------------------------------
is 0 for all $k \geq 1$.
In this convention, the differential on the tensor product of these chain complexes is simply the sum $\del^1_{\bullet}+\del^2_{\bullet}$, without having to take signs into consideration. 

For $1 \leq a \leq n-1$, we will use $F_{a, \bullet}$ to denote the tensor product of the (symplectic) Sharbly resolutions for $\bigoplus_{g(W)=a} \St(W)_\Q \otimes \St^{\w}(W^{\perp})_\Q$. Thus we have the following diagram where the bottom row and all the columns are exact:

%%----------------------------------------------------
  \begin{center}
\begin{tikzcd}
& \vdots \arrow[d] & \vdots \arrow[d] & & \vdots \arrow[d] & \\
& F_{n,1} \arrow[d] & F_{n-1,1} \arrow[d] &  & F_{1,1} \arrow[d] & \\
 & F_{n,0} \arrow[d] & F_{n-1,0} \arrow[d] &  & F_{1,0} \arrow[d] & \\
 0 \arrow[r] & \St(V)_\Q \arrow[r] & \bigoplus\limits_{g(W) = n-1} \St(W)_\Q \otimes \St^{\w}(W^{\perp})_\Q \arrow[r] & \dots \arrow[r] & \bigoplus\limits_{g(W) = 1} \St(W)_\Q \otimes \St^{\w}(W^{\perp})_\Q \arrow[r] & \St^{\w}(V)_\Q \to 0 
\end{tikzcd}
\end{center}
%%----------------------------------------------------
Let $d^i_{0,-1} : F_{i,j} \to F_{i, j-1}$ denote the differential on $F_{i, \bullet}$. We have the following key proposition.

%%----------------------------------------------------
\begin{prop}\label{statingdr}
  Let ${R}$ be a PID with finitely many units and $K$ its field of fractions. Suppose $M \cong {R}^{2n}$ is a symplectic ${R}$-module and let $V = M \otimes_R K$. For $1 \leq a \leq n-1$, let $F_{a, \bullet}$ denote the resolution for $\bigoplus_{g(W)=a} \St(W)_\Q \otimes \St^{\w}(W^{\perp})_\Q$, where the sum is taken over symplectic subspaces of $V$, and let $d^i_{0,-1} : F_{i,j} \to F_{i, j-1}$ denote the differential on $F_{i, \bullet}$. Then, the degree $d$ term in the symplectic Sharbly resolution for $\St^\w(V)_\Q$ is $\bigoplus_{i+j=d+1} F_{i,j}$. For every $i$, and for $0 < r \leq i-1$, there exist maps
%%----------------------------------------------------
    \begin{align*}
        d^i_{-r,r-1}: F_{i,j} \to F_{i-r, j+r-1}
    \end{align*}
%%----------------------------------------------------
    such that

    \begin{itemize}
        \item For every $i$, $d^i_{-1, 0}: F_{i, 0} \to F_{i-1, 0}$ is a lift of the map $\bigoplus_{g(W)=i} \St(W)_\Q \otimes \St^\w(W^\perp)_\Q \to \bigoplus_{g(W)=i-1} \St(W)_\Q \otimes \St^\w(W^\perp)_\Q$ in Eq \eqref{Steinbergs}.

    \item We have
%%----------------------------------------------------
    \begin{equation}
    \label{diffcondition}
        \sum_{r+s = t}d^{i-r}_{-s, s-1} \circ d^i_{-r, r-1} = 0
    \end{equation}
%%----------------------------------------------------
    for every fixed $t$.
    \end{itemize}
    Given a choice of such maps, the sum of these maps gives a choice of differential for the synplectic Sharbly resolution of $\St^\w(V)_\Q$.
\end{prop}
%%----------------------------------------------------

The purpose of this result is that in Section \ref{difffinal}, we will explicitly construct maps $d^i_{-r, r-1}$ that satisfy the conditions stated in Proposition \ref{statingdr}, and so the assertion of Proposition \ref{statingdr} will allow us to build the differential of the symplectic Sharbly resolution using these $d^i_{-r, r-1}$. The following diagram demonstrates what the maps $d^i_{-r,r-1}$ look like when $V$ has genus 3:

%%----------------------------------------------------
\begin{tikzpicture}[
    node distance=2.8cm and 2.8cm,
    %every node/.style={font=\large},
    purplearrow/.style={->,very thick,draw=magenta},
    orangearrow/.style={->,very thick,draw=orange},
    greenarrow/.style={->,very thick,draw=teal}
]

% Nodes
\node (F32) at (0,4) {$F_{3,2}$};
\node (F22) at (3.5,4) {$F_{2,2}$};
\node (F12) at (8,4) {$F_{1,2}$};

\node (F31) at (0,2) {$F_{3,1}$};
\node (F21) at (3.5,2) {$F_{2,1}$};
\node (F11) at (8,2) {$F_{1,1}$};

\node (F30) at (0,0) {$F_{3,0}$};
\node (F20) at (3.5,0) {$F_{2,0}$};
\node (F10) at (8,0) {$F_{1,0}$};

%% Top dots
\node at (0,5.1) {$\vdots$};
\node at (3.5,5.1) {$\vdots$};
\node at (8,5.1) {$\vdots$};

%% Vertical arrows
\foreach \A/\B/\lab in {
F32/F31/{d^{3}_{0,-1}},
F31/F30/{d^{3}_{0,-1}},
%F22/F21/{d^{2}_{0,-1}},
%F21/F20/{d^{2}_{0,-1}},
F12/F11/{d^{1}_{0,-1}},
F11/F10/{d^{1}_{0,-1}}
}{
    \draw[purplearrow] (\A) -- node[right,magenta] {$\lab$} (\B);
}

\draw[purplearrow] (F22) -- (F21);
\draw[purplearrow] (F21) -- (F20);

%% Incoming arrows from dots
\draw[purplearrow] (0,4.8) -- (F32);
\draw[purplearrow] (3.5,4.8) -- (F22);
\draw[purplearrow] (8,4.8) -- (F12);

%% Horizontal arrows
\draw[greenarrow]
(F32) -- node[above,teal] {$d^{3}_{-1,0}$} (F22);

\draw[greenarrow]
(F22) -- node[above,teal] {$d^{2}_{-1,0}$} (F12);

\draw[greenarrow]
(F31) -- node[above,teal] {$d^{3}_{-1,0}$} (F21);

\draw[greenarrow]
(F21) -- node[above,teal] {$d^{2}_{-1,0}$} (F11);

\draw[greenarrow]
(F30) -- node[above,teal] {$d^{3}_{-1,0}$} (F20);

\draw[greenarrow]
(F20) -- node[above,teal] {$d^{2}_{-1,0}$} (F10);

%% Diagonal arrows
\draw[orangearrow]
(F31) -- node[above,sloped,orange] {$d^{3}_{-2,1}$} (F12);

\draw[orangearrow]
(F30) -- node[above,sloped,orange] {$d^{3}_{-2,1}$} (F11);

%% Bottom complex
\node (A0) at (-2,-2)
{$0$};

\node (A1) at (0,-2)
{$\operatorname{St}(V)$};

%\node at (0,-2.75) {$g(W)=2$};

\node (A2) at (3.5,-2)
{$\bigoplus\limits_{g(W)=2}
\operatorname{St}(W)\otimes
\operatorname{St}^{\omega}(W^{\perp})$};

\node (A3) at (8,-2)
{$\bigoplus\limits_{g(W)=1}
\operatorname{St}(W)\otimes
\operatorname{St}^{\omega}(W^{\perp})$};

%\node at (8,-2.75) {$g(W)=2$};

\node (A4) at (12,-2)
{$\operatorname{St}^{\omega}(V)
\longrightarrow 0$};

\draw[->, thick] (A0) -- (A1);
\draw[->, thick] (A1) -- (A2);
\draw[->, thick] (A2) -- (A3);
\draw[->, thick] (A3) -- (A4);

%% Down arrows from bottom row

\draw[purplearrow] (F30) -- (A1);
\draw[purplearrow] (F20) -- (A2);
\draw[purplearrow] (F10) -- (A3);

\end{tikzpicture}
%%----------------------------------------------------

\vspace{0.5cm}

We will prove Proposition \ref{statingdr} inductively. Recall from Lemma \ref{gluestep} that the way we build our resolution is by using the $F_{a, \bullet}$ to successively get rid of the terms $\bigoplus_{g(W) = a} \St(W)_\Q \otimes \St^{\w}(W^{\perp})_\Q$. Suppose the resolution we get after getting rid of all the terms where $g(W) > a$ is:
%%----------------------------------------------------
\begin{align}\label{F'defn}
    \dots \to F'_{a,1} \to F'_{a,0} \to \bigoplus_{g(W) = a} \St(W)_\Q \otimes \St^{\w}(W^{\perp})_\Q \to \dots \to \bigoplus_{g(W) = 1} \St(W)_\Q \otimes \St^{\w}(W^{\perp})_\Q \to \St^\w(V)_\Q \to 0,
\end{align}
%%----------------------------------------------------
so that we have the following diagram:

%%----------------------------------------------------
 \begin{center}
\begin{tikzcd}
 \vdots \arrow[d] & \vdots \arrow[d] &  & \\
  F'_{a,1} \arrow[d] & F_{a-1,1} \arrow[d] &   & \\
  F'_{a,0} \arrow[d] & F_{a-1,0} \arrow[d] &  & \\
  \bigoplus\limits_{g(W) = a} \St(W)_\Q \otimes \St^{\w}(W^{\perp})_\Q \arrow[r] &\bigoplus\limits_{g(W) = a-1} \St(W)_\Q \otimes \St^{\w}(W^{\perp})_\Q \arrow[r] & \dots  \arrow[r] & \St^{\w}(V)_\Q \to 0 
\end{tikzcd}
\end{center}
%%----------------------------------------------------

By the results of Section \ref{glueprocess}, we have $F'_{a,k} = \bigoplus_{a \leq i \leq \min(n, a+k)} F_{i, a+k-i}$. 

We will prove Proposition \ref{statingdr} by proving an analogue of it for the differential $\del'_a: F'_{a, \bullet} \to F'_{a, \bullet-1}$. Note that the differential on the symplectic Sharbly resolution is precisely the same as $\del'_1: F'_{1, \bullet} \to F'_{1, \bullet-1}$. So Proposition \ref{statingdr} is in fact a direct consequence of the following lemma.

%%----------------------------------------------------
\begin{lemma}
\label{definingdr}
   Let ${R}$ be a PID with finitely many units and $K$ its field of fractions. Suppose $M \cong {R}^{2n}$ is a symplectic ${R}$-module and let $V = M \otimes_R K$. For $1 \leq a \leq n$, let $F_{a, \bullet}$ denote the resolution for $\bigoplus_{g(W)=a} \St(W)_\Q \otimes \St^{\w}(W^{\perp})_\Q$, where the sum is taken over symplectic subspaces of $V$, and let $d^i_{0,-1} : F_{i,j} \to F_{i, j-1}$ denote the differential on $F_{i, \bullet}$. Define $F'_{a,j}$ as in Equation \eqref{F'defn} above and let $\del'_a$ denote the differential $F'_{a, j} \to F'_{a,j-1}$. Note that $F'_{a,k} = \bigoplus_{a \leq i \leq \min(n, a+k)} F_{i, a+k-i}$. For $i > a$, there exist maps of the form
%%----------------------------------------------------
    \begin{align*}
        d^i_{-r,r-1}: F_{i,j} \to F_{i-r, j+r-1}
    \end{align*} 
%%----------------------------------------------------
    where 
%%----------------------------------------------------
    \begin{align*}
        0 < r \leq i-a
    \end{align*}
%%----------------------------------------------------
    such that
    \begin{itemize}
        \item For every $i>a$, $d^i_{-1, 0}: F_{i, 0} \to F_{i-1, 0}$ is a lift of the map $\bigoplus_{g(W)=i} \St(W)_\Q \otimes \St^\w(W^\perp)_\Q \to \bigoplus_{g(W)=i-1} \St(W)_\Q \otimes \St^\w(W^\perp)_\Q$ in Eq \eqref{Steinbergs}.
    
%%----------------------------------------------------
    \item We have 
    \begin{align*}
        \sum_{r+s = t}d^{i-r}_{-s, s-1} \circ d^i_{-r, r-1} = 0
    \end{align*}
%%----------------------------------------------------
    for every fixed $t$.
    \end{itemize}
    Given such maps, a choice of differential for $\del'_a$ is as follows: on a summand $F_{i,j}$ of $F'_{a,k}$ with $i \geq a$, the differential $\del'_a$ is the sum of these maps $\sum_{0 \leq r \leq i-a} d^i_{-r,r-1}$.
\end{lemma}
%%----------------------------------------------------
\begin{proof}
    Set $a = n-b$ and induct on $b$, with the base case being $b=0$.
    When $b=0$, we have 
%%----------------------------------------------------
    \begin{align*}
        F'_{n, \bullet} = F_{n,\bullet}
    \end{align*}
%%----------------------------------------------------
 and 
 %%----------------------------------------------------
    \begin{align*}
        \del'_n = d^n_{0,-1},
    \end{align*} 
%%----------------------------------------------------
    and the claim is trivially true.
    So now assume the result is true for some $a = n-b \leq n$. We shall prove the assertion for $a-1 = n-(b+1)$.

    Note that 
%%----------------------------------------------------
    \begin{align*}
        F'_{a-1,k} = F_{a-1,k} \oplus F'_{a,k-1} \text{ for } k \geq 1
    \end{align*}
%%----------------------------------------------------
    and 
%%----------------------------------------------------
    \begin{align*}
        F'_{a-1,0} = F_{a-1,0}
    \end{align*}
%%----------------------------------------------------

Now, as in the proof of Lemma \ref{gluestep}, there exists a chain map $\delta_{a} : F'_{a, k} \to F_{a-1, k}$ lifting the map $\bigoplus_{g(W)=a} \St(W)_\Q \otimes \St^\w(W^\perp)_\Q \to \bigoplus_{g(W)=a-1} \St(W)_\Q \otimes \St^\w(W^\perp)_\Q$, such that
%%----------------------------------------------------
\begin{equation}
\label{delsumzero}
    d^{a-1}_{0,-1} \circ \delta_a + \delta_{a} \circ \del'_a = 0
\end{equation}
%%----------------------------------------------------

Recall that 
%%----------------------------------------------------
\begin{align*}
    F'_{a,k} = \bigoplus_{a \leq i \leq \min(n, a+k)} F_{i, a+k-i}
\end{align*}
%%----------------------------------------------------
Thus the existence of $\delta_a : F'_{a,k} \to F_{a-1,k}$ implies that for all $a \leq i \leq n$ we have maps

%%----------------------------------------------------
\begin{align*}
    d^i_{-(i-a+1), i-a} : F_{i, j} \to F_{a-1, j+i-a} 
\end{align*}
%%----------------------------------------------------
such that
%%----------------------------------------------------
\begin{align*}
    \delta_a|_{F'_{a,k}} = \sum_{a \leq i \leq \min(n, a+k)} d^i_{-(i-a+1), i-a}
\end{align*}
%%----------------------------------------------------
%
The fact that $\delta_a: F'_{a, \bullet} \to F_{a-1, \bullet}$ is a chain lift of the map from $\bigoplus_{g(W)=a} \St(W)_\Q \otimes \St^\w(W^\perp)_\Q$ to  $\bigoplus_{g(W)=a-1} \St(W)_\Q \otimes \St^\w(W^\perp)_\Q$ implies that the map
\begin{align*}
    d^a_{-1,0} : F_{a,0} \to F_{a-1,0}
\end{align*}
is a lift of the map $\bigoplus_{g(W)=a} \St(W)_\Q \otimes \St^\w(W^\perp)_\Q \to \bigoplus_{g(W)=a-1} \St(W)_\Q \otimes \St^\w(W^\perp)_\Q$.

%%----------------------------------------------------
Now Equation \eqref{delsumzero} implies that
%%----------------------------------------------------
\begin{align*}
    0 = d^{a-1}_{0,-1} \circ d^i_{-(i-a+1), i-a} + \left( \sum_{r+s = i-a+1, s>0} d^{i-r}_{-s,s-1} \circ d^i_{-r,r-1} \right) : F_{i,j} \to F_{a-1, j+i-a-1}  \hspace{0.5cm} (i \geq a)
\end{align*}
%%----------------------------------------------------
This is equivalent to saying that
%%----------------------------------------------------
\begin{equation}
\label{delsumzeromain}
     \sum_{r+s = i-a+1} d^{i-r}_{-s,s-1} \circ d^i_{-r,r-1} = 0 \hspace{0.5cm} (i \geq a)
\end{equation}
%%----------------------------------------------------

The differential $\del'_{a-1}$ is given by:
%%----------------------------------------------------
\begin{align*}
    \del'_{a-1}|_{F_{a-1,k}} & =  d^{a-1}_{0,-1} : F_{a-1,k} \to F_{a-1, k-1} \\
     \del'_{a-1}|_{F_{a,k}}  & = \del'_a + \delta_a : F_{a,k} \to F_{a,k-1} \oplus F_{a-1,k} \\
\end{align*}
%%----------------------------------------------------

It follows that on a term $F_{i,j}$ with $i \geq a-1$, the differential $\del'_{a-1}$ is the sum of the maps
%%----------------------------------------------------
\begin{align*}
    \sum_{0 \leq r \leq i-a+1} d^i_{-r,r-1}
\end{align*}
%%----------------------------------------------------

It remains to show that
%%----------------------------------------------------
\begin{align*}
      \sum_{r+s = t}d^{i-r}_{-s, s-1} \circ d^i_{-r, r-1} = 0
\end{align*}
%%----------------------------------------------------
    for every fixed $t$, and $i \geq a-1$.

    When $i \geq a$ and $t \leq i-a$, this holds by the inductive hypothesis on $F'_{a,\bullet}$.
    When $i = a-1$, this is just the assertion that $d^{a-1}_{0,-1} \circ d^{a-1}_{0,-1} = 0$, which is true since $d^{a-1}_{0,-1}$ is the differential on $F_{a-1, \bullet}$.
    And finally, for $i \geq a$ and $t = i-a+1$, this is precisely Equation \eqref{delsumzeromain}.
\end{proof}
%%----------------------------------------------------

%%--------------------------------------------------------------
\subsection{The Differential}
\label{difffinal}
%%--------------------------------------------------------------
We are now ready to state the differential on our resolution.

Keeping track of signs will be important in proving our result about the differential, so we first recall the convention for the differential on the tensor product of two chain complexes.
Suppose that $(\tC^1_\bullet, \del_{\bullet}^1)$ and $(\tC^2_\bullet, \del_{\bullet}^2)$ are two projective resolutions of $M_1$, $M_2$ respectively over a ring $\mc{R}$. Then recall that the tensor product of these resolutions gives us a projective $\mc{R}$-resolution of $M_1 \otimes M_2$, with the differential given by:
%%----------------------------------------------------
\begin{align}\label{tensordiff}
    \del^1_p + (-1)^p\del^2_q: \tC_{p,q} \to \tC_{p-1,q} \oplus \tC_{p,q-1}
\end{align}
%%----------------------------------------------------

We shall in particular make use of the case when $M_1, M_2$ are equal to $\St(W)$ for some symplectic vector space $W$, with the $(\tC^i_\bullet, \del_{\bullet}^i)$ being the Sharbly resolution.
Note that the degree of a Sharbly term $[\lin{\vec{v}_0}, \dots, \lin{\vec{v}_{l}}]$ in the resolution for $\St(W)$ is 
%%----------------------------------------------------
\begin{align*}
    (l+1)-\dim(W)
\end{align*}
%%----------------------------------------------------
Since $W$ is symplectic, $\dim(W)$ is even, and so the degree has the same parity as $l+1$. Thus we don't need to know the dimension of $W$ to plug into \eqref{tensordiff}.

Though we have to take signs into account when tensoring two Sharbly resolutions, we will still use the anticommutativity sign convention for chain maps between resolutions, as mentioned in the beginning of Section \ref{diffprelim}. It will be important to keep track of which convention applies in which case as we proceed through our proof.

Now, recall from Proposition \ref{sharblyterms} that the degree $d$ term in the symplectic Sharbly resolution is a direct sum over terms of the form
%%----------------------------------------------------
\begin{align*}
    \Sh_{k_1}(W_1)_\Q \otimes \dots \otimes \Sh_{k_m}(W_m)_\Q
\end{align*}
%%----------------------------------------------------
where 
%%----------------------------------------------------
\begin{align*}
    W_1 \oplus \dots \oplus W_m = V
\end{align*}
%%----------------------------------------------------
is a decomposition of $V$ into mutually perpendicular symplectic subspaces, and
%%----------------------------------------------------
\begin{align*}
    d = k_1 + \dots + k_m + n-m
\end{align*}
%%----------------------------------------------------

Our differential will reduce the degree either by reducing one of the $k_i$'s by 1, or by increasing $m$ by 1.
For brevity, in this section we shall denote a vector using lowercase letters (for eg, $v$ instead of $\vec{v}$), and write simply $[{{v}_0}, \dots, {{v}_{l}}]$ instead of $[\lin{\vec{v}_0}, \dots, \lin{\vec{v}_{l}}]$ to represent a Sharbly term. We will also use the convention that for a permutation $\sigma$, $\sgn(\sigma) = 0$ if $\sigma$ is an even permutation and $\sgn(\sigma) = 1$ if $\sigma$ is odd (as opposed to $\sgn(\sigma) = \pm 1$) This is to ensure that $(-1)^{\sgn(\sigma)} = 1$ if and only if $\sigma$ is an even permutation.
%%----------------------------------------------------
\begin{prop} 
\label{differential}
Let ${R}$ be a PID with finitely many units and $K$ its field of fractions. Suppose $M \cong {R}^{2n}$ is a symplectic ${R}$-module and let $V = M \otimes_R K$. Suppose that
%%----------------------------------------------------
\begin{align*}
    [{v}^1_{0}, {v}^1_1, \dots, {v}^1_{d_1}]\otimes \dots \otimes [{v}^m_0, {v}^m_1, \dots, {v}^m_{d_m}] \in \Sh_{k_1}(W_1)_\Q \otimes \dots \otimes \Sh_{k_m}(W_m)_\Q
\end{align*}
%%----------------------------------------------------
is a simple tensor in the symplectic Sharbly resolution for $\St^\w(V)_\Q$. Then the image of the differential on this term is the sum of the following two types of terms:
\begin{itemize}
    \item \emph{Omit-terms:} Suppose ${v}^j_i$ is such that the span of ${v}^j_0, \dots, \widehat{{v}^j_i}, \dots, {v}^j_{d_j}$ is the same as the span of $v^j_0, \dots, v^j_{d_j}$. Then the \emph{omit-term} corresponding to $v^j_i$ is
%%----------------------------------------------------
    \begin{align*}
        (-1)^{d_1 -1 + \dots + d_{j-1}-1}(-1)^i [v^1_{0}, v^1_1, \dots, v^1_{d_1}]\otimes \dots \otimes [v^j_0, \dots, \widehat{v^j_i}, \dots, v^j_{d_j}] \otimes \dots \otimes [v^m_0, v^m_1, \dots, v^m_{d_m}]
    \end{align*}
%%----------------------------------------------------

    \item \emph{Split-terms:} Suppose $v^j_{i_0}, v^j_{i_1}, \dots, v^j_{i_p}$ are vectors such that $v^j_0, \dots, \widehat{v^j_{i_0}}, v^j_{i_0+1,} \dots, v^j_{i_p-1,} \widehat{v^j_{i_p}}, \dots, v^j_{d_j}$ spans a proper nonzero symplectic subspace $X_j$ of $W_j$. 
    Let 
%%----------------------------------------------------
    \begin{align*}
        \lin{w^j_0}, \dots, \lin{w^j_p}
    \end{align*}
%%----------------------------------------------------
    be the projections of
%%----------------------------------------------------
    \begin{align*}
        \lin{v^j_{i_0}}, \dots, \lin{v^j_{i_p}}
    \end{align*}
%%----------------------------------------------------
    respectively onto $X_j^\perp \subset W_j$.
    Then the \emph{split-term} corresponding to these vectors is
%%----------------------------------------------------
    \begin{align*}
        (-1)^\nu [v^1_{0}, \dots, v^1_{d_1}]\otimes \dots \otimes [v^j_0, \dots, \widehat{v^j_{i_0}}, v^j_{i_0+1,} \dots, v^j_{i_p-1,} \widehat{v^j_{i_p}}, \dots, v^j_{d_j}] \otimes [w^j_0, \dots w^j_p] \otimes \dots\otimes [v^m_0, \dots, v^m_{d_m}]
    \end{align*}
%%----------------------------------------------------
    with
%%----------------------------------------------------
    \begin{align}\label{nudefn}
        \nu = (d_1-1 + d_2-1 + \dots + d_{j-1}-1) + pd_j + \sgn(\sigma) + (d_{j+1} + d_{j+2} + \dots + d_m)
    \end{align}
%%----------------------------------------------------
    where $\sigma$ is the shuffle permutation that shuffles $v^j_{i_0}, \dots, v^j_{i_p}$ to the first $p+1$ positions in $[v^j_0, \dots, v^j_{d_j}]$.
\end{itemize}
\end{prop}
%%----------------------------------------------------

In the case when the genus $n=1$, we are forced to have $m=1$, and we can only have omit-terms, in which case we recover the usual Sharbly resolution for $\SL_2$.

We shall prove the general case by inducting on $n$. We will break the proof up into two lemmas.

%%----------------------------------------------------
\begin{lemma}
\label{d_0,-1}
   Let ${R}$ be a PID with finitely many units and $K$ its field of fractions. Suppose $M \cong {R}^{2n}$ is a symplectic ${R}$-module and let $V = M \otimes K$. Suppose Proposition \ref{differential} holds for all genus $<n$ symplectic spaces.
   Recall from Section \ref{diffprelim} that $F_{a, \bullet}$ denotes the tensor products of the (symplectic) Sharbly resolution of $\bigoplus_{g(W)=a} \St(W)_\Q \otimes \St^\w(W^\perp)_\Q$, and $d^a_{0,-1}$ denotes the differential on it. 
   Then the differential $d^a_{0,-1}$ on a term $[v^1_0, \dots, v^1_{d_1}]\otimes \dots \otimes [v^m_0, \dots, v^m_{d_m}]$ is the sum of all the omit-terms on it, plus all the split-terms where $j>1$, i.e. the splitting is \emph{not} taking place on $[v^1_0, \dots, v^1_{d_1}]$.
\end{lemma}
%%----------------------------------------------------
\begin{proof}
    Recall that $F_{a, \bullet}$ is obtained by tensoring the (symplectic) Sharbly resolutions for $\St(W)_\Q$ and $\St^\w(W^\perp)_\Q$ and adding them across all $W$ with $g(W)=a$. Let $\del_1$ denote the differential on the Sharbly resolution for $\St(W)_\Q$, and $\del_2$ the differential on the symplectic Sharbly resolution of $\St^\w(W^\perp)_\Q$. Then the differential on a term 
%%----------------------------------------------------
    \begin{align*}
        [v^1_0, \dots, v^1_{d_1}]\otimes \dots \otimes [v^m_0, \dots, v^m_{d_m}]
    \end{align*}
%%----------------------------------------------------
    is
%%----------------------------------------------------
    \begin{align*}
        \del_1 + (-1)^{d_1-1}\del_2
    \end{align*}
%%----------------------------------------------------
Now note that applying $\del_1$ gives us the omit-terms where the omitted vector is in $[v^1_0, \dots, v^1_{d_1}]$, and applying $(-1)^{d_1-1}\del_2$ gives us all the omit-terms and the split-terms where the omitted or split vectors are in one of the $[v^j_0, \dots, v^j_{d_j}]$ for $j>1$.
\end{proof}
%%----------------------------------------------------

We now describe the maps $d^a_{-r,r-1}$ described in Lemma \ref{definingdr}, for $r>0$.

%%----------------------------------------------------
\begin{lemma}
\label{d_r,r-1}
Let ${R}$ be a PID with finitely many units and $K$ its field of fractions. Suppose $M \cong {R}^{2n}$ is a symplectic ${R}$-module and let $V = M \otimes_R K$. Recall from Section \ref{diffprelim} that $F_{a, \bullet}$ denotes the resolution of $\bigoplus_{g(W)=a} \St(W)_\Q \otimes \St^\w(W^\perp)_\Q$, and $d^a_{0,-1}$ denotes the differential on it. Suppose $[v^1_0, \dots, v^1_{d_1}]\otimes \dots \otimes [v^m_0, \dots, v^m_{d_m}] \in \Sh_{k_1}(W_1)_\Q \otimes \dots \otimes \Sh_{k_m}(W_m)_\Q$ is a term in some $F_{a, \bullet}$. Then for $r>0$, the map
%%----------------------------------------------------
\begin{align*}
    d^a_{-r,r-1} : F_{a,\bullet} \to F_{a-r, \bullet +r-1}
\end{align*}
%%----------------------------------------------------
whose existence is given by Lemma \ref{definingdr}, maps $[v^1_0, \dots, v^1_{d_1}]\otimes \dots \otimes [v^m_0, \dots, v^m_{d_m}]$ to the sum of all the split-terms
%%----------------------------------------------------
\begin{align*}
    (-1)^\nu [v^1_0, \dots, \widehat{v^1_{i_0}}, v^1_{i_0+1,} \dots, v^1_{i_p-1}, \widehat{v^1_{i_p}}, \dots, v^1_{d_1}] \otimes [w^1_0, \dots w^1_p] \otimes [v^2_0, \dots, v^2_{d_2}] \otimes \dots \otimes [v^m_0, \dots, v^m_{d_m}]
\end{align*}
%%----------------------------------------------------
where $\nu$ is given by Equation \eqref{nudefn}, and
%%----------------------------------------------------
\begin{align*}
    v^1_0, \dots, \widehat{v^1_{i_0}}, v^1_{i_0+1}, \dots, v^1_{i_p-1}, \widehat{v^1_{i_p}}, \dots, v^1_{d_1}
\end{align*}
%%----------------------------------------------------
span a genus $a-r$ symplectic subspace of the genus $a$ symplectic space $W_1$.
\end{lemma}
%%----------------------------------------------------
\begin{proof}
    We need to check that on setting $r=1$ that the given description of the map $d^a_{-1,0}: F_{a,0} \to F_{a-1,0}$ extends the differential described on Equation \eqref{Steinbergs} in Proposition \ref{Eq1Diff}, and that Equation \eqref{diffcondition} is satisfied. It is straightforward to verify the former, so let us turn to showing that Equation \eqref{diffcondition} holds, i.e. that
%%----------------------------------------------------
    \begin{align*}
        \sum_{r+s = t}d^{a-r}_{-s, s-1} \circ d^a_{-r, r-1} = 0
    \end{align*}
%%----------------------------------------------------
    for every fixed $t$.
    When $r=s=t=0$, this equation becomes
%%----------------------------------------------------
    \begin{align*}
        d^a_{0,-1} \circ d^a_{0,-1} = 0
    \end{align*}
%%----------------------------------------------------
    which holds since Lemma \ref{d_0,-1} says that $d^a_{0,-1}$ is a differential. So it remains to consider the case when $t > 0$.

    We shall show that all the terms appearing in 
%%----------------------------------------------------
    \begin{align*}
        d^{a-t}_{0,-1} \circ d^a_{-t,t-1}
    \end{align*}
%%----------------------------------------------------
    cancel out with those appearing in 
%%----------------------------------------------------
    \begin{align*}
        \sum_{r+s = t, s>0}d^{a-r}_{-s, s-1} \circ d^a_{-r, r-1}
    \end{align*}
%%----------------------------------------------------

    On applying $d^a_{-t,t-1}$ to $[v^1_0, \dots, v^1_{d_1}]\otimes \dots \otimes [v^m_0, \dots, v^m_{d_m}]$, we get the sum of all split-terms of the form
%%----------------------------------------------------
    \begin{align*}
    (-1)^\nu [v^1_0, \dots, \widehat{v^1_{i_0}}, v^1_{i_0+1}, \dots, v^1_{i_p-1}, \widehat{v^1_{i_p}}, \dots, v^1_{d_1}] \otimes [w^1_0, \dots w^1_p] \otimes [v^2_0, \dots, v^2_{d_2}] \otimes \dots \otimes [v^m_0, \dots, v^m_{d_m}]
\end{align*}
%%----------------------------------------------------
where 
%%----------------------------------------------------
\begin{align*}
    v^1_0, \dots, \widehat{v^1_{i_0}}, v^1_{i_0+1}, \dots, v^1_{i_p-1}, \widehat{v^1_{i_p}}, \dots, v^1_{d_1}
\end{align*}
%%----------------------------------------------------
span a genus $a-r$ symplectic subspace of the genus $a$ symplectic space $W_1$.

If we then apply $d^{a-t}_{0,-1}$ to one of these terms, we either end up omitting a vector from one of 
%%----------------------------------------------------
\begin{align*}
    [v^1_0, \dots, \widehat{v^1_{i_0}}, v^1_{i_0+1}, \dots, v^1_{i_p-1}, \widehat{v^1_{i_p}}, \dots, v^1_{d_1}] ,  [w^1_0, \dots w^1_p] , [v^2_0, \dots, v^2_{d_2}] , \dots , [v^m_0, \dots, v^m_{d_m}]
\end{align*}
%%----------------------------------------------------
or end up splitting one of
%%----------------------------------------------------
\begin{align*}
    [w^1_0, \dots w^1_p] , [v^2_0, \dots, v^2_{d_2}] , \dots , [v^m_0, \dots, v^m_{d_m}]
\end{align*}
%%----------------------------------------------------

We thus break the proof up into four cases:

\begin{enumerate}
    \item Case 1: We omit a vector from one of
%%----------------------------------------------------
    \begin{align*}
        [v^2_0, \dots, v^2_{d_2}] , \dots , [v^m_0, \dots, v^m_{d_m}]
    \end{align*}
%%----------------------------------------------------

Suppose the omitted vector is $v^j_i$.
Then the sign accompanying
%%----------------------------------------------------
\begin{IEEEeqnarray*}{l}
    \hspace*{-3em} [v^1_0, \dots, \widehat{v^1_{i_0}}, v^1_{i_0+1}, \dots, \widehat{v^1_{i_p}}, \dots, v^1_{d_1}] \otimes [w^1_0, \dots w^1_p] \otimes [v^2_0, \dots, v^2_{d_2}] \otimes \dots \otimes [v^j_0, \dots, \widehat{v^j_i}, \dots, v^j_{d_j}] \otimes \dots \otimes [v^m_0, \dots, v^m_{d_m}]
\end{IEEEeqnarray*}
%%----------------------------------------------------
is $(-1)^\tau$ with
%%----------------------------------------------------
\begin{align*}
    \tau & = \nu + ((d_1-(p+1)-1) + (p-1) + (d_2-1) + \dots + (d_{j-1}-1)) + i \\
    & = pd_1 + \sgn(\sigma) + (d_2 + \dots + d_m) + ((d_1-(p+1)-1) + (p-1) + (d_2-1) + \dots + (d_{j-1}-1)) + i 
\end{align*}
%%----------------------------------------------------

Thus modulo 2,
%%----------------------------------------------------
\begin{equation}
    \label{case1sign1}
    \tau = pd_1 + \sgn(\sigma) + (d_2 + \dots + d_m) + ((d_1-1) + (d_2-1) + \dots + (d_{j-1}-1)) + i
\end{equation}
%%----------------------------------------------------

Note that this term appears exactly once when we apply
%%----------------------------------------------------
\begin{align*}
      \sum_{r+s = t, s>0}d^{a-r}_{-s, s-1} \circ d^a_{-r, r-1},
\end{align*}
%%----------------------------------------------------
namely in 

%%----------------------------------------------------
\begin{align*}
    d^a_{-t,t-1} \circ d^a_{0,-1}
\end{align*}
%%----------------------------------------------------
when we first omit $v^j_i$ by applying $d^a_{0,-1}$, and then split up $[v^1_0, \dots, v^1_{d_1}]$ by applying $d^a_{-t,t-1}$.

The accompanying sign this time is $(-1)^\rho$ with
%%----------------------------------------------------
\begin{align}
    \label{case1sign2}
    \rho = ((d_1-1)+\dots+(d_{j-1}-1)) + i + pd_1 + \sgn(\sigma) + (d_2 + \dots + d_{j-1} + (d_j-1) + d_{j+1} + \dots + d_m)
\end{align}
%%----------------------------------------------------

On comparing \eqref{case1sign1} and \eqref{case1sign2}, we see that they give us opposite signs (in \eqref{case1sign1} we have $d_j$ in the sum, whereas in \eqref{case1sign2} we have $d_j-1$).

    \item Case 2: We split one of
%%----------------------------------------------------
\begin{align*}
        [v^2_0, \dots, v^2_{d_2}] , \dots , [v^m_0, \dots, v^m_{d_m}]
    \end{align*}
%%----------------------------------------------------

Suppose we split $[v^j_0, \dots, v^j_{d_j}]$ into
%%----------------------------------------------------
\begin{align*}
    [v^j_0, \dots, \widehat{v^j_{k_0}}, v^j_{k_0+1} \dots, v^j_{k_q-1}, \widehat{v^j_{k_q}}, \dots, v^j_{d_j}] \otimes [w^j_0, \dots, w^j_q]
\end{align*}
%%----------------------------------------------------
and let $\sigma'$ be the sign of the shuffle that permutes $v^j_{k_0}, \dots, v^j_{k_q}$ to the first $q+1$ positions.
Then the accompanying sign is $(-1)^\tau$ with
%%----------------------------------------------------
\begin{IEEEeqnarray*}{l}
    \hspace*{-2em} \tau = \nu + ((d_1-(p+1)-1) + (p-1) + (d_2-1) + \dots + (d_{j-1}-1)) + qd_j + \sgn(\sigma') + (d_{j+1} + \dots + d_m) \\
     = pd_1 + \sgn(\sigma) + (d_2 + \dots + d_m) + ((d_1-(p+1)-1) + (p-1) \\
    \quad \quad + (d_2-1) + \dots + (d_{j-1}-1)) + qd_j + \sgn(\sigma') + (d_{j+1} + \dots + d_m)
\end{IEEEeqnarray*}
%%----------------------------------------------------
Thus modulo 2,
%%----------------------------------------------------
\begin{IEEEeqnarray}{l}
    \label{case2sign1}
    \hspace*{-2em} \tau = pd_1 + \sgn(\sigma) + (d_2 + \dots + d_m) + ((d_1-1) + (d_2-1) + \dots + (d_{j-1}-1)) + qd_j \nonumber \\
    \quad \quad + \sgn(\sigma') + (d_{j+1} + \dots + d_m)
\end{IEEEeqnarray}
%%----------------------------------------------------

Note that this term appears exactly once when we apply
%%----------------------------------------------------
\begin{align*}
      \sum_{r+s = t, s>0}d^{a-r}_{-s, s-1} \circ d^a_{-r, r-1},
\end{align*}
%%----------------------------------------------------
namely in 
%%----------------------------------------------------
\begin{align*}
    d^a_{-t,t-1} \circ d^a_{0,-1}
\end{align*}
%%----------------------------------------------------
when we first split $[v^j_0, \dots, v^j_{d_j}]$ by applying $d^a_{0,-1}$ and then split $[v^1_0, \dots, v^1_{d_1}]$ by applying $d^a_{-t,t-1}$.

The accompanying sign is $(-1)^\rho$ with
%%----------------------------------------------------
\begin{align*}
    \rho & = ((d_1-1) + \dots + (d_{j-1}-1)) + qd_j + \sgn(\sigma') + (d_{j+1} + \dots + d_m) + pd_1 + \sgn(\sigma) \\
    & + (d_2 + \dots + d_{j-1} + (d_j-(q+1)) + q + d_{j+1} + \dots d_m) \\
    & = ((d_1-1) + \dots + (d_{j-1}-1)) + qd_j + \sgn(\sigma') + (d_{j+1} + \dots + d_m) + pd_1 + \sgn(\sigma) \\
    & + (d_2 + \dots + d_{j-1} + (d_j-1) + d_{j+1} + \dots d_m)
\end{align*}
%%----------------------------------------------------

Thus
%%----------------------------------------------------
\begin{IEEEeqnarray}{l}
\label{case2sign2}
    \hspace*{-2em} \rho = ((d_1-1) + \dots + (d_{j-1}-1)) + qd_j + \sgn(\sigma') + (d_{j+1} + \dots + d_m) + pd_1 + \sgn(\sigma) \nonumber \\
    \quad + (d_2 + \dots  + (d_j-1)  + \dots d_m)
\end{IEEEeqnarray}
   
%%----------------------------------------------------

Comparing \eqref{case2sign1} and \eqref{case2sign2}, we again see that they give us opposite signs.

\item Case 3: We omit a vector from one of
%%----------------------------------------------------
\begin{align*}
     [v^1_0, \dots, \widehat{v^1_{i_0}}, v^1_{i_0+1}, \dots, v^1_{i_p-1}, \widehat{v^1_{i_p}}, \dots, v^1_{d_1}] ,  [w^1_0, \dots w^1_p]
\end{align*}
%%----------------------------------------------------

This case further splits into two:

\begin{itemize}
    \item Suppose we omit the vector $v^1_i$ from $[v^1_0, \dots, \widehat{v^1_{i_0}}, v^1_{i_0+1}, \dots, v^1_{i_p-1}, \widehat{v^1_{i_p}}, \dots, v^1_{d_1}]$.
    Let $q$ be the largest index such that $i_q < i$.

    Then the sign accompanying our term is $(-1)^\tau$ with
%%----------------------------------------------------
    \begin{align*}
        \tau & = \nu + (i-q-1)
    \end{align*}
%%----------------------------------------------------
    Thus
%%----------------------------------------------------
    \begin{align}
        \label{case3case1sign1}
        \tau & = pd_1 + \sgn(\sigma) + (d_2 + \dots + d_m) + i - q -1
    \end{align}
%%----------------------------------------------------

    On the other hand, this term appears exactly once when we apply
%%----------------------------------------------------
    \begin{align*}
      \sum_{r+s = t, s>0}d^{a-r}_{-s, s-1} \circ d^a_{-r, r-1},
\end{align*}
%%----------------------------------------------------
namely, when we first omit $v_i$ by applying $d^a_{0,-1}$, and then split the $v^1_{i_0}, \dots, v^1_{i_p}$ away.
Omitting $v_i$ first changes the sign of the required shuffle by $(-1)^{p-q}$

Thus the accompanying sign now is $(-1)^\rho$ with
%%----------------------------------------------------
\begin{align*}
    \rho = i + p(d_1-1) + \sgn(\sigma) + p-q + (d_2 + \dots + d_m)
\end{align*}
%%----------------------------------------------------
Thus
%%----------------------------------------------------
\begin{align}
    \label{case3case1sign2}
    \rho = i + pd_1 + \sgn(\sigma) - q + (d_2 + \dots + d_m)
\end{align}
%%----------------------------------------------------
Comparing \eqref{case3case1sign1} and \eqref{case3case1sign2}, we see that they give us opposite signs.

\item Suppose now that we omit the vector $w^1_j$ from $[w^1_0, \dots, w^1_p]$.

The accompanying sign is $(-1)^\tau$ with 
%%----------------------------------------------------
\begin{align*}
    \tau = \nu + (d_1-(p+1)-1) + j
\end{align*}
%%----------------------------------------------------

Thus modulo 2,
%%----------------------------------------------------
\begin{align}
    \label{case3case2sign1}
    \tau = pd_1 + \sgn(\sigma) + (d_2 + \dots + d_m) + d_1 - p + j
\end{align}
%%----------------------------------------------------

On the other hand, this term appears exactly once when we apply
%%----------------------------------------------------
    \begin{align*}
      \sum_{r+s = t, s>0}d^{a-r}_{-s, s-1} \circ d^a_{-r, r-1},
\end{align*}
%%----------------------------------------------------
namely, when we first omit $v_{i_j}$ by applying $d^a_{0,-1}$, and then split the $v^1_{i_0}, \dots \widehat{v^1_{i_j}},\dots, v^1_{i_p}$ away.
Omitting $v_{i_j}$ first changes the sign of the required shuffle by $i_j-j$.

Thus the accompanying sign now is $(-1)^\rho$ with
%%----------------------------------------------------
\begin{align*}
    \rho & = i_j + (p-1)(d_1-1) + \sgn(\sigma) + i_j-j + (d_2 + \dots + d_m) \\
    & = 2i_j + pd_1 -d_1 - p +1 + \sgn(\sigma) -j + (d_2 + \dots + d_m)
\end{align*}
%%----------------------------------------------------

Thus modulo 2, we have
%%----------------------------------------------------
\begin{align}
    \label{case3case2sign2}
    \rho = pd_1 + \sgn(\sigma) + (d_2 + \dots + d_m) - d_1 - p - j + 1
\end{align}
%%----------------------------------------------------
On comparing \eqref{case3case2sign1} and \eqref{case3case2sign2}, we see that they give us opposite signs. 
\end{itemize}

\item Case 4: We split up $[w^1_0, \dots w^1_p]$.

Suppose we split $[w^1_0, \dots, w^1_p]$ into
%%----------------------------------------------------
\begin{align*}
    [w^1_0, \dots, \widehat{w^1_{j_0}}, w^1_{j_0+1}, \dots, w^1_{j_q-1}, \widehat{w^1_{j_q}}, \dots, w^1_p] \otimes [x^1_0, \dots, x^1_q]
\end{align*}
%%----------------------------------------------------

Let us now look at the accompanying sign.

Let $\sigma$ be the sign of the permutation taking $v^1_{i_0}, \dots, v^1_{i_p}$ to the first $p+1$ positions in $[v^1_0, \dots, v^1_{d_1}]$, let $\sigma'$ be the shuffle taking $w^1_{j_0}, \dots, w^1_{j_q}$ to the first $q+1$ positions of $[w^1_0, \dots, w^1_p]$, and let $\sigma''$ be the shuffle taking $v^1_{i_{j_0}}, \dots, v^1_{i_{j_q}}$ to the first $q+1$ positions of $[v^1_0, \dots, v^1_{d_1}]$ and $v^1_0, \dots, \widehat{v^1_{i_0}}, \dots, \widehat{v^1_{i_p}}, \dots, v^1_{d_1}$ to the next $p-q$ positions.
Then note that
%%----------------------------------------------------
\begin{align*}
    (-1)^{\sigma''} = (-1)^{\sigma'} (-1)^\sigma
\end{align*}
%%----------------------------------------------------

The sign then accompanying the term where we first split $[v^1_0, \dots, v^1_{d_1}]$ and then split $[w^1_0, \dots, w^1_p]$ is $(-1)^\tau$ with:
%%----------------------------------------------------
\begin{align*}
    \tau & = ( pd_1 + \sgn(\sigma) + (d_2 + \dots + d_m)) + ((d_1-(p+1)-1) +qp + \sgn(\sigma') + (d_2 + \dots + d_m))
\end{align*}
%%----------------------------------------------------
Modulo 2, this is equivalent to
%%----------------------------------------------------
\begin{align}
    \label{case4sign1}
    pd_1 + \sgn(\sigma) + (d_1-p) + qp + \sgn(\sigma')
\end{align}
%%----------------------------------------------------
This term appears exactly once in 
%%----------------------------------------------------
\begin{align*}
      \sum_{r+s = t, s>0}d^{a-r}_{-s, s-1} \circ d^a_{-r, r-1},
\end{align*}
%%----------------------------------------------------
namely when we first split the $v^1_{i_{j_0}}, \dots, v^1_{i_{j_q}}$ from $[v^1, \dots, v^1_{d_1}]$, and then split the remaining $v^1_{i_k}$ from
$[v^1_0, \dots, \widehat{v^1_{i_{j_0}}}, \dots, \widehat{v^1_{i_{j_q}}}, \dots, v^1_{d_1}]$.

The accompanying sign in this case is $(-1)^\rho$ with
%%----------------------------------------------------
\begin{align}
    \rho & = \sgn(\sigma'') + qd_1 + (d_2 + \dots + d_m) + (p-q-1)(d_1-q-1) + (q + d_2 + \dots + d_m)
\end{align}
%%----------------------------------------------------
which modulo 2 is equivalent to
%%----------------------------------------------------
\begin{align}
\label{case4sign2}
    \sgn(\sigma'') + qd_1 + (p-q-1)(d_1-q-1) + q
\end{align}
%%----------------------------------------------------
On comparing \eqref{case4sign1} and \eqref{case4sign2}, and using the fact that $(-1)^{\sigma''} = (-1)^\sigma(-1)^{\sigma'}$, we see that we need
%%----------------------------------------------------
\begin{align*}
    pd_1 + d_1-p + qp
\end{align*}
%%----------------------------------------------------
and 
%%----------------------------------------------------
\begin{align*}
    qd_1 + (p-q-1)(d_1-q-1) + q
\end{align*}
%%----------------------------------------------------
to be of opposite parities.
Equivalently, we want 
%%----------------------------------------------------
\begin{align}
\label{case4eqn}
    (pd_1 + d_1-p + qp) + (qd_1 + (p-q-1)(d_1-q-1) + q) + 1
\end{align}
%%----------------------------------------------------
to be even.

We have
%%----------------------------------------------------
\begin{align*}
    (p-q-1)(d_1-q-1) & = (p-q)d_1 - (p-q)q - (p-q) + d_1 - q - 1
\end{align*}
%%----------------------------------------------------

Thus Equation \eqref{case4eqn} becomes
%%----------------------------------------------------
\begin{align*}
   & (pd_1 + d_1 - p + qp) + (qd_1 + ((p-q)d_1 - (p-q)q - (p-q) + d_1 - q - 1) + q) + 1\\
   = &  (pd_1 + qd_1 + (p-q)d_1) + 2d_1 - p + qp - (p-q)q - (p-q)\\
   = & 2pd_1 + 2d_1 - p + qp - pq + q^2 -p + q \\
   = & 2pd_1 + 2d_1 -2p + q(q+1)
\end{align*}
%%----------------------------------------------------
which is indeed always even.

\end{enumerate}

Furthermore, on running the proof, we can see that there are no additional terms that appear upon applying
%%----------------------------------------------------
\begin{align*}
    \sum_{r+s = t, s>0}d^{a-r}_{-s, s-1} \circ d^a_{-r, r-1}
\end{align*}
%%----------------------------------------------------
besides the ones that showed up in the preceding casework.

This completes the proof that 
%%----------------------------------------------------
\begin{align*}
    \sum_{r+s = t}d^{a-r}_{-s, s-1} \circ d^a_{-r, r-1} = 0
\end{align*}
%%----------------------------------------------------

\end{proof}
%%----------------------------------------------------

%%----------------------------------------------------
\subsection{Proof of Theorem \ref{thm:sharblyres}}
\label{sec:integralres}
We now outline how to prove Theorem \ref{thm:sharblyres}. The arguments in this section so far prove a version of the theorem when we tensor everything with $\otimes_\Z \Q$, and when $R$ has finitely many units. But note that Equation \eqref{Steinbergs} holds without any of these assumptions. The reason we made these assumptions was that Lemma \ref{CPproj} only tells us how to recognise when modules over a rational group ring $\Q G$ are projective, and we needed projectivity to know the existence of the maps described in Proposition \ref{statingdr}, as implied by Lemma \ref{gluestep}.

However, if we knew the existence of maps that satisfied the conditions of Proposition \ref{statingdr} through some other means, then we would be able to use the argument in the proof of Lemma \ref{gluestep} to piece resolutions together. In that case, we would be able to avoid the assumption of projectivity.

Now, the maps defined in Proposition \ref{differential} can be defined on the integral versions of the groups in question, and the same proof checks that they satisfy the required conditions of Proposition \ref{statingdr}. Thus we can run all the same arguments of this section, to prove Theorem \ref{thm:sharblyres}.

%%-----------------------------------------------------

%%--------------------------------------------------------------

\section{The top degree cohomology of some principal congruence subgroups}
\label{sec:topdegcoh}

In this section we shall prove Theorem \ref{thm:congsubgp}, that combined with Borel-Serre duality \cite{borel1973corners} computes the top degree cohomology of certain congruence subgroups of symplectic groups over number rings that are Euclidean domains. 

%%-----------------------------------------------------------------
\begin{comment}
\begin{prop}
    \label{prop:topdegcong}
    Let ${R}$ be a Euclidean domain and $K$ its field of fractions. Let $\F$ be the quotient of ${R}$ by some prime $p \in {R}$. Suppose that all units of $\F$ can be lifted to units of ${R}$. Then, for all $n \geq 1$, we have an isomorphism
    \begin{align*}
        (\St^\w_{2n}(K))_{\Gamma^\w_{2n}(p)} \cong \St^\w_{2n}(\F)
    \end{align*}
\end{prop}
\end{comment}
%%--------------------------------------------------------------

The proof will use a subcomplex of the symplectic Sharbly resolution (Theorem \ref{thm:sharblyres}) that in particular gives a presentation of $\St^\w(\F^{2n})$ for a field $\F$, which we summarize in the following lemma.
To avoid clutter, we will drop $\otimes$ signs while writing out simple tensors in the groups in the resolution.

%%----------------------------------------------------
\begin{lemma}
\label{lem:Stwpresent}
Let $\F$ be a field, and let $V$ be a genus $n$ symplectic space over $\F$. Then we have the following complex of $\Z[\Sp(V)]$-modules 
\begin{align*}
V_2 \xrightarrow{\del_2} V_{1,2} \oplus V_{1,1} \xrightarrow{\del_1} V_0 \to \St^\w(V)
\end{align*}
where $V_0, V_{1,2}, V_{1,2}, V_2$ are as follows:
\begin{itemize}
    \item $V_0$ is the direct sum of all terms of $\Sh_0(W_1)\otimes\dots\otimes\Sh_0(W_n)$, where $W_1 \oplus \dots \oplus W_n = V$ is a decomposition of $V$ into mutually perpendicular genus 1 symplectic subspaces.
    \item $V_{1,1}$ is the direct sum of terms of the form $\Sh_0(W_1)\otimes \dots \otimes\Sh_1(W_i)\otimes \dots \otimes \Sh_0(W_n)$, where $W_1 \oplus \dots \oplus W_n = V$ is a decomposition of $V$ into mutually perpendicular genus 1 symplectic subspaces.
    \item $V_{1,2}$ is the direct sum over all terms of the form $\Sh_0(W_1)\otimes\dots \otimes \Sh_0(W_{n-1})$, where $V = W_1 \oplus \dots \oplus W_{n-1}$ is a decomposition into mutually perpendicular subspaces, where all but one of the $W_i$ have genus 1 and the remaining one has genus 2.
    \item $V_2$ is the direct sum over all terms of the form $\Sh_0(W_1)\otimes \dots \otimes\Sh_1(W_i)\otimes \dots \otimes \Sh_0(W_{n-1})$, where $V = W_1 \oplus \dots \oplus W_{n-1}$ is a decomposition into mutually perpendicular subspaces, $W_i$ has genus 2 and all the other $W_j$ have genus 1.
\end{itemize}
The maps $\del_1, \del_2$ are as follows:
\begin{itemize}
    \item $[\lin{\vec{v}_1}, \lin{\vec{v}_{\bar{1}}}]\dots[\lin{\vec{u}_i}, \lin{\vec{v}_i}, \lin{\vec{w}_i}]\dots[\lin{\vec{v}_n}, \lin{\vec{v}_{\bar{n}}}] \in V_{1,1}$ maps to the sum
    \begin{align*}
        \left([\lin{\vec{v}_1}, \lin{\vec{v}_{\bar{1}}}]\dots[\lin{\vec{v}_i}, \lin{\vec{w}_i}]\dots[\lin{\vec{v}_n}, \lin{\vec{v}_{\bar{n}}}]\right) - \left([\lin{\vec{v}_1}, \lin{\vec{v}_{\bar{1}}}]\dots[\lin{\vec{u}_i}, \lin{\vec{w}_i}]\dots[\lin{\vec{v}_n}, \lin{\vec{v}_{\bar{n}}}]\right)\\ + \left([\lin{\vec{v}_1}, \lin{\vec{v}_{\bar{1}}}]\dots[\lin{\vec{u}_i}, \lin{\vec{v}_i}, ]\dots[\lin{\vec{v}_n}, \lin{\vec{v}_{\bar{n}}}]\right).
    \end{align*}
    
    \item $[\lin{\vec{v}_1}, \lin{\vec{v}_{\bar{1}}}]\dots[\lin{\vec{v}_{i,0}}, \lin{\vec{v}_{i,1}}, \lin{\vec{v}_{i,2}}, \lin{\vec{v}_{i,3}}]\dots[\lin{\vec{v}_{n-1}}, \lin{\vec{v}_{\overline{n-1}}}] \in V_{1,2}$ maps to a sum of 'split terms' as defined in Proposition \ref{differential}:
    \begin{align*}
        (-1)^\nu[\lin{\vec{v}_1}, \lin{\vec{v}_{\bar{1}}}]\dots[\lin{\vec{w}_i}, \lin{\vec{x}_i}][\lin{\vec{y}_i}, \lin{\vec{z}_i}]\dots[\lin{\vec{v}_{n-1}}, \lin{\vec{v}_{\overline{n-1}}}] 
    \end{align*}
    where $\vec{w}_i, \vec{x}_i$ are a subset of the vectors $\vec{v}_{i,0}, \vec{v}_{i,1}, \vec{v}_{i,2}, \vec{v}_{i,3}$ spanning a genus 1 subspace of $W_i$, and $\vec{y}_i$ and $\vec{z}_i$ are the projections of the other two vectors onto its perpendicular subspace. The sign $\nu$ is given by the formula 
    \begin{align*}
        \nu = j_1+j_2 + n-1-i
    \end{align*}
    where $j_1, j_2$ are the indices of the $\vec{v}_{i,t}$ that correspond to $\vec{w}_i$ and $\vec{x}_i$.

    \item Elements of $V_2$ map to a sum of `split' and `omit' terms, as described in Proposition \ref{differential}. (For our purposes in this section, we will not need the description of $\del_2$ on $V_2$).
\end{itemize}
This complex is exact at $V_0$, i.e.
\begin{align*}
    V_{1,2} \oplus V_{1,1} \xrightarrow{\del_1} V_0 \to \St^\w(V)
\end{align*}
is a presentation of $\St^\w(V)$ over $\Sp(V)$.
\end{lemma}
%%----------------------------------------------------

%%------------------------------------------
\begin{definition}
    \label{defn:stdsharblies}
    Let $\F$ be a field, and $V$ a genus $n$ symplectic space over $\F$.
    We say an element of $V_{1,2}$ is a \emph{standard Sharbly} if it can be written as
    \begin{align*}
        [\lin{\vec{v}_1}, \lin{\vec{v}_{\bar{1}}}]\dots[\lin{\vec{v}_i}, \lin{\vec{v}_{\bar{i}}}, \lin{\vec{v}_{i+1}+a\vec{v}_i}, \lin{\vec{v}_{\overline{i+1}}}]\dots[\lin{\vec{v}_{n-1}}, \lin{\vec{v}_{\overline{n-1}}}]
    \end{align*}

    where $\vec{v}_i, \vec{v}_{\bar{i}}, \vec{v}_{i+1}, \vec{v}_{\overline{i+1}}$ is a symplectic basis for a genus 2 symplectic subspace of $V$, and $a \in \F$.
\end{definition}
%%------------------------------------------

We now have the following key lemma.
%%------------------------------------------
\begin{lemma}
    \label{lem:stdsharblies}
    Let $\F$ be a field, and $V$ a genus $n$ symplectic space over $\F$.
    Suppose $\sigma \in V_{1,2}$.
    %\begin{align*}
     %   \sigma = [\lin{\vec{w}_0}, \lin{\vec{w}_1}, \lin{\vec{w}_2}, \lin{\vec{w}_3}][\lin{\vec{v}_2}, \lin{\vec{v}_{\bar{2}}}]\dots[\lin{\vec{v}_{n-1}}, \lin{\vec{v}_{\overline{n-1}}}]
    %\end{align*}
   % is an element of $V_{1,2}$ of the form $\Sh_0(W_1) \otimes \dots \otimes \Sh_0(W_{n-1})$ , as defined in Lemma \ref{lem:Stwpresent}.
    Then, there exists $\tau \in V_2$ such that $\sigma - \del_2\tau \in V_{1,2} \oplus V_{1,1}$ is a sum of terms in $V_{1,1}$ and standard Sharblies in $V_{1,2}$.
\end{lemma}
\begin{proof}
    It is enough to prove the assertion when $\sigma$ is of the form 
    \begin{align*}
        [\lin{\vec{v}_1}, \lin{\vec{v}_{\bar{1}}}]\dots[\lin{\vec{v}_{i,0}}, \lin{\vec{v}_{i,1}}, \lin{\vec{v}_{i,2}}, \lin{\vec{v}_{i,3}}]\dots[\lin{\vec{v}_{n-1}}, \lin{\vec{v}_{\overline{n-1}}}]
    \end{align*}
    There must be a pair among the $\vec{v}_{i,t}$ that pair nontrivially. Set $\vec{w}_1, \vec{w}_{\bar{1}}$ to be scalar multiples of these, so that $\w(\vec{w}_1, \vec{w}_{\bar{1}}) = 1$. We can complete this to a symplectic basis $\vec{w}_1, \vec{w}_{\bar{1}}, \vec{w}_2, \vec{w}_{\bar{2}}$ for the span of the $\vec{v}_{i,t}$, and write the remaining two vectors as (scalar multiples of) $\vec{w}_2 + a\vec{w}_1 + b\vec{w}_{\bar{1}}$ and $\vec{w}_{\bar{2}} + c\vec{w}_1 + d\vec{w}_{\bar{1}}$ for some $a,b,c,d \in \F$. 
    Thus, upto sign, we have
    \begin{align*}
        \sigma = [\lin{\vec{v}_1}, \lin{\vec{v}_{\bar{1}}}]\dots[\lin{\vec{w}_1}, \lin{\vec{w}_{\bar{1}}}, \lin{\vec{w}_2 + a\vec{w}_1 + b\vec{w}_{\bar{1}}}, \lin{\vec{w}_{\bar{2}} + c\vec{w}_1 + d\vec{w}_{\bar{1}}}]\dots[\lin{\vec{v}_{n-1}}, \lin{\vec{v}_{\overline{n-1}}}]
    \end{align*}
We shall successively show that atleast three of the elements $a,b,c,d \in \F$ can be assumed to be 0. Note that this will reduce us to a standard Sharbly.

%%--------------------------------------------------------------
\begin{enumerate}
    \item \textbf{Case 1: $a,b,c,d \in \F^\times$:}
    Consider the element of $V_2$ given by
%%----------------------------------------------------    
    \begin{align*}
       \tau = [\lin{\vec{v}_1}, \lin{\vec{v}_{\bar{1}}}]\dots[\lin{\vec{w}_1}, \lin{\vec{w}_{\bar{1}}},\lin{\vec{w}_2 + a\vec{w}_1}, \lin{\vec{w}_2 + a\vec{w}_1 + b\vec{w}_{\bar{1}}}, \lin{\vec{w}_{\bar{2}} + c\vec{w}_1 + d\vec{w}_{\bar{1}}}]\dots[\lin{\vec{v}_{n-1}}, \lin{\vec{v}_{\overline{n-1}}}]
    \end{align*}
%%----------------------------------------------------
    
    Then $\del_2\tau$ is equal to an element of $V_{1,1}$, plus the following sum in $V_{1,2}$:
    %%----------------------------------------------------
    \begin{IEEEeqnarray}{l}
        \hspace*{-2em}[\lin{\vec{v}_1}, \lin{\vec{v}_{\bar{1}}}]\dots[\lin{\vec{w}_1}, \lin{\vec{w}_{\bar{1}}}, \lin{\vec{w}_2 + a\vec{w}_1 + b\vec{w}_{\bar{1}}}, \lin{\vec{w}_{\bar{2}} + c\vec{w}_1 + d\vec{w}_{\bar{1}}}]\dots[\lin{\vec{v}_{n-1}}, \lin{\vec{v}_{\overline{n-1}}}] \nonumber \\ 
        \hspace*{-1em}\text{} - [\lin{\vec{v}_1}, \lin{\vec{v}_{\bar{1}}}]\dots[\lin{\vec{w}_1}, \lin{\vec{w}_{\bar{1}}},\lin{\vec{w}_2 + a\vec{w}_1}, \lin{\vec{w}_{\bar{2}} + c\vec{w}_1 + d\vec{w}_{\bar{1}}}]\dots[\lin{\vec{v}_{n-1}}, \lin{\vec{v}_{\overline{n-1}}}] 
        \label{eq:case1term2}
        \\
         \hspace*{-1em}\text{} - [\lin{\vec{v}_1}, \lin{\vec{v}_{\bar{1}}}]\dots[\lin{\vec{w}_1}, \lin{\vec{w}_2 + a\vec{w}_1}, \lin{\vec{w}_2 + a\vec{w}_1 + b\vec{w}_{\bar{1}}}, \lin{\vec{w}_{\bar{2}} + c\vec{w}_1 + d\vec{w}_{\bar{1}}}]\dots[\lin{\vec{v}_{n-1}}, \lin{\vec{v}_{\overline{n-1}}}]
         \label{eq:case1term3}
    \end{IEEEeqnarray}
    %%----------------------------------------------------
    
   Thus note that it is enough to prove the assertion for the latter two terms \eqref{eq:case1term2} and \eqref{eq:case1term3} in the above sum. We shall show that both these terms fit into Case 2. Clearly term \eqref{eq:case1term2} already fits into Case 2. For \eqref{eq:case1term3}, consider the Sharbly $\sigma' = [\lin{\vec{w}_1}, \lin{\vec{w}_2 + a\vec{w}_1}, \lin{\vec{w}_2 + a\vec{w}_1 + b\vec{w}_{\bar{1}}}, \lin{\vec{w}_{\bar{2}} + c\vec{w}_1 + d\vec{w}_{\bar{1}}}]$.

   %%----------------------------------------------------
   Set
   \begin{align*}
       \vec{u}_1 = \vec{w}_1,\hspace{0.5cm} \vec{u}_{\bar{1}} = \vec{w}_{\bar{1}}+ab^{-1}\vec{v}_1 + b^{-1}\vec{w}_2 \\
       \vec{u}_2 = \vec{w}_2, \hspace{0.5cm} \vec{u}_{\bar{2}} = \vec{w}_{\bar{2}}+b^{-1}\vec{w}_1-ab^{-1}\vec{w}_2
   \end{align*}
   %%----------------------------------------------------
   Then note that $\vec{u}_1, \vec{u}_{\bar{1}}, \vec{u}_2, \vec{u}_{\bar{2}}$ are a symplectic basis for the span of $\sigma'$, and that we can write $\sigma'$ as
   %%----------------------------------------------------
   \begin{align*}
       \sigma' & = [\lin{\vec{u}_1}, \lin{\vec{u}_2+a\vec{u}_1}, \lin{\vec{u}_{\bar{1}}}, \lin{\vec{u}_{\bar{2}}+d\vec{u}_{\bar{1}}+(c-b^{-1}-ab^{-1}d)\vec{u}_1}] \\
       & = - [\lin{\vec{u}_1}, \lin{\vec{u}_{\bar{1}}},\lin{\vec{u}_2+a\vec{u}_1}, \lin{\vec{u}_{\bar{2}}+d\vec{u}_{\bar{1}}+ (c-b^{-1}-ab^{-1}d)\vec{u}_1}]
   \end{align*}
   %%----------------------------------------------------
   This calculation shows us that the term \eqref{eq:case1term3} also falls within Case 2.
   
%%----------------------------------------------------
%%----------------------------------------------------
   \item \textbf{Case 2: At least one of $a,b$ and one of $c,d$ are non-zero:}
   We can assume WLOG that $a$ and $c$ are nonzero. We can also assume that $b = 0$. (Otherwise, if both $b, d$ were non-zero, then we'd be back in Case 1).
   Thus we want to prove the assertion of our lemma for Sharblies of the form
   %%----------------------------------------------------
   \begin{align*}
       [\lin{\vec{v}_1}, \lin{\vec{v}_{\bar{1}}}]\dots[\lin{\vec{w}_1}, \lin{\vec{w}_{\bar{1}}},\lin{\vec{w}_2 + a\vec{w}_1}, \lin{\vec{w}_{\bar{2}} + c\vec{w}_1 + d\vec{w}_{\bar{1}}}]\dots[\lin{\vec{v}_{n-1}}, \lin{\vec{v}_{\overline{n-1}}}]
   \end{align*}
   %%----------------------------------------------------
   Consider $\tau \in V_2$ given by
   %%----------------------------------------------------
   \begin{align*}
       \tau = [\lin{\vec{v}_1}, \lin{\vec{v}_{\bar{1}}}]\dots[\lin{\vec{w}_1}, \lin{\vec{w}_{\bar{1}}},\lin{\vec{w}_2},\lin{\vec{w}_2 + a\vec{w}_1}, \lin{\vec{w}_{\bar{2}} + c\vec{w}_1 + d\vec{w}_{\bar{1}}}]\dots[\lin{\vec{v}_{n-1}}, \lin{\vec{v}_{\overline{n-1}}}]
   \end{align*}
   %%----------------------------------------------------
    Then $\del_2\tau$ is equal to an element of $V_{1,1}$, plus the following sum in $V_{1,2}$:
    %%----------------------------------------------------
    \begin{IEEEeqnarray}{l}
        \hspace*{-2em}[\lin{\vec{v}_1}, \lin{\vec{v}_{\bar{1}}}]\dots[\lin{\vec{w}_{\bar{1}}},\lin{\vec{w}_2},\lin{\vec{w}_2 + a\vec{w}_1}, \lin{\vec{w}_{\bar{2}} + c\vec{w}_1 + d\vec{w}_{\bar{1}}}]\dots[\lin{\vec{v}_{n-1}}, \lin{\vec{v}_{\overline{n-1}}}]
        \label{eq:case2term1}\\
        \text{} + [\lin{\vec{v}_1}, \lin{\vec{v}_{\bar{1}}}]\dots[\lin{\vec{w}_1}, \lin{\vec{w}_{\bar{1}}},\lin{\vec{w}_2 + a\vec{w}_1}, \lin{\vec{w}_{\bar{2}} + c\vec{w}_1 + d\vec{w}_{\bar{1}}}]\dots[\lin{\vec{v}_{n-1}}, \lin{\vec{v}_{\overline{n-1}}}] 
        \label{eq:case2term2}\\
        \text{} - [\lin{\vec{v}_1}, \lin{\vec{v}_{\bar{1}}}]\dots[\lin{\vec{w}_1}, \lin{\vec{w}_{\bar{1}}},\lin{\vec{w}_2}, \lin{\vec{w}_{\bar{2}} + c\vec{w}_1 + d\vec{w}_{\bar{1}}}]\dots[\lin{\vec{v}_{n-1}}, \lin{\vec{v}_{\overline{n-1}}}]
        \label{eq:case2term3}
    \end{IEEEeqnarray}
    %%----------------------------------------------------

    %%----------------------------------------------------
    In the above sum, the second term \eqref{eq:case2term2} is the one we are considering. It is enough to reduce the other two terms to Case 3. 
    The term \eqref{eq:case2term3} already falls into Case 3, so let us consider \eqref{eq:case2term1}.
    %%----------------------------------------------------
    
    Let $\sigma'$ be given by
    %%----------------------------------------------------
    \begin{align*}
        \sigma' = [\lin{\vec{w}_{\bar{1}}},\lin{\vec{w}_2},\lin{\vec{w}_2 + a\vec{w}_1}, \lin{\vec{w}_{\bar{2}} + c\vec{w}_1 + d\vec{w}_{\bar{1}}}]
    \end{align*}
    %%----------------------------------------------------
    Set
    %%----------------------------------------------------
    \begin{IEEEeqnarray*}{l}
        \vec{u}_1 = \vec{w}_1+a^{-1}\vec{w}_2, \hspace{0.5cm} \vec{u}_{\bar{1}} = \vec{w}_{\bar{1}} \\
        \vec{u}_2 = \vec{w}_2, \hspace{0.5cm} \vec{u}_{\bar{2}}= \vec{w}_{\bar{2}} -a^{-1}\vec{w}_{\bar{1}}-a^{-1}c\vec{w}_2
    \end{IEEEeqnarray*}
    %%----------------------------------------------------
    Then we have
    %%----------------------------------------------------
    \begin{align*}
        \sigma' = [\lin{\vec{u}_{\bar{1}}}, \lin{\vec{u}_2}, \lin{\vec{u_1}}, \lin{\vec{u}_{\bar{2}}+c\vec{u}_1+(d-a^{-1})\vec{u}_{\bar{1}}}]
    \end{align*}
    %%----------------------------------------------------
    which tells us \eqref{eq:case2term1} also falls in Case 3.
%%----------------------------------------------------
%%----------------------------------------------------
    \item \textbf{Case 3: One of the pairs $(a,b)$ or $(c,d)$ is equal to $(0,0)$:}
    Assume WLOG that $(a,b) = (0,0)$. Thus we are considering
    %%----------------------------------------------------
    \begin{align*}
        [\lin{\vec{v}_1}, \lin{\vec{v}_{\bar{1}}}]\dots[\lin{\vec{w}_1}, \lin{\vec{w}_{\bar{1}}},\lin{\vec{w}_2}, \lin{\vec{w}_{\bar{2}} + c\vec{w}_1 + d\vec{w}_{\bar{1}}}]\dots[\lin{\vec{v}_{n-1}}, \lin{\vec{v}_{\overline{n-1}}}]
    \end{align*}
    %%----------------------------------------------------
    If one of $c,d$ is also 0, then this is already a standard Sharbly.
    If not, then consider $\tau \in V_2$ given by
    %%----------------------------------------------------
    \begin{align*}
        \tau = [\lin{\vec{v}_1}, \lin{\vec{v}_{\bar{1}}}]\dots[\lin{\vec{w}_1}, \lin{\vec{w}_{\bar{1}}}, \lin{c\vec{w}_1+d\vec{w}_{\bar{1}}},\lin{\vec{w}_2}, \lin{\vec{w}_{\bar{2}} + c\vec{w}_1 + d\vec{w}_{\bar{1}}}]\dots[\lin{\vec{v}_{n-1}}, \lin{\vec{v}_{\overline{n-1}}}]
    \end{align*}
    %%----------------------------------------------------
    Then $\del_2\tau$ is equal to an element of $V_{1,1}$ plus the following sum in $V_{1,2}$:
    %%----------------------------------------------------
    \begin{IEEEeqnarray*}{l}
        \hspace*{-1em}[\lin{\vec{v}_1}, \lin{\vec{v}_{\bar{1}}}]\dots[\lin{\vec{w}_1}, \lin{\vec{w}_{\bar{1}}}, \lin{\vec{w}_{\bar{2}} + c\vec{w}_1 + d\vec{w}_{\bar{1}}}]\dots[\lin{\vec{v}_{n-1}}, \lin{\vec{v}_{\overline{n-1}}}] \\
        \text{} - [\lin{\vec{v}_1}, \lin{\vec{v}_{\bar{1}}}]\dots[\lin{\vec{w}_1}, \lin{c\vec{w}_1+d\vec{w}_{\bar{1}}},\lin{\vec{w}_2}, \lin{\vec{w}_{\bar{2}} + c\vec{w}_1 + d\vec{w}_{\bar{1}}}]\dots[\lin{\vec{v}_{n-1}}, \lin{\vec{v}_{\overline{n-1}}}] \\
        \text{} + [\lin{\vec{v}_1}, \lin{\vec{v}_{\bar{1}}}]\dots[\lin{\vec{w}_{\bar{1}}}, \lin{c\vec{w}_1+d\vec{w}_{\bar{1}}},\lin{\vec{w}_2}, \lin{\vec{w}_{\bar{2}} + c\vec{w}_1 + d\vec{w}_{\bar{1}}}]\dots[\lin{\vec{v}_{n-1}}, \lin{\vec{v}_{\overline{n-1}}}]
    \end{IEEEeqnarray*}
    %%----------------------------------------------------
    This is equal to 
    %%----------------------------------------------------
    \begin{IEEEeqnarray*}{l}
        \hspace*{-1em}[\lin{\vec{v}_1}, \lin{\vec{v}_{\bar{1}}}]\dots[\lin{\vec{w}_1}, \lin{\vec{w}_{\bar{1}}}, \lin{\vec{w}_{\bar{2}} + c\vec{w}_1 + d\vec{w}_{\bar{1}}}]\dots[\lin{\vec{v}_{n-1}}, \lin{\vec{v}_{\overline{n-1}}}] \\
        \text{} - [\lin{\vec{v}_1}, \lin{\vec{v}_{\bar{1}}}]\dots[\lin{\vec{w}_1}, \lin{cd^{-1}\vec{w}_1+\vec{w}_{\bar{1}}},\lin{\vec{w}_2}, \lin{\vec{w}_{\bar{2}} + d(cd^{-1}\vec{w}_1 + \vec{w}_{\bar{1}})}]\dots[\lin{\vec{v}_{n-1}}, \lin{\vec{v}_{\overline{n-1}}}] \\
        \text{} + [\lin{\vec{v}_1}, \lin{\vec{v}_{\bar{1}}}]\dots[\lin{\vec{w}_{\bar{1}}}, \lin{\vec{w}_1+c^{-1}d\vec{w}_{\bar{1}}},\lin{\vec{w}_2}, \lin{\vec{w}_{\bar{2}} + c(\vec{w}_1 + c^{-1}d\vec{w}_{\bar{1}})}]\dots[\lin{\vec{v}_{n-1}}, \lin{\vec{v}_{\overline{n-1}}}]
    \end{IEEEeqnarray*}
    %%----------------------------------------------------
    and one sees that all three of these terms are standard Sharblies.
\end{enumerate}
%%-------------------------------------------------------
\end{proof}
%%----------------------------------------------

Next we will need a notion of the symplectic Tits building and the Steinberg module over rings other than fields.

%%----------------------------------------------------
\begin{definition}[Tits building]
    Let ${R}$ be a commutative ring and let $V \cong {R}^{2n}$ be equipped with the standard symplectic form. Define $\mc{T}^\w(V)$ to be the poset of isotropic summands of $V$, ordered by inclusion. Also, let $\bb{T}^\w(V)$ denote the geometric realization of $\mc{T}(V)$, viewed as a simplicial complex. For $n \geq 1$, we will write $\mc{T}^\w_{2n}({R}) = \mc{T}^\w({R}^{2n})$ and $\bb{T}^\w_{2n}({R}) = \bb{T}^\w({R}^{2n})$.
\end{definition}
%%--------------------------------------------------

%%--------------------------------------------------
\begin{lemma}
    Suppose ${R}$ is a PID, and $K$ its field of fractions. For $n \geq 1$, we have $\mc{T}^\w_{2n}(R) \cong \mc{T}^\w_{2n}(K)$.
\end{lemma}
%%----------------------------------------------------
\begin{proof}
    This follows from the fact that there is a bijection between subspaces of $K^{2n}$ and direct summands of ${R}^{2n}$ taking a subspace $V \subset K^{2n}$ to $V \cap R^{2n}$ and a direct summand $W \subset R^{2n}$ to $W \otimes_R K$, and $V$ is isotropic if and only if $W$ is.
\end{proof}
%%-----------------------------------------------

%In light of this, we can define the symplectic Steinberg module over a Euclidean domain.
%%-----------------------------------------------
%\begin{definition}
 %   For $\mc{R}$ a PID, let $\St^\w_{2n}(\mc{R})$ denote the group $\widetilde{\tH}_{n-1}(\bb{T}^\w_{2n}(\mc{R}))$.
%\end{definition}
%%-----------------------------------------------

For our proof we will need to know a cyclicity result for $\St^\w_{2n}(K)$ when $K$ is the field of fractions of a Euclidean domain, which is due to Gunnells \cite{gunnells2000symplectic}. We will also need to know certain relations that hold among these apartments classes, which can easily be verified to hold using the description of apartments given in Definition \ref{Apt}. We summarize this in the following lemma. 
%%-----------------------------------------------
\begin{lemma}[{Gunnells\cite[Theorem 4.11]{gunnells2000symplectic}}]
    \label{lem:intapt}
    Let ${R}$ be a Euclidean domain and $K$ its field of fractions.
    Then $\St^\w_{2n}(K)$ is generated by \emph{integral apartment classes}, i.e.\space apartment classes of the form
    \begin{align*}
        [\lin{\vec{v}_1}, \lin{\vec{v}_{\bar{1}}}, \dots, \lin{\vec{v}_n}, \lin{\vec{v}_{\bar{n}}}]
    \end{align*}
    so that $\vec{v}_1, \vec{v}_{\bar{1}}, \dots, \vec{v}_n, \vec{v}_{\bar{n}}$ form a symplectic basis of ${R}^{2n}$.
    Furthermore, the following relations hold among these apartment classes:
%%----------------------------------------------------
    \begin{align}\label{rel:perm}
        [\lin{v_1}, \lin{v_{\bar{1}}}, \dots, \lin{\vec{v}_n}, \lin{\vec{v}_{\bar{n}}}] = (-1)^{\sgn (\sigma)}[\lin{v_{\sigma(1)}}, \lin{v_{\sigma(\bar{1})}}, \dots, \lin{\vec{v}_{\sigma(n)}}, \lin{\vec{v}_{\sigma(\bar{n})}}]
    \end{align}
    %%----------------------------------------------------
    for any signed permutation $\sigma$.
    %%----------------------------------------------------
    \begin{IEEEeqnarray}{l}
        \hspace*{-2em}[\lin{\vec{v}_1}, \lin{\vec{v}_{\bar{1}}}, \dots, \lin{\vec{v}_n}, \lin{\vec{v}_{\bar{n}}}] \nonumber \\ 
         \label{rel:byk1}
        = [\lin{\vec{v}_1}, \lin{a\vec{v}_1+b\vec{v}_{\bar{1}}}, \dots, \lin{\vec{v}_n}, \lin{\vec{v}_{\bar{n}}}] + [\lin{a\vec{v}_1+b\vec{v}_{\bar{1}}}, \lin{\vec{v}_{\bar{1}}}, \dots, \lin{\vec{v}_n}, \lin{\vec{v}_{\bar{n}}}]
    \end{IEEEeqnarray}
    for units $a,b \in {R}$.
    %%----------------------------------------------------
    %%----------------------------------------------------
    \begin{IEEEeqnarray}{l}
        \hspace*{-2em}[\lin{\vec{v}_1}, \lin{\vec{v}_{\bar{1}}}, \dots, \lin{\vec{v}_n}, \lin{\vec{v}_{\bar{n}}}] 
        = [\lin{\vec{v}_1}, \lin{\vec{v}_{\bar{1}}-a\vec{v}_{\bar{2}}}, \lin{\vec{v}_2+a\vec{v}_1}, \lin{\vec{v}_{\bar{2}}}, \dots, \lin{\vec{v}_n}, \lin{\vec{v}_{\bar{n}}}] 
        \nonumber
        \\
        \hspace{2cm} \text{} + [\lin{\vec{v}_{\bar{1}}-a\vec{v}_{\bar{2}}}, \lin{\vec{v}_2}, \lin{\vec{v}_2+a\vec{v}_1}, \lin{\vec{v}_{\bar{1}}}, \dots, \lin{\vec{v}_n}, \lin{\vec{v}_{\bar{n}}}]
        \label{rel:byk2}
    \end{IEEEeqnarray}
    for a unit $a \in {R}$.
    %%----------------------------------------------------
\end{lemma}
%%---------------------------------------

%%---------------------------------------
\begin{definition}\label{def:integralgen}
    Let ${R}$ be a PID and $K$ its field of fractions. For a fixed $n \geq 1$, define $\mc{V}_0$ to be the quotient of the free abelian group on lines in $K^{2n}$ $[\lin{\vec{v}_1}, \lin{\vec{v}_{\bar{1}}}, \dots, \lin{\vec{v}_n}, \lin{\vec{v}_{\bar{n}}}]$ that are generated by symplectic bases of ${R}^{2n}$, modulo the relation $[\lin{\vec{v}_1}, \lin{\vec{v}_{\bar{1}}}, \dots,\lin{\vec{v}_i}, \lin{\vec{v}_{\bar{i}}}, \dots, \lin{\vec{v}_n}, \lin{\vec{v}_{\bar{n}}}] = -[\lin{\vec{v}_1}, \lin{\vec{v}_{\bar{1}}}, \dots, \lin{\vec{v}_{\bar{i}}}, \lin{\vec{v}_i}, \dots, \lin{\vec{v}_n}, \lin{\vec{v}_{\bar{n}}}]$ for all $1 \leq i \leq n$.
\end{definition}
%%---------------------------------------

The previous discussion then says that when ${R}$ is a Euclidean domain, then $\mc{V}_0$ surjects onto $\St^\w_{2n}({R}) \cong \St^\w_{2n}(K)$.
Note that when ${R}$ is a field, then $\mc{V}_0$ is the same as the generating set $V_0$ from Lemma \ref{lem:Stwpresent}.
%%---------------------------------------

%%---------------------------------------
\begin{lemma}
    \label{lem:coinvgenerators}
    Let ${R}$ be a PID, and let $\F$ be a field that is the quotient of ${R}$ by some prime $p \in {R}$. For a fixed $n \geq 1$, let $\mc{V}_0$ be the group defined in Definition \ref{def:integralgen} and let $V_0$ be the generating set for $\St^\w_{2n}(\F)$ defined in Lemma \ref{lem:Stwpresent}.
    Suppose all the units of $\F$ can be lifted to units of ${R}$. Then, we have an isomorphism
    \begin{align*}
        (\mc{V}_0)_{\Gamma^\w_{2n}(p)} \cong V_0
    \end{align*}
    induced by the natural mod $p$ reduction map.
\end{lemma}
%%----------------------------------------------------
\begin{proof}
    Surjectivity follows from the classical fact that the mod $p$ reduction map $\Sp_{2n}({R}) \to \Sp_{2n}(\F)$ is surjective: 
    %(See for example, \cite[Lemma 3.18]{miller2024uniform}) 
    if $[\lin{\vec{v}_1}, \lin{\vec{v}_{\bar{1}}}, \dots, \lin{\vec{v}_n}, \lin{\vec{v}_{\bar{n}}}]$ is a basis element of $V_0$, so that $\vec{v}_1, \vec{v}_{\bar{1}}, \dots, \vec{v}_n, \vec{v}_{\bar{n}}$ form a symplectic basis for $\F^{2n}$, then we can lift these to a symplectic basis of vectors in ${R}^{2n}$.
    For injectivity, suppose that 
    %%----------------------------------------------------
    \begin{align*}
        \vec{V}_1, \vec{V}_{\bar{1}}, \dots, \vec{V}_n, \vec{V}_{\bar{n}} \text{ and } \vec{W}_1, \vec{W}_{\bar{1}}, \dots, \vec{W}_n, \vec{W}_{\bar{n}}
    \end{align*}
    %%----------------------------------------------------
    are two symplectic bases of ${R}^{2n}$ such that the corresponding bases elements of $\mc{V}_0$ map to the same element 
    %%----------------------------------------------------
    \begin{align*}
        [\lin{\vec{v}_1}, \lin{\vec{v}_{\bar{1}}}, \dots, \lin{\vec{v}_n}, \lin{\vec{v}_{\bar{n}}}]
    \end{align*}
    %%----------------------------------------------------
    of $V_0$, where $\vec{v}_1, \vec{v}_{\bar{1}}, \dots, \vec{v}_n, \vec{v}_{\bar{n}}$ form a symplectic basis for $\F^{2n}$. We can assume that the basis of $\F^{2n}$ is such that for every $1 \leq i \leq n$, there exist units $u_i, u_{\bar{i}} \in \F$ so that $\vec{V}_i, \vec{V}_{\bar{i}}$ map to $\vec{v}_i, \vec{v}_{\bar{i}}$ respectively under the mod $p$ reduction map, and $\vec{W}_i, \vec{W}_{\bar{i}}$ map to $u_i\vec{v}_i, u_{\bar{i}}\vec{v}_{\bar{i}}$. 
    By assumption, we can lift these units to units ${U}_1, {U}_{\bar{1}}, \dots, {U}_n, {U}_{\bar{n}}$ of ${R}$. The above implies that the symplectic matrices in $\Sp_{2n}({R})$ whose columns are $\left({U}_1\vec{V}_1, {U}_{\bar{1}}\vec{V}_{\bar{1}}, \dots, {U}_n\vec{V}_n, {U}_{\bar{n}}\vec{V}_{\bar{n}}\right)$ and $\left(\vec{W}_1, \vec{W}_{\bar{1}}, \dots, \vec{W}_n, \vec{W}_{\bar{n}}\right)$, map to the same element of $\Sp_{2n}(\F)$. Since $\Gamma^\w_{2n}(p) = \ker(\Sp_{2n}R \to \Sp_{2n}\F)$, this means that there exists $A \in \Gamma^\w_{2n}(p)$ such that
    %%----------------------------------------------------
    \begin{align*}
        A \left(\vec{W}_1, \vec{W}_{\bar{1}}, \dots, \vec{W}_n, \vec{W}_{\bar{n}}\right) = \left({U}_1\vec{V}_1, {U}_{\bar{1}}\vec{V}_{\bar{1}}, \dots, {U}_n\vec{V}_n, {U}_{\bar{n}}\vec{V}_{\bar{n}}\right)
    \end{align*}
    %%----------------------------------------------------
In particular this implies that $[\lin{\vec{V}_1}, \lin{\vec{V}_{\bar{1}}}, \dots, \lin{\vec{V}_n}, \lin{\vec{V}_{\bar{n}}}]$ and $[\lin{\vec{W}_1}, \lin{\vec{W}_{\bar{1}}}, \dots, \lin{\vec{W}_n}, \lin{\vec{W}_{\bar{n}}}]$ are in the same coset of $(\mc{V}_0)_{\Gamma^\w_{2n}(p)}$. This proves injectivity.
\end{proof}
%%------------------------------------------

%%------------------------------------------
\begin{proof}[Proof of Theorem \ref{thm:congsubgp}]
    Since ${R}$ is a Euclidean domain, the map $\mc{V}_0 \to \St^\w_{2n}(K)$ sending an element
    \begin{align*}
        [\lin{\vec{v}_1}, \lin{\vec{v}_{\bar{1}}}, \dots, \lin{\vec{v}_n}, \lin{\vec{v}_{\bar{n}}}]
    \end{align*} 
    to the corresponding apartment class of $\St^\w_{2n}(K)$ is a surjection. Thus $(\mc{V}_0)_{\Gamma^\w_{2n}(p)}$ surjects onto $(\St^\w_{2n}(K))_{\Gamma_n(p)}$. Now, Lemma \ref{lem:coinvgenerators} says that $(\mc{V}_0)_{\Gamma^\w_{2n}(p)} \cong V_0$. Consider the map 
    %%----------------------------------------------------
    \begin{align*}
        (\St^\w_{2n}(K))_{\Gamma^\w_{2n}(p)} \to \St^\w_{2n}(\F)
    \end{align*}
    %%----------------------------------------------------
    induced by the mod $p$ reduction map $\bb{T}^\w_{2n}({R}) \cong \bb{T}^\w_{2n}(K) \to \bb{T}^\w_{2n}(\F)$. It is not hard to check that the composition
    %%----------------------------------------------------
    \begin{align*}
        (\mc{V}_0)_{\Gamma^\w_{2n}(p)} \to (\St^\w_{2n}(K))_{\Gamma^\w_{2n}(p)} \to \St^\w_{2n}(\F) \cong \coker(V_{1,2}\oplus V_{1,1} \to V_0)
    \end{align*}
    %%----------------------------------------------------
    is equal to the natural quotient map
    %%----------------------------------------------------
    \begin{align*}
        (\mc{V}_0)_{\Gamma^\w_{2n}(p)} \cong V_0 \to \coker(V_{1,2}\oplus V_{1,1} \to V_0)
    \end{align*}
    %%----------------------------------------------------
    From this we easily get surjectivity of the map $(\St^\w_{2n}(K))_{\Gamma^\w_{2n}(p)} \to \St^\w_{2n}(\F)$. This also tells us that to prove injectivity, it is enough to show that all elements in the image of the map $V_{1,2}\oplus V_{1,1} \to V_0$ lift to elements of $ (\mc{V}_0)_{\Gamma^\w_{2n}(p)}$ corresponding to relations in $\St^\w_{2n}(K)$.
    By Lemma \ref{lem:stdsharblies}, we can write any element $\sigma \in V_{1,2} \oplus V_{1,1}$ as a sum of the form
    %%----------------------------------------------------
    \begin{align*}
        \tau_{1,2} + \tau_{1,1} + \del_2\tau
    \end{align*}
    %%----------------------------------------------------
    where $\tau \in V_2$, $\tau_{1,1} \in V_{1,1}$, and $\tau_{1,2} \in V_{1,2}$ is a sum of standard Sharblies.
    Thus we have
    %%----------------------------------------------------
    \begin{align*}
        \del_1\sigma & = \del_1\tau_{1,2} + \del_1\tau_{1,1} + \del_1\del_2\tau \\ & = \del_1\tau_{1,2} + \del_1\tau_{1,1}
    \end{align*}
    %%----------------------------------------------------
    So it's enough to show that the image in $V_0$ of elements of $V_{1,1}$ and standard Sharblies in $V_{1,2}$ lift to relations of $\St^\w_{2n}(K)$.
    First let us look at standard Sharblies. The image of $\del_1$ on a standard Sharbly of the form
    %%----------------------------------------------------
    \begin{align*}
        [\lin{\vec{v}_1}, \lin{\vec{v}_{\bar{1}}}]\dots[\lin{\vec{v}_i}, \lin{\vec{v}_{\bar{i}}}, \lin{\vec{v}_{i+1}}, \lin{\vec{v}_{\overline{i+1}}}]\dots[\lin{\vec{v}_{n-1}}, \lin{\vec{v}_{\overline{n-1}}}]
    \end{align*}
    %%----------------------------------------------------
    is equal to (up to sign)
    %%----------------------------------------------------
    \begin{IEEEeqnarray}{l}
        [\lin{\vec{v}_1}, \lin{\vec{v}_{\bar{1}}}]\dots[\lin{\vec{v}_i},\lin{\vec{v}_{\bar{i}}}] [\lin{\vec{v}_{i+1}}, \lin{\vec{v}_{\overline{i+1}}}]\dots[\lin{\vec{v}_{n-1}}, \lin{\vec{v}_{\overline{n-1}}}] 
        \nonumber \\
        \hspace{1cm} \text{} -[\lin{\vec{v}_1}, \lin{\vec{v}_{\bar{1}}}]\dots[\lin{\vec{v}_{i+1}},\lin{\vec{v}_{\overline{i+1}}}] [\lin{\vec{v}_{i}}, \lin{\vec{v}_{\overline{i}}}]\dots[\lin{\vec{v}_{n-1}}, \lin{\vec{v}_{\overline{n-1}}}]
        \label{eqn:perm}
    \end{IEEEeqnarray}
    %%----------------------------------------------------
    The image of $\del_1$ on a standard Sharbly of the form
    %%----------------------------------------------------
    \begin{align*}
        [\lin{\vec{v}_1}, \lin{\vec{v}_{\bar{1}}}]\dots[\lin{\vec{v}_i}, \lin{\vec{v}_{\bar{i}}}, \lin{\vec{v}_{i+1}+a\vec{v}_i}, \lin{\vec{v}_{\overline{i+1}}}]\dots[\lin{\vec{v}_{n-1}}, \lin{\vec{v}_{\overline{n-1}}}]
    \end{align*}
    %%----------------------------------------------------
    is equal to (up to sign)
    %%----------------------------------------------------
    \begin{IEEEeqnarray}{l}
        \hspace*{-2em}[\lin{\vec{v}_1}, \lin{\vec{v}_{\bar{1}}}]\dots[\lin{\vec{v}_i}, \lin{\vec{v}_{\bar{i}}}][\lin{\vec{v}_{i+1}}, \lin{\vec{v}_{\overline{i+1}}}]\dots[\lin{\vec{v}_{n-1}}, \lin{\vec{v}_{\overline{n-1}}}] 
        \nonumber
        \\
        \text{} - [\lin{\vec{v}_1}, \lin{\vec{v}_{\bar{1}}}]\dots[\lin{\vec{v}_{i+1}+a\vec{v}_i}, \lin{\vec{v}_{\overline{i+1}}}][\lin{\vec{v}_i}, \lin{\vec{v}_{\bar{i}}-a\vec{v}_{\overline{i+1}}}] \dots[\lin{\vec{v}_{n-1}}, \lin{\vec{v}_{\overline{n-1}}}] 
        \nonumber
        \\
        \text{} - [\lin{\vec{v}_1}, \lin{\vec{v}_{\bar{1}}}]\dots[\lin{\vec{v}_{\bar{i}}}, \lin{\vec{v}_{{i+1}}+a\vec{v}_i}][\lin{\vec{v}_{i+1}},\lin{\vec{v}_{\bar{i}}-a\vec{v}_{\overline{i+1}}}] \dots[\lin{\vec{v}_{n-1}}, \lin{\vec{v}_{\overline{n-1}}}]
        \label{eqn:byk2}
    \end{IEEEeqnarray}
%%----------------------------------------------------
    Finally, suppose we have an element of $V_{1,1}$
    %%----------------------------------------------------
    \begin{align*}
        [\lin{\vec{v}_1}, \lin{\vec{v}_{\bar{1}}}]\dots[\lin{\vec{u}_i}, \lin{\vec{v}_i}, \lin{\vec{w}_i}]\dots[\lin{\vec{v}_n}, \lin{\vec{v}_{\bar{n}}}]
    \end{align*}
    %%----------------------------------------------------
    If this element is non-zero, then the vectors $\vec{v}_i, \vec{w}_i$ must be linearly independent and thus have nontrivial pairing. We then have $\vec{u}_i = a\vec{v}_i + b\vec{w}_i$ for some $a,b \in \F$. Thus the image of $\del_1$ on this element is
%%----------------------------------------------------
    \begin{IEEEeqnarray}{l}
        \left([\lin{\vec{v}_1}, \lin{\vec{v}_{\bar{1}}}]\dots[\lin{\vec{v}_i}, \lin{\vec{w}_i}]\dots[\lin{\vec{v}_n}, \lin{\vec{v}_{\bar{n}}}]\right) - \left([\lin{\vec{v}_1}, \lin{\vec{v}_{\bar{1}}}]\dots[\lin{a\vec{v}_i+b\vec{w}_i}, \lin{\vec{w}_i}]\dots[\lin{\vec{v}_n}, \lin{\vec{v}_{\bar{n}}}]\right)
        \nonumber
        \\ 
        \hspace{2cm} \text{} + \left([\lin{\vec{v}_1}, \lin{\vec{v}_{\bar{1}}}]\dots[\lin{a\vec{v}_i+b\vec{w}_i}, \lin{\vec{v}_i}, ]\dots[\lin{\vec{v}_n}, \lin{\vec{v}_{\bar{n}}}]\right)
        \label{eqn:byk1}
    \end{IEEEeqnarray}
%%----------------------------------------------------
    Now using the fact that any symplectic basis of $\F^{2n}$ lifts to a symplectic basis for ${R}^{2n}$, and that all units of $\F$ lift to units of ${R}$, we see that the three sums \eqref{eqn:perm}, \eqref{eqn:byk2}, and \eqref{eqn:byk1} we got above all lift to the relations \eqref{rel:perm}, \eqref{rel:byk2}, and \eqref{rel:byk1}, respectively, described in Lemma \ref{lem:intapt}. This completes the proof.
\end{proof}
%%------------------------------------------

\bibliographystyle{amsalpha}
\bibliography{Bibliography}
\quad \\ 

\end{document}